%% file: main.tex
\documentclass{article}
\usepackage[T2A]{fontenc}
\usepackage[utf8]{inputenc}
\usepackage[english]{babel}
\usepackage{csquotes}
\usepackage{latexsym}
\usepackage{amssymb}
\usepackage{amsmath}
\usepackage{amsthm}
\usepackage{environ}
\usepackage{tikz,subfig}
\usepackage{hyperref}
\usepackage{mathtools}
\usepackage{relsize}
\usepackage[nottoc,numbib]{tocbibind}
\usepackage{setspace}
\usepackage[small]{titlesec}
\usepackage{titlesec}
\usepackage{changepage}
\usepackage[hang,flushmargin]{footmisc}
\usepackage{relsize}

\usetikzlibrary{shapes.geometric}

\usepackage[font=small,labelsep=period]{caption}
\usepackage[
backend=bibtex,
style=ext-numeric,
articlein=false
]{biblatex}

\DeclareFieldFormat{pages}{#1}

\usepackage{geometry}
\theoremstyle{plain}
\newtheorem{theorem}{Theorem}
\newtheorem*{theorem*}{Theorem}
\newtheorem{lemma}{Lemma}
\newtheorem{corollary}{Corollary}

\theoremstyle{definition}
\newtheorem{definition}{Definition}

\theoremstyle{remark}
\newtheorem{remark}{Remark}

\let\svthefootnote\thefootnote
\newcommand\blankfootnote[1]{%
  \let\thefootnote\relax\footnotetext{#1}%
  \let\thefootnote\svthefootnote%
}
\let\svfootnote\footnote
\renewcommand\footnote[2][?]{%
  \if\relax#1\relax%
    \blankfootnote{#2}%
  \else%
    \if?#1\svfootnote{#2}\else\svfootnote[#1]{#2}\fi%
  \fi
}

\DeclareMathOperator{\ord}{\mathrm{ord}}
\DeclareMathOperator{\cl}{\mathrm{cl}}

\DeclarePairedDelimiter\intpart{\lfloor}{\rfloor}

\newcommand{\GG}{\hat{G}}
\newcommand{\bigast}{\mathbin{\text{\Large$*$}}}
\newcommand{\normalast}{\mathbin{\text{\normalsize$*$}}}
\newcommand{\dsumI}{\mathbin{\vcenter{\hbox{$\mathsmaller{\mathsmaller{\stackrel{\prec}{\prec}}}$}}}}
\newcommand{\dsumIIr}{\mathbin{\mathrm{\sqsupset}}}
\newcommand{\dsumIIl}{\mathbin{\mathrm{\sqsubset}}}
\newcommand{\dsumIII}{\mathbin{\mathrm{\asymp}}}

\titlelabel{\thetitle.\:}

\begin{document}

{\LARGE\bfseries\center
Powers with minimal commutator length \par
in free products of groups \par
}

\vspace{0.3cm}

{\scshape\large\center Vadim Bereznyuk \par}

\vspace{0.7cm}

\begin{adjustwidth}{1.6cm}{1.6cm}
\begin{spacing}{1}
{\footnotesize Given a free product of groups $G=\normalast_{j\in J} A_j$ and a natural number $n$, what is the minimal possible commutator length of an element $g^n \in G$ not conjugate to elements of the free factors? We give an exhaustive answer to this question.
\footnote[]{This work was supported by the Russian Science Foundation, project no. 22-11-00075.}}
\end{spacing}

\vspace{0.2cm}
\end{adjustwidth}

\section{Introduction}
It is well known that a proper power of a nonidentity element cannot be a commutator in a free group~\cite{Sch59}. Clearly, the square of a nonidentity element can be a product of two commutators and the cube of a nonidentity element can be a product of three commutators. Culler~\cite{Cull81} showed that in the free group~$F(a, b)$ a cube can be a product of two commutators: $[a,b]^3 = [a^{-1}ba, a^{-2}bab^{-1}][bab^{-1}, b^2]$, where $[a, b] \coloneqq a^{-1} b^{-1} a b$. Moreover, Culler showed that the element~$[a, b]^n$ can always be decomposed into a product of $\intpart*{n \mathbin{/} 2} + 1$ commutators (where $\intpart*{x}$ is the integer part of $x$). The minimal integer~$k$ such that an element~$g$ of a group~$G$ can be decomposed into a product of $k$ commutators is called the commutator length of~$g$ and denoted by $\cl(g)$. Thus $\cl([a, b]^n) \leq \intpart*{n \mathbin{/} 2} + 1$.

It turned out that this estimate is sharp for free groups: for any nonidentity element~$g$ of a free group $\cl(g^n) \geq \intpart*{n \mathbin{/} 2} + 1$. First it was proved by Comerford, Comerford and Edmunds~\cite{CCE91} for products of only $2$ commutators and then it was proved by Duncan and Howie~\cite{DH91} in the general case. Moreover, they proved a similar assertion for free products of locally indicable groups: if $g$ is an element of a free product of locally indicable groups such that $g$ is not conjugate to elements of the free factors, then $\cl(g^n) \geq \intpart*{n \mathbin{/} 2} + 1$. This assertion turned out to be true in a free product of arbitrary torsion-free groups, it was independently discovered by Ivanov and Klyachko~\cite{IK18} and Chen~\cite{Ch18}.

\begin{definition}
Let $G$ be a group with a fixed free-product decomposition: $G=\bigast_{j\in J} A_j$. We denote by $\GG$ the set of all elements of $G$ not conjugate to elements of the free factors, and define $k(G, n)$ as the minimal number $k$ such that an element $g^n \in \GG$ can be decomposed into a product of $k$ commutators.
\end{definition}

Thus it follows from~\cite{Cull81} in conjunction with~\cite{IK18} or~\cite{Ch18} that
$$k(G, n) = \intpart*{\frac{n}{2}} + 1$$
for free products of torsion-free groups. For free products of arbitrary groups Culler's estimate is not sharp any longer. For example, in the free product $\langle a \rangle_3 * \langle b \rangle$ a cube can be a commutator: $[a, b]^3 = [b^{-1}aba, ab^{-1}ab]$. We denote by $N(G)$ the minimal order of a nonidentity element of $G$. In~\cite{IK18} it was proved that the same estimate as for free products of torsion-free groups holds true for free products of arbitrary groups, but only if $n$ is relatively small: 
$$k(G, n) = \intpart*{\frac{n}{2}} + 1 \text{,}\quad \text{ if } n < N(G)\text{.}$$
Whereas in~\cite{Ch18} it was proved that 
$$k(G, n) \geq \intpart*{\frac{n - \intpart*{\frac{2n}{N(G)}}}{2}} + 1\text{.}$$
The author and Klyachko~\cite{BK22} generalized these estimates. It was shown that
$$k(G, n) \geq \intpart*{\frac{n}{2}} - \intpart*{\frac{n}{N(G)}} + 1\text{.}$$
In this work we prove that this estimate is sharp.

\begin{theorem} \label{main_theorem}
Let $G=\bigast_{j\in J} A_j$ be a free product of nontrivial groups and $n$ be a positive integer. Then
$$k(G,n) = \intpart*{\frac{n}{2}} - \intpart*{\frac{n}{N(G)}} + 1\text{.}$$ 
\end{theorem}

Actually, we prove a more general result.

\begin{definition}
Let $G$ be a group with a fixed free-product decomposition: $G=\bigast_{j\in J} A_j$. For an element $g \in \GG$ with a cyclically reduced form $a_{j_1, 1} \ldots a_{j_m, m}$ (where $a_{j_i, i} \in A_{j_i}$), we denote by $N(g)$ the minimal order of its letters $a_{j_1, 1}$, ..., $a_{j_m, m}$. For $N \in \{N(g) \mid g \in \GG\}$ we define $k(G, n, N)$ as the minimal number $k$ such that an element $g^n \in \GG$ with $N(g)=N$ can be decomposed into a product of $k$ commutators.  
\end{definition}

\begin{theorem} \label{general_theorem}
Let $G=\bigast_{j\in J} A_j$ be a free product of nontrivial groups and $n$ be a positive integer. If $N \in \{N(g) \mid g \in \GG\}$, then 
$$k(G,n,N) = \intpart*{\frac{n}{2}} - \intpart*{\frac{n}{N}} + 1\text{.}$$
\end{theorem}

Similarly to Culler's examples, the minimal commutator length is achieved on powers of commutators. 

\begin{theorem} \label{a_t_theorem}
Let $G=\bigast_{j\in J} A_j$ be a free product of nontrivial groups and $n$ be a positive integer. If $a \in A_{j_1}$ and $t \in A_{j_2}$ are two nonidentity elements lying in different free factors such that $\ord(a) \leq \ord(t)$, then 
$$\cl([a, t]^n) = \intpart*{\frac{n}{2}} - \intpart*{\frac{n}{\ord(a)}}+ 1\text{.}$$
\end{theorem}

\begin{remark}
If $N(G)$, $N$ or $\ord(a)$ is infinite, we naturally consider $\intpart*{n \mathbin{/} N(G)}$, $\intpart*{n \mathbin{/} N}$ or $\intpart*{n \mathbin{/} \ord(a)}$ to be zero. In that case $k(G, n)$, $k(G, n, N)$ or $\cl([a, t]^n)$ is equal to $\intpart*{n \mathbin{/} 2} + 1$ which corresponds to the known results for free products of torsion-free groups.
\end{remark}

\begin{corollary}
If $a$ and $b$ are two nonidentity elements of a group $G$ and $n$ is a positive integer, then
$$\cl([a, b]^n) \leq \intpart*{\frac{n}{2}} - \intpart*{\frac{n}{\ord(a)}} + 1\text{.}$$
\end{corollary}

For example, if $a^3=1$, then all powers of $[a, b]$ up to $7$ are equal to a product of $2$ commutators. Moreover, $[a, b]^9$ is also a product of $2$ commutators and $[a, b]^3$ is a commutator itself. If $a^4=1$, then $[a, b]^4$ is a product of $2$ commutators and $[a, b]^8$ is a product of $3$ commutators.

A geometric language is used to prove these theorems: for each $a$, $t$ and $n$ we construct a Howie diagram $D$ on a closed oriented surface of genus $\intpart*{n \mathbin{/} 2} - \intpart*{n \mathbin{/} \ord(a)} + 1$, such that $D$ has only one face, the label of this face is $[a, t]^n$ and all vertices of $D$ are interior.

We start with the definition of Howie diagrams and their relation to products of commutators in Section~\ref{howie_diagram_section}. Diagrams for $[a, t]^n$ with the minimal possible genus are constructed in Section~\ref{minimal_diagrams_section}. These diagrams are used to prove the theorems in Section~\ref{proof_section}.

\section{Howie diagrams and products of commutators} \label{howie_diagram_section}

Diagrams similar to those we will now define were introduced by Howie in~\cite{How83} and were considered in~\cite{Kl93}, \cite{Le09}, \cite{IK18}, \cite{BK22}, and many other works. Here we use the definitions from~\cite{BK22}. Namely, suppose that $S$ is a closed oriented surface, and $\Gamma$ is a finite undirected graph which is embedded into $S$ and divides it into simply connected domains. Such a graph determines a cell decomposition of $S$, i.e., a mapping $\mathrm{M}$ called a map on $S$:
$$
\mathrm{M} : \bigsqcup_{i=1}^m D_i \to S \text{,}
$$
where $D_i$ are two-dimensional disks. The mapping $\mathrm{M}$ is continuous, surjective and injective on the interior (i.e., on $\bigsqcup\limits_{i=1}^m(D_i\setminus \partial D_i)$), the preimage of each point is finite, and the preimage of the graph $\Gamma$ is the union of the boundaries of the faces: ${\mathrm{M}}^{-1}(\Gamma)=\bigsqcup_{i=1}^m \partial D_i$. The preimages of the vertices of $\Gamma$ are called corners of the map. We say that a corner $c$ is at a vertex~$v$ if ${\mathrm{M}}(c)=v$. The vertices and edges of $\Gamma$ are referred to as vertices and edges of the map $\mathrm{M}$. The disks $D_i$ are called faces or cells of the map.

Such a map is called a diagram over a free product $A*B$ if:

\begin{enumerate}
    \item The graph $\Gamma$ is bipartite. There are two types of vertices: $A$-vertices and $B$-vertices, and each edge joins an~$A$-vertex with a $B$-vertex.
    \item The corners at $A$-vertices are labeled by elements of the group $A$ and the corners at $B$-vertices are labeled by elements of $B$.
    \item Some vertices are distinguished and called exterior. All the other vertices are called interior.
    \item The label of each interior $A$-vertex equals $1$ in the group $A$ and the label of each interior $B$-vertex equals $1$ in the group $B$, where the label of a vertex is the product of labels of corners at this vertex taken clockwise (thus the label of a vertex is defined up to conjugation in $A$ or $B$).
\end{enumerate}

\begin{remark}
    Note that a single vertex on a sphere is not a correct diagram since we require each point of the sphere to have a finite preimage, but the preimage of this vertex is $\partial D_1$.
\end{remark}

The label of a face of a diagram is the product of labels of all corners of this face taken counterclockwise. It is an element of the free product $A*B$ defined up to conjugation.

Let us look at an example of a diagram over the free product $\langle a \rangle_3 * \langle b \rangle_3$ shown in Fig.~\ref{fig:diagram_example}. It is placed on a torus represented as a rectangle with opposite sides identified. The diagram has two interior vertices, three edges and a face whose label is $(ab)^3$. This is a geometric interpretation of the fact that the element $(ab)^3$ is a commutator in the free product $\langle a \rangle_3 * \langle b \rangle_3$. It follows from the next lemma.

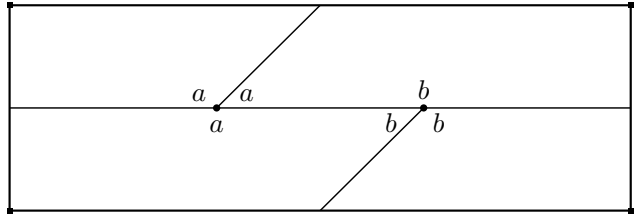
\begin{figure}[!h]
    \centering
    \bigskip
    \resizebox{0.6\textwidth}{!}{\input{images/diagram_example.tikz}}
    \caption{A diagram on a torus.}
    \label{fig:diagram_example}
\end{figure}

\begin{lemma}\label{diagram_lemma}
Let $u$ be a cyclically reduced element of a free product $A * B$ not conjugate to elements of the free factors. If there is a diagram $D$ over $A * B$ on a closed oriented surface of genus $k$ such that $D$ has only one face, the label of this face is $u$ and all the vertices of $D$ are interior, then $u$ is a product of $k$ commutators.
\end{lemma}

To prove this lemma we need to define the label of a path. First, we construct an auxiliary graph~$\Gamma'$ by inserting an additional vertex of degree $2$ in the middle of each edge of the graph~$\Gamma$. Let us call these vertices auxiliary vertices and let us call a path in $\Gamma'$ whose endpoints are auxiliary vertices an auxiliary path. Labels are defined only for auxiliary paths. Let $p$ be such a path. We represent it as a composition of paths $p_1 \ldots p_m$ such that each $p_i$ is an auxiliary path traversing only one vertex of $\Gamma$. It means that each $p_i$ consists of $2$ oriented edges $(e_i^1, e_i^2)$ of $\Gamma'$ such that $e_i^1$ starts at some auxiliary vertex and ends at some vertex $v_i$ of $\Gamma$, while $e_i^2$ starts at $v_i$ and ends at some auxiliary vertex. The label $l(p)$ of the path $p$ is defined as the product $l(p_1) \ldots l(p_m)$, where the label $l(p_i)$ is the product of labels of corners at the vertex~$v_i$ taken clockwise, starting from the corner adjacent to the left side of the oriented edge $e_i^1$ and ending with the corner adjacent to the left side of the oriented edge $e_i^2$.

See Fig.~\ref{fig:path_label} for examples. The path $p_1$ traverses only one vertex of $\Gamma$ and the path $p_2$ traverses two vertices of $\Gamma$. Their labels are $l(p_1) = a_2 a_3 a_4$ and $l(p_2) = a_1 a_2 a_3 b_1$. The labels of their inverses are $l(p_1^{-1}) = a_5 a_1$ and $l(p_2^{-1}) = b_2 b_3 a_4 a_5$. Let $p_3$ be a path $e_3 e_3^{-1}$ and $p_4$ be a path $e_4 e_4^{-1}$. These paths make a U-turn at vertices of $\Gamma$ and their labels are $l(p_3) = b_3 b_1 b_2$ and $l(p_4) = b_4$.

\begin{figure}[!h]
    \centering
    \bigskip
    \resizebox{0.5\textwidth}{!}{\input{images/path_label.tikz}}
    \medskip
    \caption{Labels of paths.}
    \label{fig:path_label}
\end{figure}
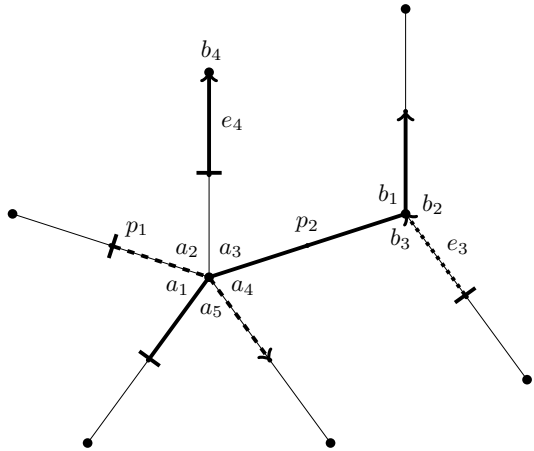

We will use the following property of path labels: if all vertices of a diagram are interior and an auxiliary path $p$ can be transformed into a trivial path by consecutive removals of subpaths of the form $e e^{-1}$, then $l(p)=1$. Indeed, we can assume that $e$ is an edge of the auxiliary graph $\Gamma'$ and then it is sufficient to consider the following two cases: 
\begin{enumerate}
    \item The edge $e$ ends at a vertex of $\Gamma$. Then the path $p$ can be represented as $p_1 e e^{-1} p_2$, where $p_1$ and $p_2$ are some auxiliary paths (possibly trivial). In that case $l(e e^{-1})$ is equal to the label of the terminal vertex of $e$. Since all the vertices of $\Gamma$ are interior, this label is equal to $1$ and thus $l(p_1 e e^{-1} p_2) = l(p_1 p_2)$.
    \item The edge $e$ starts at a vertex of $\Gamma$. Then the path $p$ can be represented as $p_1 e_1 e e^{-1} e_2 p_2$, where $p_1$ and $p_2$ are some auxiliary paths (possibly trivial), and $e_1$ and $e_2$ are some edges of $\Gamma'$. To prove that $l(p_1 e_1 e e^{-1} e_2 p_2) = l(p_1 e_1 e_2 p_2)$ it is sufficient to show that $l(e_1 e e^{-1} e_2) = l(e_1 e_2)$. Let us consider the cases:
    \begin{itemize}
        \item[-] if $e_1 \neq e^{-1}$ and $e_2 \neq e$, then depending on the relative position of $e_1, e_2$ and $e$ we have either $l(e_1 e e^{-1} e_2) = l(e_1 e_2)$ or $l(e_1 e e^{-1} e_2) = l(e_1 e_1^{-1}) l(e_1 e_2)$; in both cases it is equal to $l(e_1 e_2)$ since $l(e_1 e_1^{-1})$ is equal to the label of the terminal vertex of $e_1$ and all the vertices of $\Gamma$ are interior;
        \item[-] if $e_1 = e^{-1}$, then $l(e_1 e e^{-1} e_2) = l(e^{-1} e) l(e_1 e_2) = l(e_1 e_2)$ since $l(e^{-1} e)$ is equal to the label of the initial vertex of $e$ and all the vertices of $\Gamma$ are interior;
        \item[-] if $e_2 = e$, the argument is the same as for the previous case.
    \end{itemize}
\end{enumerate}

Let $D$ be a diagram on a surface $S$ defined by a graph $\Gamma$. In the following proof by a path we mean either a path in the auxiliary graph $\Gamma'$, consisting of edges, or a path on the surface $S$ as a continuous map from the unit interval to $S$. If $p$ is a path in $\Gamma'$, we also denote by $p$ the corresponding path on $S$ obtained by natural identification of each edge of the path with the unit interval.

\begin{proof}[Proof of Lemma~\ref{diagram_lemma}]
Let us represent the surface $S$ of the diagram $D$ as a standard $4k$-gon $P_{4k}$ with identified boundary edges. Choose an auxiliary vertex $Q$ of the auxiliary graph $\Gamma'$ such that the label of the face read starting from this vertex is $u$.
Denote the closed boundary path of the face starting at $Q$ as $p$. Note that $l(p)=u$ and $p$ is homotopic to the boundary path of $P_{4k}$ conjugated by a simple path $q$ connecting the vertex $Q$ to some vertex of $P_{4k}$. This conjugated boundary path is equal to $[a_1, b_1] \ldots [a_k, b_k]$, where $a_i$ and $b_i$ are the closed paths corresponding to the edges of $P_{4k}$ conjugated by $q$ (all the vertices of $P_{4k}$ are the same point on $S$, so all the edges of $P_{4k}$ are closed loops on $S$; we take these loops and conjugate them by $q$). 

Choose a point $C$ on the surface such that $C$ does not belong to $\Gamma$, does not belong to the boundary of $P_{4k}$ and does not belong to the path $q$. An arbitrary closed path $r$ on the surface $S \setminus C$ is homotopic to some closed path in $\Gamma'$. Indeed, consider the preimage $\mathrm{M}^{-1}(r)$ of this path on the disk $D_1$ (where $M$ and $D_1$ are taken from the definition of $D$). Centrally project $\mathrm{M}^{-1}(r)$ on the boundary $\partial D_1$ through the point $\mathrm{M}^{-1}(C)$ and take the image of this projection. Denote this new closed path as $r'$. Clearly, it lies in $\Gamma'$ and it is homotopic to $r$ on $S \setminus C$ by its construction. Finally, we turn $r'$ into a path in $\Gamma'$ by additional homotopy.

Hence the paths $a_i$ and $b_i$ are homotopic to some paths $a'_i$ and $b'_i$ in the graph $\Gamma'$ and the path $p$ is homotopic to $p' = [a'_1, b'_1] \ldots [a'_k, b'_k]$. We can assume that this homotopy lies in $\Gamma'$ because it can be projected through $\mathrm{M}^{-1}(C)$. Thus the path $p^{-1} p'$ is homotopic in the graph to the trivial path. It means that $p^{-1} p'$ can be transformed into the trivial path by consecutive removals of subpaths of the form $q q^{-1}$. Since all the vertices of the diagram are interior, $1 = l(p^{-1} p') = l(p^{-1}) l(p') = l(p)^{-1} l(p')$. Thus
$$u = l(p) = l(p') = [l(a'_1), l(b'_1)] \ldots [l(a'_k), l(b'_k)]\text{.}$$
\end{proof}

\section[]{Diagrams for $\boldsymbol{[a, t]^n}$} \label{minimal_diagrams_section}

In this section we prove the following lemma.
\begin{lemma}\label{D_n_N_lemma}
Let $a \in A$ and $t \in T$ be two elements of two groups, and let $n$ and $N$ be two natural numbers such that $n \geq N \geq 3$. If $N$ is even or $N$ and $n$ are odd, then there is a diagram $D_{n, N}$ over the free product $A * T$ on a closed oriented surface of genus $\intpart*{n \mathbin{/} 2} - \intpart*{n \mathbin{/} N} + 1$ such that $D_{n, N}$ has only one face and the label of this face is $[a, t]^n$. All the vertices of $D_{n, N}$ are interior if $a^N=1$.
\end{lemma}

\begin{proof}

We explicitly construct a desired diagram on a closed surface represented by a rectangle whose top and bottom sides subdivided into $2k$ edges. Each edge from the top side has a corresponding edge on the bottom side. If we denote the top edges as $e_1, \ldots, e_{2k}$, then the bottom edges are $e_2, e_1, \ldots, e_{2k}, e_{2k-1}$. The surface is obtained by identification of the top edges with the corresponding bottom edges. The left and right sides of the rectangle are contracted to a single point. See Fig.~\ref{fig:contour} for an example. Such a rectangle forms a closed oriented surface of genus $k$ if there are $2k$ edges on the top. We also admit a degenerate rectangle without top and bottom edges which represents a sphere.

\begin{figure}[!h]
    \centering
    \bigskip
    \resizebox{0.6\textwidth}{!}{\input{images/contour.tikz}}
    \caption{A closed oriented surface of genus $k$.}
    \label{fig:contour}
\end{figure}
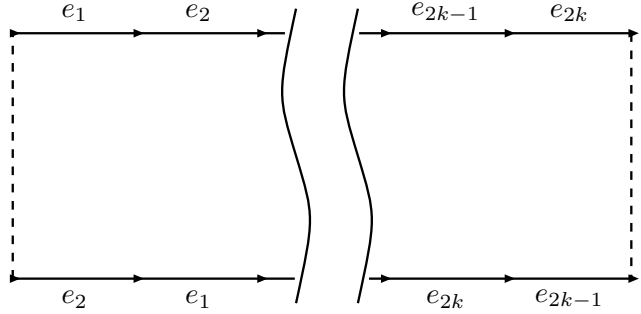

The diagram is defined by its graph $\Gamma$ drawn on this rectangle. All corners at $A$-vertices will have label $a$ or $a^{-1}$. All corners at $T$-vertices will have label $t$ or $t^{-1}$. Some $A$-vertices will be labeled by $A^+$ or $A^-$. All corners at such vertices are labeled by $a$ or $a^{-1}$ respectively.

Let us represent $n$ as $n=rN+q$, where $0 \leq q < N$. The diagram is constructed depending on the parities of $N$, $r$ and $q$. We consider the following cases (where $s = \intpart*{q \mathbin{/} 2}$):
\begin{enumerate}
    \item $n=rN+2s$, where $N$ and $r$ are odd.
    \item $n=rN+2s+1$, where $N$ is odd and $r$ is even.
    \item $n=rN+2s$, where $N$ is even.
    \item $n=rN+2s+1$, where $N$ is even.
\end{enumerate}

\vspace{0.2cm}
\noindent
\textbf{Case 1. $\boldsymbol{n=rN+2s}$, where $\boldsymbol{N}$ and $\boldsymbol{r}$ are odd}
\vspace{0.05cm}

\noindent
First, we construct a diagram $D_{N, N}$ for each odd $N$. The diagrams $D_{1, 1}$ and $D_{3, 3}$ are shown in Fig.~\ref{fig:D_1_1_and_D_3_3_and_D_5_5}. Clearly, the diagram $D_{1, 1}$ has genus $0$ and only one face. The label of this face is $[a, t]$. All the vertices of $D_{1, 1}$ are interior if $a=1$. The diagram $D_{3, 3}$ has genus $1$ and only one face. The label of this face is $[a, t]^3$. All the vertices of $D_{3, 3}$ are interior if $a^3=1$.

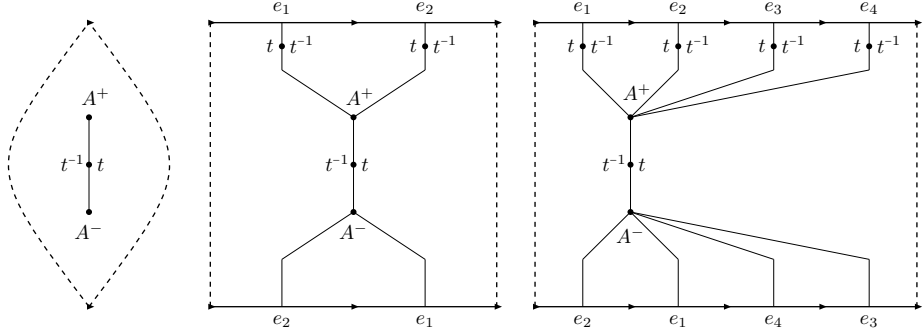
\begin{figure}[!h]
    \centering
    \bigskip
    \resizebox{0.9\textwidth}{!}{
        \input{images/D_1_1.tikz}
        \input{images/D_3_3.tikz}
        \hspace{0.3cm}
        \input{images/D_5_5.tikz}
    }
    \caption{$D_{1,1}$, $D_{3,3}$ and $D_{5,5}=D_{3, 3} \dsumI D_{3, 3}$.}
    \label{fig:D_1_1_and_D_3_3_and_D_5_5}
\end{figure}

To get the diagrams for other $N$ we define a composition of two diagrams.

\begin{definition} \label{def:borders}
Let $D$ be a diagram drawn on a rectangle with $2k$ edges $e_i$, where $k > 0$. We call the left border of $D$ a path going along the graph $\Gamma$ from the leftmost intersection point of $\Gamma$ with the bottom side of the rectangle to the leftmost intersection point of $\Gamma$ with the top side of the rectangle. We call the right border of $D$ a path going along the graph $\Gamma$ from the rightmost intersection point of $\Gamma$ with the top side of the rectangle to the rightmost intersection point of $\Gamma$ with the bottom side of the rectangle.
\end{definition}

Note that the left and right borders are oriented. The left border is traversed from bottom to top and the right border is traversed from top to bottom.

\begin{definition}
Let $D$ be a diagram drawn on a rectangle. Let $(v^{+}, v^{-})$ be a pair of an $A^{+}$-vertex and an $A^{-}$-vertex, both lying on the left border of $D$ or both lying on the right border of $D$. We call such a pair positively oriented if moving along the respective border of $D$ we first meet $v^{+}$. Otherwise we call such a pair negatively oriented. 
\end{definition}

\begin{definition}
Let $D_1$ and $D_2$ be two diagrams drawn on rectangles $R_1$ and $R_2$. We denote by $D_1 + D_2$ a new diagram obtained as follows: remove the right side of $R_1$, remove the left side of $R_2$ and attach the right side of $R_1$ to the left side of $R_2$.
\end{definition}

\begin{remark}
Note that $D_1 + D_2$ is not actually a correct diagram since it has non-simply connected domains. Nevertheless, we call it a diagram for simplicity of notation.
\end{remark}

\begin{definition}
Let $D_1$ and $D_2$ be two diagrams drawn on rectangles. Suppose that there is an $A^{+}$-vertex $v^{+}_1$ and an $A^{-}$-vertex $v^{-}_1$ both lying on the right border of $D_1$ and there is an $A^{+}$-vertex $v^{+}_2$ and an $A^{-}$-vertex $v^{-}_2$ both lying on the left border of $D_2$, such that $v^{+}_2$ and $v^{-}_2$ are connected by a $t$-arc. Suppose also that the pair $(v^{+}_1, v^{-}_1)$ and the pair $(v^{+}_2, v^{-}_2)$ have different orientations. We denote by $D_1 \dsumI D_2$ a diagram obtained from $D_1 + D_2$ as follows: remove the $t$-arc connecting $v^{+}_2$ and $v^{-}_2$, glue $v^{+}_2$ to $v^{+}_1$ and glue $v^{-}_2$ to $v^{-}_1$.
\end{definition}

For odd $N \geq 5$ we define 
$$
D_{N, N} = \underbrace{D_{3, 3} \dsumI \ldots \dsumI D_{3, 3}}_{(N-1)\mathbin{/}2 \text{ summands}}\text{.}
$$
See Fig.~\ref{fig:D_1_1_and_D_3_3_and_D_5_5} for an example. Clearly, if $D_1$ and $D_2$ are two diagrams with one face, then $D_1 \dsumI D_2$ is also a diagram with one face. If the labels of their faces are $[a, t]^{n_1}$ and $[a, t]^{n_2}$ respectively, then the label of the face of $D_1 \dsumI D_2$ is $[a, t]^{n_1+n_2-1}$. If $D_1$ has genus $k_1$ and $D_2$ has genus $k_2$, then $D_1 \dsumI D_2$ has genus $k_1 + k_2$.

Since $D_{3, 3}$ has only one face, it follows from the above representation that $D_{N, N}$ has only one face and its label is $[a, t]^N$. The genus is equal to $(N - 1) \mathbin{/} 2$. There is one $A^+$-vertex of degree $N$ and one $A^-$-vertex of degree $N$. Thus all the vertices of $D_{N, N}$ are interior if $a^N=1$.

\begin{definition}
Let $D$ be a diagram drawn on a rectangle. We denote by $D^{-}$ the same diagram where labels of all corners are inverted (in particular, $A^+$-vertices turn into $A^-$-vertices and vice versa).
\end{definition}

\begin{definition}\label{def:borders_D_1_1}
We define the left border of $D_{1, 1}^-$ as a path going from the $A^{+}$-vertex to the $A^{-}$-vertex and the right border of $D_{1, 1}^-$ as a path going from the $A^{-}$-vertex to the $A^{+}$-vertex.
\end{definition}

Note that we need a separate definition for the left and right border of $D_{1, 1}^-$ since Definition~\ref{def:borders} works only for diagrams with positive genus.

\begin{definition}
Let $D_1$ and $D_2$ be two diagrams drawn on rectangles. Suppose that there is a $t$-arc $f_1$ in $D_1$ connecting an $A^{+}$-vertex $v^{+}_1$ with an $A^{-}$-vertex $v^{-}_1$ and there is a $t$-arc $f_2$ in $D_2$ connecting an $A^{+}$-vertex $v^{+}_2$ with an $A^{-}$-vertex $v^{-}_2$. Suppose also that the interior of $f_1$ has a nonempty intersection with the right border of $D_1$ and the interior of $f_2$ has a nonempty intersection with the left border of $D_2$. Orient $f_1$ and $f_2$ such that they start at $A^{+}$-vertices. Suppose that these orientations are both the same as (or both different from) the orientation of the respective border. Take the diagram $D_1 + D_2$. Add a new $t$-arc connecting $v_1^{+}$ with $v_2^{-}$ and add a new $t$-arc connecting  $v_1^{-}$ with $v_2^{+}$. The diagram obtained after the removal of the $t$-arc connecting $v_1^{+}$ with $v_1^{-}$ is denoted by $D_1 \dsumIIr D_2$. The diagram obtained after the removal of the $t$-arc connecting $v_2^{+}$ with $v_2^{-}$ is denoted by $D_1 \dsumIIl D_2$.
\end{definition}

See Fig.~\ref{fig:sumII_example} for an example. Clearly, if $D_1$ and $D_2$ are two diagrams with one face, then $D_1 \dsumIIr D_2$ ($D_1 \dsumIIl D_2$) is also a diagram with one face. If the labels of their faces are $[a, t]^{n_1}$ and $[a, t]^{n_2}$, then the label of the face of $D_1 \dsumIIr D_2$ ($D_1 \dsumIIl D_2$) is $[a, t]^{n_1+n_2+1}$. If $D_1$ has genus $k_1$ and $D_2$ has genus $k_2$, then $D_1 \dsumIIr D_2$ ($D_1 \dsumIIl D_2$) has genus $k_1 + k_2$.

\begin{figure}[!h]
    \centering
    \bigskip
    \resizebox{0.9\textwidth}{!}{\input{images/sumII_example.tikz}}
    \caption{$D_{3, 3}^-$, $D_{3, 3}$ and $D_{3, 3}^- \dsumIIr D_{3,3}$.}
    \label{fig:sumII_example}
\end{figure}
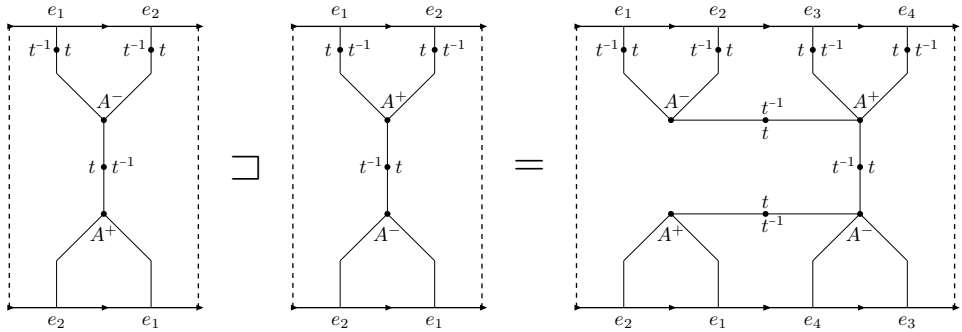

Now recall that $r$ is odd and $N\geq3$ is odd. For $r \geq 3$ we define
$$D_{rN, N} = D_{(r-2)N, N} \dsumIIr D_{N-2, N-2}^- \dsumIIl D_{N, N}\text{.}$$
See Fig.~\ref{fig:D_15_5_and_D_25_5} for examples. Clearly, all the $A^+$- and $A^-$-vertices of $D_{rN, N}$ have degree $N$. Due to the properties of $\dsumIIr$ and $\dsumIIl$ we get by induction that $D_{rN, N}$ has only one face and its label is $[a,t]^{((r-2)N + (N-2) + 1) + N + 1} = [a, t]^{rN}$. The genus is equal to $\intpart*{rN \mathbin{/} 2} - \intpart*{rN \mathbin{/} N} + 1$ since
$$
\intpart*{\frac{(r-2)N}{2}} - \intpart*{\frac{(r-2)N}{N}} + 1 + \frac{N - 3}{2} + \frac{N - 1}{2} = \intpart*{\frac{rN}{2}} - \intpart*{\frac{rN}{N}} + 1\text{.}
$$

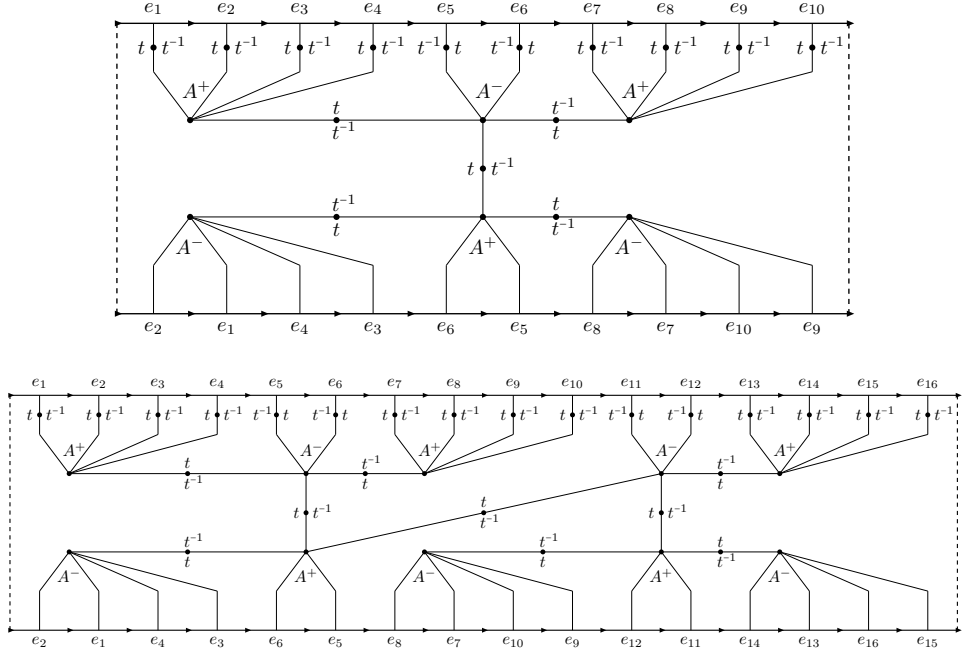
\begin{figure}[!h]
    \centering
    \bigskip
    \subfloat{\resizebox{0.7\textwidth}{!}{\input{images/D_15_5.tikz}}}
    
    \subfloat{\resizebox{0.9\textwidth}{!}{\input{images/D_25_5.tikz}}}
    \caption{$D_{15,5}$ and $D_{25,5}$.}
    \label{fig:D_15_5_and_D_25_5}
\end{figure}

\begin{remark}
For $N=3$ there arises the degenerate diagram $D_{1, 1}^-$. In that case the diagram $D_{(r-2)3, 3} \dsumIIr D_{1, 1}^- \dsumIIl D_{3, 3}$ is correctly defined due to Definition~\ref{def:borders_D_1_1}. See Fig.~\ref{fig:D_15_3} for an example.
\end{remark}

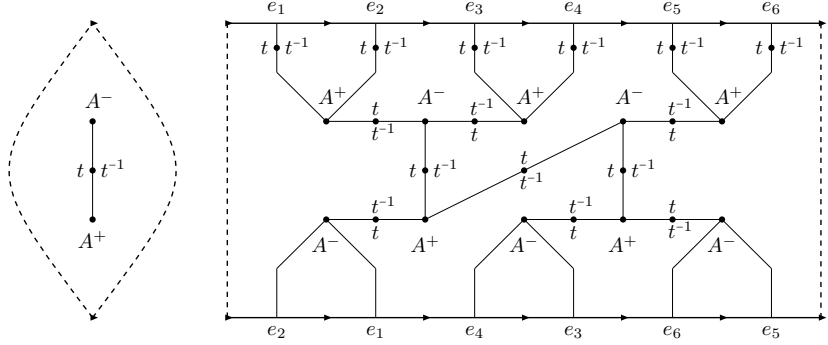
\begin{figure}[!h]
    \centering
    \bigskip
    \resizebox{0.8\textwidth}{!}{%
        \input{images/D_1_1_minus.tikz}%
        \hspace{0.3cm}%
        \input{images/D_15_3.tikz}%
    }
    \caption{$D_{1,1}^{-}$ and $D_{15,3}$.}
    \label{fig:D_15_3}
\end{figure}

Finally, let us construct a diagram for $n=rN+2s$. This is done by composing the diagram $D_{rN, N}$ with $s$ auxiliary diagrams $D_{+2}$ shown in Fig.~\ref{fig:D_+2_and_D_+2_2}. 

\begin{definition}
Let $D_1$ and $D_2$ be two diagrams drawn on rectangles. Suppose that there is an edge $f_1$ in $D_1$ whose adjacent corners have labels $t$, $t^{-1}$, $a^{\varepsilon}$ and $a^{\varepsilon}$, and there is an edge $f_2$ in $D_2$ whose adjacent corners have labels $t$, $t^{-1}$, $a^{\varepsilon}$ and $a^{\varepsilon}$ (where $\varepsilon \in \{+1, -1\}$). Suppose also that the interior of $f_1$ has a nonempty intersection with the right border of $D_1$ and the interior of $f_2$ has a nonempty intersection with the left border of $D_2$. Orient $f_1$ and $f_2$ such that they start at $T$-vertices. Suppose that these orientations are both the same as (or both different from) the orientation of the respective border. We denote by $D_1 \dsumIII D_2$ a diagram obtained from $D_1 + D_2$ as follows: add a new edge connecting the $A$-vertex adjacent to $f_1$ with the $T$-vertex adjacent to $f_2$, add a new edge connecting the $T$-vertex adjacent to $f_1$ with the $A$-vertex adjacent to $f_2$, and remove the edges $f_1$ and $f_2$.
\end{definition}

See Fig.~\ref{fig:D_+2_and_D_+2_2} for examples. It is not hard to see that this operation has the following properties: If $D_1$ ($D_2$) has only one face and $D_2$ ($D_1$) has two faces such that the edge $f_2$ ($f_1$) is adjacent to both these faces, then $D_1 \dsumIII D_2$ has only one face. If the labels of the faces of $D_1$ and $D_2$ are $[a, t]^{n_1}$, $[a, t]^{n_2}$ and $[a, t]^{n_3}$, then the label of the face of $D_1 \dsumIII D_2$ is $[a, t]^{n_1 + n_2 + n_3}$.
If $D_1$ has genus $k_1$ and $D_2$ has genus $k_2$, then $D_1 \dsumIII D_2$ has genus $k_1 + k_2$. If all the vertices of $D_1$ and $D_2$ are interior, then all the vertices of $D_1 \dsumIII D_2$ are also interior.

We define
$$D_{rN+2s, N} = D_{rN, N} \dsumIII \underbrace{D_{+2} \dsumIII \ldots \dsumIII D_{+2}}_{s \text{ summands}}\text{.}$$
See Fig.~\ref{fig:D_+2_and_D_+2_2} for an example. Clearly, $D_{+2}$ has two faces and their labels are equal to $[a, t]$. All the vertices of $D_{+2}$ are interior and the whole left border of $D_{+2}$ is adjacent to both faces of $D_{+2}$. The genus of $D_{+2}$ is $1$. Due to the properties of $\dsumIII$ we get that $D_{rN+2s, N}$ has only one face and the label of this face is $[a,t]^{rN + 2s}$. The genus of $D_{rN+2s, N}$ is
$$
\intpart*{\frac{rN}{2}} - \intpart*{\frac{rN}{N}} + 1 + s = \intpart*{\frac{rN + 2s}{2}} - \intpart*{\frac{rN + 2s}{N}} + 1
$$
if $2s < N$. All the vertices of $D_{rN+2s, N}$ are interior if $a^N=1$.

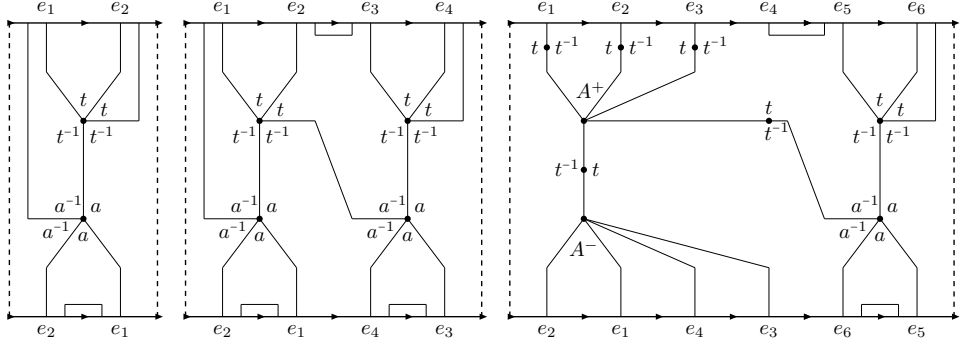
\begin{figure}[!h]
    \centering
    \bigskip
    \resizebox{0.9\textwidth}{!}{%
        \input{images/D_+2.tikz}%
        \hspace{0.3cm}%
        \input{images/D_+4.tikz}%
        \hspace{0.3cm}%
        \input{images/D_7_5.tikz}%
    }
    \caption{$D_{+2}$, $D_{+2} \dsumIII D_{+2}$ and $D_{7, 5} = D_{5, 5} \dsumIII D_{+2}$.}
    \label{fig:D_+2_and_D_+2_2}
\end{figure}

\vspace{0.35cm}
\noindent
\textbf{Case 2. $\boldsymbol{n=rN+2s+1}$, where $\boldsymbol{N}$ is odd and $\boldsymbol{r}$ is even}
\vspace{0.05cm}

\noindent
A diagram for this case is obtained by composing an auxiliary diagram $D_{+N+1, N}$ with the diagram $D_{(r-1)N+2s,N}$. The diagram $D_{+N+1, N}$ is obtained from $D_{N, N}^-$ in the following way: add a new edge which starts from the left corner of the central $T$-vertex, then consecutively traverses all the edges of the rectangle $e_2, e_1, \ldots, e_{N-1}, e_{N-2}$ and ends at the right corner of the central $T$-vertex. Label both left corners at this new vertex of degree $4$ with $t$ and both right corners with $t^{-1}$. Add a new $A$-vertex of degree $2$ to this new edge. Label its corners with $a$ and $a^{-1}$ so that the labels of the faces are equal to powers of $[a, t]$. See Fig.~\ref{fig:D_+N+1} for examples. Since $D_{N, N}$ has one face, the diagram $D_{+N+1, N}$ has $2$ faces. Due to symmetry, their labels are equal to $[a, t]^{\frac{N+1}{2}}$ and each $t$-arc is adjacent to these two faces.

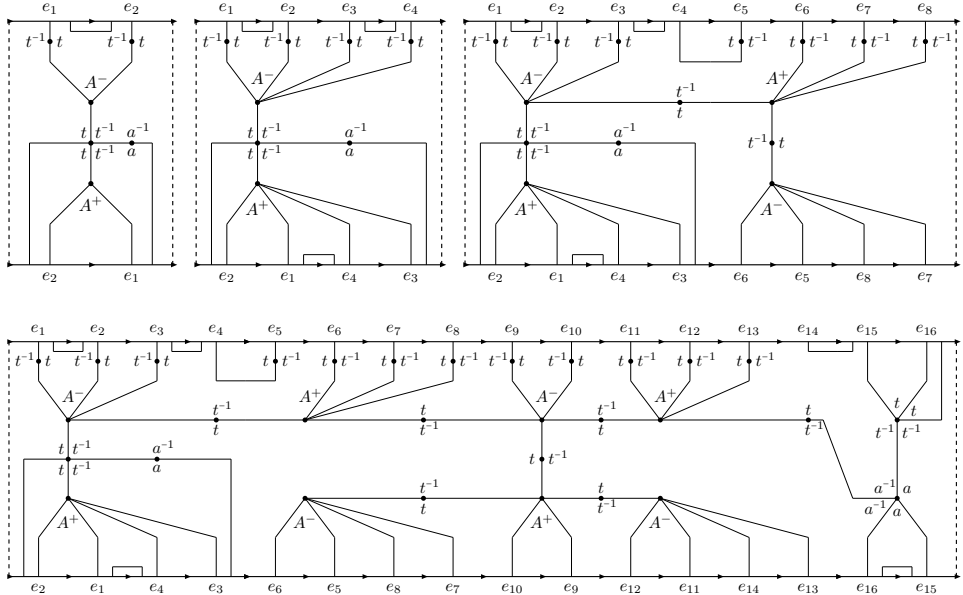
\begin{figure}[!h]
    \centering
    \bigskip
    \subfloat{
        \resizebox{0.9\textwidth}{!}{%
            \input{images/D_+N+1_3.tikz}%
            \hspace{0.3cm}%
            \input{images/D_+N+1_5.tikz}%
            \hspace{0.3cm}%
            \input{images/D_11_5.tikz}%
        }
    }
    
    \subfloat{
        \resizebox{0.9\textwidth}{!}{\input{images/D_23_5.tikz}}
    }
    \caption{$D_{+4,3}$, $D_{+6,5}$, $D_{11,5} = D_{+6,5} \dsumIII D_{5,5}$ and $D_{23, 5}$.}
    \label{fig:D_+N+1}
\end{figure}

\pagebreak

We define
$$D_{rN + 2s + 1, N} = D_{+N+1, N} \dsumIII D_{(r-1)N + 2s, N}\text{.}$$
See Fig.~\ref{fig:D_+N+1} for examples. Since $D_{(r-1)N + 2s, N}$ has one face, the diagram $D_{+N+1, N}$ has two faces and all $t$-arcs of $D_{+N+1, N}$ are adjacent to these two faces, we obtain that $D_{rN + 2s + 1, N}$ has one face. Its label is 
$$[a, t]^{\frac{N+1}{2} + \frac{N+1}{2} + (r-1)N + 2s} = [a, t]^{rN + 2s + 1}\text{.}$$
The genus of $D_{rN + 2s + 1, N}$ is 
\begin{gather*}
\frac{N-1}{2} + \intpart*{\frac{(r-1)N + 2s}{2}} - \intpart*{\frac{(r-1)N + 2s}{N}} + 1
=\\
=
\intpart*{\frac{(r-1)N + 2s}{2} + \frac{N + 1}{2}} - \intpart*{\frac{(r-1)N + 2s}{N} + 1} + 1
=\\
=
\intpart*{\frac{rN + 2s + 1}{2}} - \intpart*{\frac{rN + 2s + 1}{N}} + 1  
\end{gather*}
if $2s + 1 < N$. All the vertices of $D_{rN+2s+1, N}$ are interior if $a^N=1$.

\vspace{0.35cm}
\noindent
\textbf{Case 3. $\boldsymbol{n=rN+2s}$, where $\boldsymbol{N}$ is even}
\vspace{0.05cm}

\noindent
A diagram for this case is obtained from the diagram $D_{2,2}$ shown in Fig.~\ref{fig:D_2_2}. Its genus is $1$ and it has one face with label $[a, t]^2$. We define 
\begin{gather*}
D_{N, N} = D_{2, 2} \dsumI \underbrace{D_{3, 3} \dsumI \ldots \dsumI D_{3, 3}}_{(N - 2)\mathbin{/} 2 \text{ summands}}\text{,} \\
D_{rN, N} = D_{N, N} \dsumIIr \underbrace{D_{N-1, N-1} \dsumIIr \ldots \dsumIIr D_{N-1, N-1}}_{r-1 \text{ summands}}\text{,} \\
D_{rN+2s, N} = D_{rN, N} \dsumIII \underbrace{D_{+2} \dsumIII \ldots \dsumIII D_{+2}}_{s \text{ summands}}\text{.}
\end{gather*}
See Fig.~\ref{fig:D_2_2} for examples. Since $D_{2, 2}$ has one face, the diagram $D_{rN+2s, N}$ also has one face. The label of this face is $[a, t]^{rN+2s}$. The genus of $D_{rN+2s, N}$ is 
$$
\intpart*{\frac{rN + 2s}{2}} - \intpart*{\frac{rN + 2s}{N}} + 1
$$
if $2s < N$. It is computed analogously to the previous cases. All the vertices of $D_{rN+2s, N}$ are interior if $a^N=1$.

\vspace{0.35cm}
\noindent
\textbf{Case 4. $\boldsymbol{n=rN+2s+1}$, where $\boldsymbol{N}$ is even}
\vspace{0.05cm}

\noindent
A diagram for this case is obtained from the diagram $D_{3,2}$ shown in Fig.~\ref{fig:D_3_2}. Its genus is $1$ and it has one face with the label $[a, t]^3$. We define 
\begin{gather*}
D_{N+1, N} = D_{3, 2} \dsumI \underbrace{D_{3, 3} \dsumI \ldots \dsumI D_{3, 3}}_{(N - 2)\mathbin{/} 2 \text{ summands}}\text{,} \\
D_{rN+1, N} = D_{N+1, N} \dsumIIr \underbrace{D_{N-1, N-1} \dsumIIr \ldots \dsumIIr D_{N-1, N-1}}_{r-1 \text{ summands}}\text{,} \\
D_{rN+2s+1, N} = D_{rN+1, N} \dsumIII \underbrace{D_{+2} \dsumIII \ldots \dsumIII D_{+2}}_{s \text{ summands}}\text{.}
\end{gather*}
See Fig.~\ref{fig:D_3_2} for examples. The diagram~$D_{rN+2s+1, N}$ is exactly the same as the diagram~$D_{rN+2s, N}$, we should only replace $D_{2, 2}$ with $D_{3, 2}$. So it is easy to see that the label of its face is $[a, t]^{rN+2s+1}$ and it has the same genus as $D_{rN+2s, N}$:
$$
\intpart*{\frac{rN + 2s}{2}} - \intpart*{\frac{rN + 2s}{N}} + 1 = \intpart*{\frac{rN + 2s + 1}{2}} - \intpart*{\frac{rN + 2s + 1}{N}} + 1
$$
if $2s+1 < N$. All the vertices of $D_{rN+2s+1, N}$ are interior if $a^N=1$.

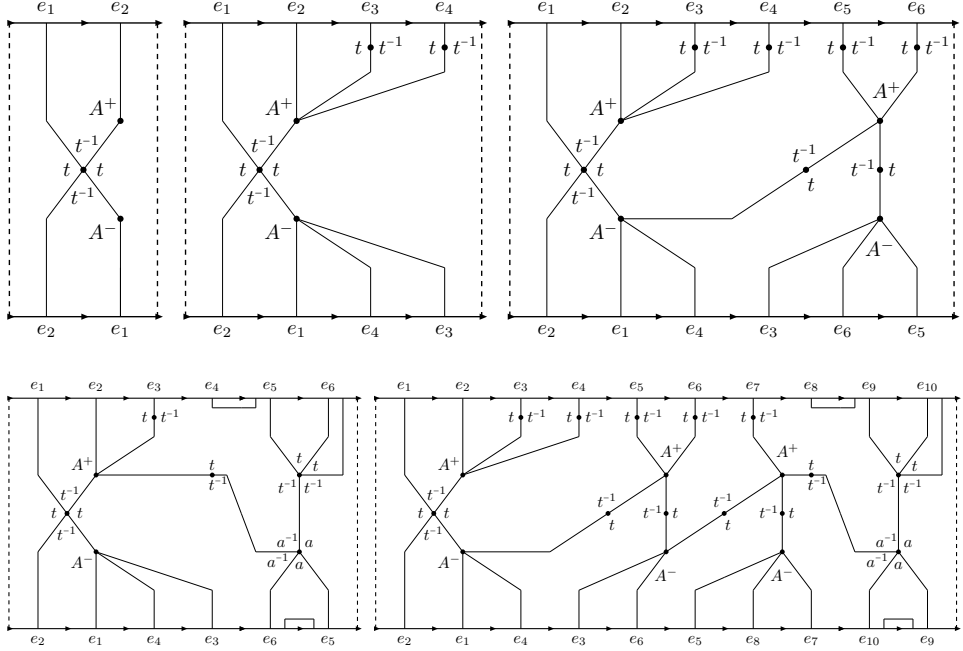
\begin{figure}[!h]
    \centering
    \bigskip
    \subfloat{
        \resizebox{0.9\textwidth}{!}{%
            \input{images/D_2_2.tikz}%
            \hspace{0.3cm}%
            \input{images/D_4_4.tikz}%
            \hspace{0.3cm}%
            \input{images/D_8_4.tikz}%
        }
    }
    
    \subfloat{
        \resizebox{0.9\textwidth}{!}{%
            \input{images/D_6_4.tikz}%
            \hspace{0.2cm}%
            \input{images/D_14_4.tikz}%
        }
    }
    \caption{$D_{2,2}$, $D_{4,4}$, $D_{8,4}$, $D_{6,4}$ and $D_{14,4}$.}
    \label{fig:D_2_2}
\end{figure}
\begin{figure}[!h]
    \centering
    \resizebox{0.9\textwidth}{!}{%
        \input{images/D_3_2.tikz}%
        \hspace{0.3cm}%
        \input{images/D_5_4.tikz}%
        \hspace{0.3cm}%
        \input{images/D_9_4.tikz}%
    }
    \caption{$D_{3, 2}$, $D_{5, 4}$ and $D_{9, 4}$.}
    \label{fig:D_3_2}
\end{figure}
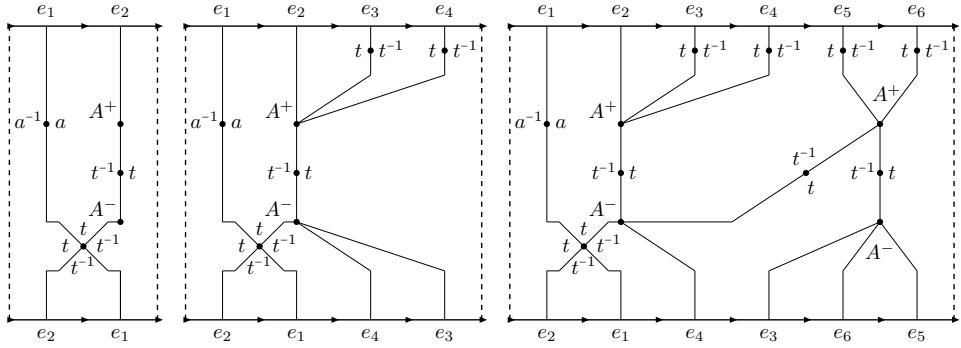

We considered all the cases and thus Lemma~\ref{D_n_N_lemma} is proved.

\end{proof}

\section{Proofs of the theorems} \label{proof_section}

\begin{proof}[Proof of Theorem~\ref{a_t_theorem}]
Let
$$
\hat{k}(n, N) = \intpart*{\frac{n}{2}} - \intpart*{\frac{n}{N}} + 1 \text{.} 
$$
We want to prove that $\cl([a, t]^n) = \hat{k}(n, \ord(a))$ if $a \in A_{j_1}$ and $t \in A_{j_2}$ are two nonidentity elements such that $\ord(t) \geq \ord(a)$ and $j_1 \neq j_2$. Let us denote $\ord(a)$ by $N$. Main theorem from~\cite{BK22} states that $\cl([a, t]^n) \geq \hat{k}(n, N)$. Hence it is sufficient to show that
$$\cl([a, t]^n) \leq \hat{k}(n, N)\text{.}$$
If $n < N$, then it follows from Culler's examples (in particular, if $N$ is infinite). If $N=2$, then it follows because 
$$[a, t]^n = [a, t^{(-1)^{n+1}} a t^{(-1)^{n}} \ldots a t^{(-1)^{2}}]\text{.}$$

Thus we assume that $N \geq 3$ and $n \geq N$. If $N$ is even or $N$ and $n$ are odd, then the desired inequality follows from Lemma~\ref{D_n_N_lemma} and Lemma~\ref{diagram_lemma}. If $N$ is odd and $n$ is even, let us consider two cases:
\begin{enumerate}
    \item $n=rN+2s$, where $r$ is even and $s$ is such that $0 \leq 2s < N$. Lemma~\ref{D_n_N_lemma} and Lemma~\ref{diagram_lemma} imply that $\cl([a, t]^{(r-1)N+2s}) \leq \hat{k}((r-1)N+2s, N)$ and $\cl([a, t]^{N}) \leq \hat{k}(N, N)$. Hence
    \begin{align*}
        \cl([a, t]&^{rN+2s}) \leq \cl([a, t]^{(r-1)N+2s}) + \cl([a, t]^{N})
        \leq \\
        &\leq \hat{k}((r-1)N+2s, N) + \hat{k}(N, N)
        = \\
        &= \frac{(r-1)N + 2s - 1}{2} - (r-1) + 1 + \frac{N-1}{2} - 1 + 1
        = \\
        &= \frac{rN + 2s}{2} - r + 1 = \hat{k}(rN+2s, N)\text{.}
    \end{align*}

    \item $n=rN+2s+1$, where $r$ is odd and $s$ is such that $0 \leq 2s+1 < N$. Lemma~\ref{D_n_N_lemma} and Lemma~\ref{diagram_lemma} imply that $\cl([a, t]^{rN}) \leq \hat{k}(rN, N)$. Moreover, $\cl([a, t]^{2s+1}) = s+1$ since $2s+1 < N$. Hence
    \begin{multline*}
        \cl([a, t]^{rN+2s+1})
        \leq
        \cl([a, t]^{rN}) + \cl([a, t]^{2s+1})
        \leq
        \hat{k}(rN, N) + s + 1
        = \\
        =
        \frac{rN - 1}{2} - r + 1 + s + 1
        =
        \frac{rN + 2s + 1}{2} - r + 1
        =
        \hat{k}(rN+2s+1, N)\text{.}
    \end{multline*}
\end{enumerate}
\end{proof}

\begin{proof}[Proof of Theorem~\ref{general_theorem}]
If $N \in \{ N(g) \mid g \in \GG \}$, then there are elements $a \in A_{j_1}$ and $t \in A_{j_2}$ such that $\ord(a)=N$, $\ord(t)\geq N$ and $j_1 \neq j_2$. Main theorem from~\cite{BK22} implies that $k(G, n, N) \geq \hat{k}(n, N)$ and Theorem~\ref{a_t_theorem} implies that $k(G, n, N) \leq \hat{k}(n, N)$ since $N([a, t])=N$. Thus $k(G, n, N) = \hat{k}(n, N)$.
\end{proof}

\begin{proof}[Proof of Theorem~\ref{main_theorem}]
There are elements $a \in A_{j_1}$ and $t \in A_{j_2}$ such that $\ord(a)=N(G)$, $\ord(t)\geq N(G)$ and $j_1 \neq j_2$. Theorem 1 from~\cite{BK22} implies that $k(G, n) \geq \hat{k}(n, N(G))$ and Theorem~\ref{a_t_theorem} implies that $k(G, n) \leq \hat{k}(n, N(G))$. Thus $k(G, n) = \hat{k}(n, N(G))$.

\end{proof}

\vspace{0.5cm}

\begin{spacing}{1.05}
{\small
\noindent
\textbf{\upshape Acknowledgments.} The author thanks his advisor, Anton Klyachko, for helpful discussions and remarks, Olga Kulikova for a valuable observation which helped the author to come to the results presented in this paper, Pavel Izmailov for proofreading the text and an anonymous referee for helpful comments. The author also thanks the Theoretical Physics and Mathematics Advancement Foundation “BASIS”.}
\end{spacing}

\begin{spacing}{1.1}
\printbibliography
\end{spacing}
\vspace{0.5cm}


{
\small
{\scshape
\noindent
Faculty of mechanics and mathematics of Moscow State University,
\par
\noindent
Moscow 119991, Leninskie gory, MSU.
\par
\noindent
Moscow Center for Fundamental and Applied Mathematics.
\par}
{\noindent
{\it Email}: {\rmfamily kuynzereb@gmail.com}}}

\end{document}

%% file: images/diagram_example.tikz
\begin{tikzpicture}
    [xscale=2,
     yscale=1,
     inner sep=0pt,
     vertex/.style={shape=circle,fill=black,minimum size=2pt},
     rect_vertex/.style={shape=rectangle,fill=black,minimum size=1.5pt},
     label/.style={scale=0.75,shape=circle,minimum size=0pt}]

\fontsize{10}{12}\selectfont

\draw[semithick]
(0, 0) rectangle (3, 2)
(0, 0) node[rect_vertex] {}
(0, 2) node[rect_vertex] {}
(3, 0) node[rect_vertex] {}
(3, 2) node[rect_vertex] {}

(1, 1) node[vertex] {}
(1, 1) node[label,left=6.5pt,above=1pt] {$a$}
(1, 1) node[label,right=11pt,above=1pt] {$a$}
(1, 1) node[label,left=0pt,below=3pt] {$a$}

(2, 1) node[vertex] {}
(2, 1) node[label,left=12pt,below=1pt] {$b$}
(2, 1) node[label,right=5.5pt,below=1pt] {$b$}
(2, 1) node[label,left=0pt,above=2.5pt] {$b$}
;

\draw
(0, 1) -- (3, 1)
(1, 1) -- (1.5, 2)
(1.5, 0) -- (2, 1)
;
\end{tikzpicture}

%% file: images/path_label.tikz
\begin{tikzpicture}
    [xscale=3,
     yscale=3,
     inner sep=0pt,
     vertex/.style={shape=circle,fill=black,minimum size=4pt},
     small_vertex/.style={shape=circle,fill=black,minimum size=2pt},
     label/.style={shape=circle,minimum size=0pt}]

\fontsize{10}{12}\selectfont

\draw (0,0) node[vertex] {};

\draw
(0,0) -- (90:1) node[vertex] {}
(90:0.5) node[small_vertex] {}

(0,0) -- (18:1) node[vertex] {}
(18:0.5) node[small_vertex] {}

(0,0) -- (-54:1) node[vertex] {}
(-54:0.5) node[small_vertex] {}

(0,0) -- (162:1) node[vertex] {}
(162:0.5) node[small_vertex] {}

(0,0) -- (-126:1) node[vertex] {}
(-126:0.5) node[small_vertex] {}

(18:1) -- +(90:1) node[vertex] {}
(18:1) +(90:0.5) node[small_vertex] {}

(18:1) -- +(-54:1) node[vertex] {}
(18:1) +(-54:0.5) node[small_vertex] {}
;

\draw (0,0)
node[label,left=13pt,below=-1.5pt] {$a_1$}
node[label,left=9pt,below=-17pt] {$a_2$}
node[label,right=9pt,below=-17pt] {$a_3$}
node[label,right=14pt,below=-1.5pt] {$a_4$}
node[label,left=-1pt,below=8pt] {$a_5$};

\draw (18:1)
node[label,left=7pt,above=2pt] {$b_1$}
node[label,above=4pt,right=5pt] {$b_2$}
node[label,left=2pt,below=4pt] {$b_3$};

\draw (90:1)
node[label,right=1pt,above=3pt] {$b_4$};

\draw[|->, ultra thick] (-126:0.5) -- (0, 0) -- (18:1) -- +(90:0.5);
\draw[|->, ultra thick, dashed] (162:0.5) -- (0, 0) -- (-54:0.5);
\draw[|->, ultra thick, dotted] (18:1) +(-54:0.5) -- (18:1);
\draw[|->, ultra thick] (90:0.5) -- (90:1);

\draw (162:0.5) node[label,above=7pt,right=5pt] {$p_1$};
\draw (18:0.5) node[label,above=3pt] {$p_2$};
\draw (18:1) +(-54:0.2) node[label,right=6pt] {$e_3$};
\draw (90:0.75) node[label,right=4pt] {$e_4$};

\end{tikzpicture}

%% file: images/contour.tikz
\begin{tikzpicture}[xscale=1.5,yscale=3]
\fontsize{10}{12}\selectfont
\newcommand*{\DrN}{%
\draw[thick] (0,1) -- (2.2,1);
\draw[thick] (2.8,1) -- (5,1);
\draw[thick] (5,0) -- (2.88,0);
\draw[thick] (2.28,0) -- (0,0);
\draw[dashed, thick] (0,0) -- (0,1);
\draw[dashed, thick] (5,1) -- (5,0);
\draw (0,0) node[inner sep=0pt,shape=isosceles triangle,minimum size=3pt,fill=black] {};
\draw (0,1) node[inner sep=0pt,shape=isosceles triangle,minimum size=3pt,fill=black] {};
\draw (1,0) node[inner sep=0pt,shape=isosceles triangle,minimum size=3pt,fill=black] {};
\draw (1,1) node[inner sep=0pt,shape=isosceles triangle,minimum size=3pt,fill=black] {};
\draw (2,0) node[inner sep=0pt,shape=isosceles triangle,minimum size=3pt,fill=black] {};
\draw (2,1) node[inner sep=0pt,shape=isosceles triangle,minimum size=3pt,fill=black] {};
\draw (3,0) node[inner sep=0pt,shape=isosceles triangle,minimum size=3pt,fill=black] {};
\draw (3,1) node[inner sep=0pt,shape=isosceles triangle,minimum size=3pt,fill=black] {};
\draw (4,0) node[inner sep=0pt,shape=isosceles triangle,minimum size=3pt,fill=black] {};
\draw (4,1) node[inner sep=0pt,shape=isosceles triangle,minimum size=3pt,fill=black] {};
\draw (5,0) node[inner sep=0pt,shape=isosceles triangle,minimum size=3pt,fill=black] {};
\draw (5,1) node[inner sep=0pt,shape=isosceles triangle,minimum size=3pt,fill=black] {};
\draw (0.5,0) node[minimum size=14pt,below=0pt] {$e_{2}$};
\draw (0.5,1) node[minimum size=14pt,above=0pt] {$e_{1}$};
\draw (1.5,0) node[minimum size=14pt,below=0pt] {$e_{1}$};
\draw (1.5,1) node[minimum size=14pt,above=0pt] {$e_{2}$};
\draw (3.5,0) node[minimum size=14pt,below=0pt] {$e_{2k}$};
\draw (3.5,1) node[minimum size=14pt,above=0pt] {$e_{2k-1}$};
\draw (4.5,0) node[minimum size=14pt,below=0pt] {$e_{2k-1}$};
\draw (4.5,1) node[minimum size=14pt,above=0pt] {$e_{2k}$};
}


\begin{scope}[shift={(0,0)}]
    \DrN
\end{scope}

\begin{scope}[shift={(0.04,0)}]
\draw[thick]
(2.25, 1.1) 
.. controls (2.1, 0.75) ..
(2.25, 0.5)
.. controls (2.4, 0.25) ..
(2.25, -0.1);
\draw[thick]
(2.75, 1.1) 
.. controls (2.6, 0.75) ..
(2.75, 0.5)
.. controls (2.9, 0.25) ..
(2.75, -0.1);
\end{scope}

\end{tikzpicture}

%% file: images/D_1_1.tikz
\begin{tikzpicture}[xscale=3.0,yscale=6]
\providecommand{\unaryminus}{\scalebox{0.5}[0.75]{\( - \)}}
\fontsize{10}{12}\selectfont
\newcommand*{\DrN}{%
\draw[dashed, thick] (1,0) .. controls(0.25, 0.5) .. (1,1);
\draw[dashed, thick] (1,0) .. controls(1.75, 0.5) .. (1,1);
\draw (1,0) node[scale=1.2000000000000002,inner sep=0pt,shape=isosceles triangle,minimum size=3pt,fill=black] {};
\draw (1,1) node[scale=1.2000000000000002,inner sep=0pt,shape=isosceles triangle,minimum size=3pt,fill=black] {};
\draw (1,0.3333333333333333) node[scale=1.2000000000000002,inner sep=0pt,shape=circle,minimum size=3pt,fill=black] {};
\draw (1,0.3333333333333333) node[scale=1.2000000000000002,inner sep=0pt,below=5.3665631459994945pt] {$A^{-}$};
\draw (1,0.6666666666666666) node[scale=1.2000000000000002,inner sep=0pt,shape=circle,minimum size=3pt,fill=black] {};
\draw (1,0.6666666666666666) node[scale=1.2000000000000002,inner sep=0pt,above=10.366563145999494pt,right=-4pt] {$A^{+}$};
    (1.5,0.8333333333333334);
\draw[] (1,0.3333333333333333) -- (1,0.6666666666666666);
\draw (1,0.5) node[scale=1.2000000000000002,inner sep=0pt,shape=circle,minimum size=3pt,fill=black] {};
\draw (1,0.5) node[scale=1.2000000000000002,inner sep=0pt,right=3.8pt] {$t$};
\draw (1,0.5) node[scale=1.2000000000000002,inner sep=0pt,left=8pt,above=-3.3pt] {$t^{\unaryminus 1}$};
\draw (1,0) node[scale=1.2000000000000002,inner sep=0pt,minimum size=14pt,below=0pt] {};
\draw (1,1) node[scale=1.2000000000000002,inner sep=0pt,minimum size=14pt,above=0pt] {};
}

\begin{scope}[shift={(0,0)}]
    \DrN
\end{scope}

\end{tikzpicture}

%% file: images/D_3_3.tikz
\begin{tikzpicture}[xscale=3.0,yscale=6]
\providecommand{\unaryminus}{\scalebox{0.5}[0.75]{\( - \)}}
\fontsize{10}{12}\selectfont
\newcommand*{\DrN}{%
\draw[thick] (0,1) -- (2,1);
\draw[thick] (2,0) -- (0,0);
\draw[dashed, thick] (0,0) -- (0,1);
\draw[dashed, thick] (2,1) -- (2,0);
\draw (0,0) node[scale=1.2000000000000002,inner sep=0pt,shape=isosceles triangle,minimum size=3pt,fill=black] {};
\draw (0,1) node[scale=1.2000000000000002,inner sep=0pt,shape=isosceles triangle,minimum size=3pt,fill=black] {};
\draw (1,0) node[scale=1.2000000000000002,inner sep=0pt,shape=isosceles triangle,minimum size=3pt,fill=black] {};
\draw (1,1) node[scale=1.2000000000000002,inner sep=0pt,shape=isosceles triangle,minimum size=3pt,fill=black] {};
\draw (2,0) node[scale=1.2000000000000002,inner sep=0pt,shape=isosceles triangle,minimum size=3pt,fill=black] {};
\draw (2,1) node[scale=1.2000000000000002,inner sep=0pt,shape=isosceles triangle,minimum size=3pt,fill=black] {};
\draw (0.5,0) node[scale=1.2000000000000002,inner sep=0pt,minimum size=14pt,below=0pt] {$e_{2}$};
\draw (0.5,1) node[scale=1.2000000000000002,inner sep=0pt,minimum size=14pt,above=0pt] {$e_{1}$};
\draw (1.5,0) node[scale=1.2000000000000002,inner sep=0pt,minimum size=14pt,below=0pt] {$e_{1}$};
\draw (1.5,1) node[scale=1.2000000000000002,inner sep=0pt,minimum size=14pt,above=0pt] {$e_{2}$};
\draw (1,0.3333333333333333) node[scale=1.2000000000000002,inner sep=0pt,shape=circle,minimum size=3pt,fill=black] {};
\draw (1,0.3333333333333333) node[scale=1.2000000000000002,inner sep=0pt,below=5.3665631459994945pt] {$A^{-}$};
\draw (1,0.6666666666666666) node[scale=1.2000000000000002,inner sep=0pt,shape=circle,minimum size=3pt,fill=black] {};
\draw (1,0.6666666666666666) node[scale=1.2000000000000002,inner sep=0pt,above=10.366563145999494pt,right=-4pt] {$A^{+}$};
\draw[] (1,0.3333333333333333) -- (0.5,0.16666666666666666);
\draw[] (0.5,0.16666666666666666) -- (0.5,0);
\draw[] (1,0.6666666666666666) -- (0.5,0.8333333333333334);
\draw[] (0.5,0.8333333333333334) -- (0.5,1);
\draw (0.5,0.9166666666666666) node[scale=1.2000000000000002,inner sep=0pt,shape=circle,minimum size=3pt,fill=black] {};
\draw (0.5,0.9166666666666666) node[scale=1.2000000000000002,inner sep=0pt,left=3.8pt] {$t$};
\draw (0.5,0.9166666666666666) node[scale=1.2000000000000002,inner sep=0pt,right=10pt,above=-3.3pt] {$t^{\unaryminus 1}$};
\draw[] (1,0.3333333333333333) -- (1.5,0.16666666666666666);
\draw[] (1.5,0.16666666666666666) -- (1.5,0);
\draw[] (1,0.6666666666666666) -- (1.5,0.8333333333333334);
\draw[] (1.5,0.8333333333333334) -- (1.5,1);
\draw (1.5,0.9166666666666666) node[scale=1.2000000000000002,inner sep=0pt,shape=circle,minimum size=3pt,fill=black] {};
\draw (1.5,0.9166666666666666) node[scale=1.2000000000000002,inner sep=0pt,left=3.8pt] {$t$};
\draw (1.5,0.9166666666666666) node[scale=1.2000000000000002,inner sep=0pt,right=10pt,above=-3.3pt] {$t^{\unaryminus 1}$};
\draw[] (1,0.3333333333333333) -- (1,0.6666666666666666);
\draw (1,0.5) node[scale=1.2000000000000002,inner sep=0pt,shape=circle,minimum size=3pt,fill=black] {};
\draw (1,0.5) node[scale=1.2000000000000002,inner sep=0pt,right=3.8pt] {$t$};
\draw (1,0.5) node[scale=1.2000000000000002,inner sep=0pt,left=8pt,above=-3.3pt] {$t^{\unaryminus 1}$};
}


\begin{scope}[shift={(0,0)}]
    \DrN
\end{scope}

\end{tikzpicture}

%% file: images/D_5_5.tikz
\begin{tikzpicture}[xscale=2.0,yscale=6]
\providecommand{\unaryminus}{\scalebox{0.5}[0.75]{\( - \)}}
\fontsize{10}{12}\selectfont
\newcommand*{\DrN}{%
\draw[thick] (0,1) -- (4,1);
\draw[thick] (4,0) -- (0,0);
\draw[dashed, thick] (0,0) -- (0,1);
\draw[dashed, thick] (4,1) -- (4,0);
\draw (0,0) node[scale=1.2000000000000002,inner sep=0pt,shape=isosceles triangle,minimum size=3pt,fill=black] {};
\draw (0,1) node[scale=1.2000000000000002,inner sep=0pt,shape=isosceles triangle,minimum size=3pt,fill=black] {};
\draw (1,0) node[scale=1.2000000000000002,inner sep=0pt,shape=isosceles triangle,minimum size=3pt,fill=black] {};
\draw (1,1) node[scale=1.2000000000000002,inner sep=0pt,shape=isosceles triangle,minimum size=3pt,fill=black] {};
\draw (2,0) node[scale=1.2000000000000002,inner sep=0pt,shape=isosceles triangle,minimum size=3pt,fill=black] {};
\draw (2,1) node[scale=1.2000000000000002,inner sep=0pt,shape=isosceles triangle,minimum size=3pt,fill=black] {};
\draw (3,0) node[scale=1.2000000000000002,inner sep=0pt,shape=isosceles triangle,minimum size=3pt,fill=black] {};
\draw (3,1) node[scale=1.2000000000000002,inner sep=0pt,shape=isosceles triangle,minimum size=3pt,fill=black] {};
\draw (4,0) node[scale=1.2000000000000002,inner sep=0pt,shape=isosceles triangle,minimum size=3pt,fill=black] {};
\draw (4,1) node[scale=1.2000000000000002,inner sep=0pt,shape=isosceles triangle,minimum size=3pt,fill=black] {};
\draw (0.5,0) node[scale=1.2000000000000002,inner sep=0pt,minimum size=14pt,below=0pt] {$e_{2}$};
\draw (0.5,1) node[scale=1.2000000000000002,inner sep=0pt,minimum size=14pt,above=0pt] {$e_{1}$};
\draw (1.5,0) node[scale=1.2000000000000002,inner sep=0pt,minimum size=14pt,below=0pt] {$e_{1}$};
\draw (1.5,1) node[scale=1.2000000000000002,inner sep=0pt,minimum size=14pt,above=0pt] {$e_{2}$};
\draw (2.5,0) node[scale=1.2000000000000002,inner sep=0pt,minimum size=14pt,below=0pt] {$e_{4}$};
\draw (2.5,1) node[scale=1.2000000000000002,inner sep=0pt,minimum size=14pt,above=0pt] {$e_{3}$};
\draw (3.5,0) node[scale=1.2000000000000002,inner sep=0pt,minimum size=14pt,below=0pt] {$e_{3}$};
\draw (3.5,1) node[scale=1.2000000000000002,inner sep=0pt,minimum size=14pt,above=0pt] {$e_{4}$};
\draw (1,0.3333333333333333) node[scale=1.2000000000000002,inner sep=0pt,shape=circle,minimum size=3pt,fill=black] {};
\draw (1,0.3333333333333333) node[scale=1.2000000000000002,inner sep=0pt,below=7.589466384404111pt] {$A^{-}$};
\draw (1,0.6666666666666666) node[scale=1.2000000000000002,inner sep=0pt,shape=circle,minimum size=3pt,fill=black] {};
\draw (1,0.6666666666666666) node[scale=1.2000000000000002,inner sep=0pt,above=12.589466384404112pt,right=-4pt] {$A^{+}$};
\draw[] (1,0.3333333333333333) -- (0.5,0.16666666666666666);
\draw[] (0.5,0.16666666666666666) -- (0.5,0);
\draw[] (1,0.6666666666666666) -- (0.5,0.8333333333333334);
\draw[] (0.5,0.8333333333333334) -- (0.5,1);
\draw (0.5,0.9166666666666666) node[scale=1.2000000000000002,inner sep=0pt,shape=circle,minimum size=3pt,fill=black] {};
\draw (0.5,0.9166666666666666) node[scale=1.2000000000000002,inner sep=0pt,left=3.8pt] {$t$};
\draw (0.5,0.9166666666666666) node[scale=1.2000000000000002,inner sep=0pt,right=10pt,above=-3.3pt] {$t^{\unaryminus 1}$};
\draw[] (1,0.3333333333333333) -- (1.5,0.16666666666666666);
\draw[] (1.5,0.16666666666666666) -- (1.5,0);
\draw[] (1,0.6666666666666666) -- (1.5,0.8333333333333334);
\draw[] (1.5,0.8333333333333334) -- (1.5,1);
\draw (1.5,0.9166666666666666) node[scale=1.2000000000000002,inner sep=0pt,shape=circle,minimum size=3pt,fill=black] {};
\draw (1.5,0.9166666666666666) node[scale=1.2000000000000002,inner sep=0pt,left=3.8pt] {$t$};
\draw (1.5,0.9166666666666666) node[scale=1.2000000000000002,inner sep=0pt,right=10pt,above=-3.3pt] {$t^{\unaryminus 1}$};
\draw[] (1,0.3333333333333333) -- (2.5,0.16666666666666666);
\draw[] (2.5,0.16666666666666666) -- (2.5,0);
\draw[] (1,0.6666666666666666) -- (2.5,0.8333333333333334);
\draw[] (2.5,0.8333333333333334) -- (2.5,1);
\draw (2.5,0.9166666666666666) node[scale=1.2000000000000002,inner sep=0pt,shape=circle,minimum size=3pt,fill=black] {};
\draw (2.5,0.9166666666666666) node[scale=1.2000000000000002,inner sep=0pt,left=3.8pt] {$t$};
\draw (2.5,0.9166666666666666) node[scale=1.2000000000000002,inner sep=0pt,right=10pt,above=-3.3pt] {$t^{\unaryminus 1}$};
\draw[] (1,0.3333333333333333) -- (3.5,0.16666666666666666);
\draw[] (3.5,0.16666666666666666) -- (3.5,0);
\draw[] (1,0.6666666666666666) -- (3.5,0.8333333333333334);
\draw[] (3.5,0.8333333333333334) -- (3.5,1);
\draw (3.5,0.9166666666666666) node[scale=1.2000000000000002,inner sep=0pt,shape=circle,minimum size=3pt,fill=black] {};
\draw (3.5,0.9166666666666666) node[scale=1.2000000000000002,inner sep=0pt,left=3.8pt] {$t$};
\draw (3.5,0.9166666666666666) node[scale=1.2000000000000002,inner sep=0pt,right=10pt,above=-3.3pt] {$t^{\unaryminus 1}$};
\draw[] (1,0.3333333333333333) -- (1,0.6666666666666666);
\draw (1,0.5) node[scale=1.2000000000000002,inner sep=0pt,shape=circle,minimum size=3pt,fill=black] {};
\draw (1,0.5) node[scale=1.2000000000000002,inner sep=0pt,right=3.8pt] {$t$};
\draw (1,0.5) node[scale=1.2000000000000002,inner sep=0pt,left=8pt,above=-3.3pt] {$t^{\unaryminus 1}$};
}


\begin{scope}[shift={(0,0)}]
    \DrN
\end{scope}

\end{tikzpicture}

%% file: images/sumII_example.tikz
\begin{tikzpicture}[xscale=2.0,yscale=6]
\providecommand{\unaryminus}{\scalebox{0.5}[0.75]{\( - \)}}
\fontsize{10}{12}\selectfont

\newcommand*{\DrN}{%
\draw[thick] (0,1) -- (2,1);
\draw[thick] (2,0) -- (0,0);
\draw (0,0) node[scale=1.2000000000000002,inner sep=0pt,shape=isosceles triangle,minimum size=3pt,fill=black] {};
\draw (0,1) node[scale=1.2000000000000002,inner sep=0pt,shape=isosceles triangle,minimum size=3pt,fill=black] {};
\draw (1,0) node[scale=1.2000000000000002,inner sep=0pt,shape=isosceles triangle,minimum size=3pt,fill=black] {};
\draw (1,1) node[scale=1.2000000000000002,inner sep=0pt,shape=isosceles triangle,minimum size=3pt,fill=black] {};
\draw (2,0) node[scale=1.2000000000000002,inner sep=0pt,shape=isosceles triangle,minimum size=3pt,fill=black] {};
\draw (2,1) node[scale=1.2000000000000002,inner sep=0pt,shape=isosceles triangle,minimum size=3pt,fill=black] {};
\draw (1,0.3333333333333333) node[scale=1.2000000000000002,inner sep=0pt,shape=circle,minimum size=3pt,fill=black] {};
\draw (1,0.6666666666666666) node[scale=1.2000000000000002,inner sep=0pt,shape=circle,minimum size=3pt,fill=black] {};
\draw[] (1,0.3333333333333333) -- (0.5,0.16666666666666666);
\draw[] (0.5,0.16666666666666666) -- (0.5,0);
\draw[] (1,0.6666666666666666) -- (0.5,0.8333333333333334);
\draw[] (0.5,0.8333333333333334) -- (0.5,1);

\draw[] (1,0.3333333333333333) -- (1.5,0.16666666666666666);
\draw[] (1.5,0.16666666666666666) -- (1.5,0);
\draw[] (1,0.6666666666666666) -- (1.5,0.8333333333333334);
\draw[] (1.5,0.8333333333333334) -- (1.5,1);

\draw (0.5,0.9166666666666666) node[scale=1.2000000000000002,inner sep=0pt,shape=circle,minimum size=3pt,fill=black] {};
\draw (1.5,0.9166666666666666) node[scale=1.2000000000000002,inner sep=0pt,shape=circle,minimum size=3pt,fill=black] {};

}

\begin{scope}[shift={(0,0)}]
    \DrN
    \draw[dashed, thick] (0,0) -- (0,1);
    \draw[dashed, thick] (2,1) -- (2,0);
    \draw (0.5,0) node[scale=1.2000000000000002,inner sep=0pt,minimum size=14pt,below=0pt] {$e_{2}$};
    \draw (0.5,1) node[scale=1.2000000000000002,inner sep=0pt,minimum size=14pt,above=0pt] {$e_{1}$};
    \draw (1.5,0) node[scale=1.2000000000000002,inner sep=0pt,minimum size=14pt,below=0pt] {$e_{1}$};
    \draw (1.5,1) node[scale=1.2000000000000002,inner sep=0pt,minimum size=14pt,above=0pt] {$e_{2}$};
    \draw[] (1,0.3333333333333333) -- (1,0.6666666666666666);
    \draw (1,0.5) node[scale=1.2000000000000002,inner sep=0pt,shape=circle,minimum size=3pt,fill=black] {};
    \draw (1, 0.5) node[scale=1.2000000000000002,inner sep=0pt,left=3.8pt] {$t$};
    \draw (1, 0.5) node[scale=1.2000000000000002,inner sep=0pt,right=10pt,above=-3.3pt] {$t^{\unaryminus 1}$};
    \draw (1,0.3333333333333333) node[scale=1.2000000000000002,inner sep=0pt,below=5.3665631459994945pt] {$A^{+}$};
    \draw (1,0.6666666666666666) node[scale=1.2000000000000002,inner sep=0pt,above=10.366563145999494pt,right=-4pt] {$A^{-}$};
    
    \draw (0.5,0.9166666666666666) node[scale=1.2000000000000002,inner sep=0pt,right=3.8pt] {$t$};
    \draw (1.5,0.9166666666666666) node[scale=1.2000000000000002,inner sep=0pt,right=3.8pt] {$t$};
    \draw (0.5,0.9166666666666666) node[scale=1.2000000000000002,inner sep=0pt,left=8pt,above=-3.3pt] {$t^{\unaryminus 1}$};
    \draw (1.5,0.9166666666666666) node[scale=1.2000000000000002,inner sep=0pt,left=8pt,above=-3.3pt] {$t^{\unaryminus 1}$};
\end{scope}

\begin{scope}[shift={(3,0)}]
    \DrN
    \draw[dashed, thick] (0,0) -- (0,1);
    \draw[dashed, thick] (2,1) -- (2,0);
    \draw (0.5,0) node[scale=1.2000000000000002,inner sep=0pt,minimum size=14pt,below=0pt] {$e_{2}$};
    \draw (0.5,1) node[scale=1.2000000000000002,inner sep=0pt,minimum size=14pt,above=0pt] {$e_{1}$};
    \draw (1.5,0) node[scale=1.2000000000000002,inner sep=0pt,minimum size=14pt,below=0pt] {$e_{1}$};
    \draw (1.5,1) node[scale=1.2000000000000002,inner sep=0pt,minimum size=14pt,above=0pt] {$e_{2}$};
    \draw[] (1,0.3333333333333333) -- (1,0.6666666666666666);
    \draw (1,0.5) node[scale=1.2000000000000002,inner sep=0pt,shape=circle,minimum size=3pt,fill=black] {};
    \draw (1,0.5) node[scale=1.2000000000000002,inner sep=0pt,right=3.8pt] {$t$};
    \draw (1,0.5) node[scale=1.2000000000000002,inner sep=0pt,left=8pt,above=-3.3pt] {$t^{\unaryminus 1}$};
    \draw (1,0.3333333333333333) node[scale=1.2000000000000002,inner sep=0pt,below=5.3665631459994945pt] {$A^{-}$};
    \draw (1,0.6666666666666666) node[scale=1.2000000000000002,inner sep=0pt,above=10.366563145999494pt,right=-4pt] {$A^{+}$};
    \draw (0.5,0.9166666666666666) node[scale=1.2000000000000002,inner sep=0pt,left=3.8pt] {$t$};
    \draw (0.5,0.9166666666666666) node[scale=1.2000000000000002,inner sep=0pt,right=10pt,above=-3.3pt] {$t^{\unaryminus 1}$};
    \draw (1.5,0.9166666666666666) node[scale=1.2000000000000002,inner sep=0pt,left=3.8pt] {$t$};
    \draw (1.5,0.9166666666666666) node[scale=1.2000000000000002,inner sep=0pt,right=10pt,above=-3.3pt] {$t^{\unaryminus 1}$};
\end{scope}

\begin{scope}[shift={(6,0)}]
    \begin{scope}[shift={(0,0)}]
        \DrN
        \draw (0.5,0) node[scale=1.2000000000000002,inner sep=0pt,minimum size=14pt,below=0pt] {$e_{2}$};
        \draw (0.5,1) node[scale=1.2000000000000002,inner sep=0pt,minimum size=14pt,above=0pt] {$e_{1}$};
        \draw (1.5,0) node[scale=1.2000000000000002,inner sep=0pt,minimum size=14pt,below=0pt] {$e_{1}$};
        \draw (1.5,1) node[scale=1.2000000000000002,inner sep=0pt,minimum size=14pt,above=0pt] {$e_{2}$};
        \draw (1,0.3333333333333333) node[scale=1.2000000000000002,inner sep=0pt,below=5.3665631459994945pt] {$A^{+}$};
        \draw (1,0.6666666666666666) node[scale=1.2000000000000002,inner sep=0pt,above=10.366563145999494pt,right=-4pt] {$A^{-}$};
        \draw (0.5,0.9166666666666666) node[scale=1.2000000000000002,inner sep=0pt,right=3.8pt] {$t$};
        \draw (1.5,0.9166666666666666) node[scale=1.2000000000000002,inner sep=0pt,right=3.8pt] {$t$};
        \draw (0.5,0.9166666666666666) node[scale=1.2000000000000002,inner sep=0pt,left=8pt,above=-3.3pt] {$t^{\unaryminus 1}$};
        \draw (1.5,0.9166666666666666) node[scale=1.2000000000000002,inner sep=0pt,left=8pt,above=-3.3pt] {$t^{\unaryminus 1}$};
    \end{scope}
    
    \begin{scope}[shift={(2,0)}]
        \DrN
        \draw (0.5,0) node[scale=1.2000000000000002,inner sep=0pt,minimum size=14pt,below=0pt] {$e_{4}$};
        \draw (0.5,1) node[scale=1.2000000000000002,inner sep=0pt,minimum size=14pt,above=0pt] {$e_{3}$};
        \draw (1.5,0) node[scale=1.2000000000000002,inner sep=0pt,minimum size=14pt,below=0pt] {$e_{3}$};
        \draw (1.5,1) node[scale=1.2000000000000002,inner sep=0pt,minimum size=14pt,above=0pt] {$e_{4}$};
        \draw[] (1,0.3333333333333333) -- (1,0.6666666666666666);
        \draw (1,0.5) node[scale=1.2000000000000002,inner sep=0pt,shape=circle,minimum size=3pt,fill=black] {};
        \draw (1,0.5) node[scale=1.2000000000000002,inner sep=0pt,right=3.8pt] {$t$};
        \draw (1,0.5) node[scale=1.2000000000000002,inner sep=0pt,left=8pt,above=-3.3pt] {$t^{\unaryminus 1}$};
        \draw (1,0.3333333333333333) node[scale=1.2000000000000002,inner sep=0pt,below=5.3665631459994945pt] {$A^{-}$};
        \draw (1,0.6666666666666666) node[scale=1.2000000000000002,inner sep=0pt,above=10.366563145999494pt,right=-4pt] {$A^{+}$};
        \draw (0.5,0.9166666666666666) node[scale=1.2000000000000002,inner sep=0pt,left=3.8pt] {$t$};
        \draw (0.5,0.9166666666666666) node[scale=1.2000000000000002,inner sep=0pt,right=10pt,above=-3.3pt] {$t^{\unaryminus 1}$};
        \draw (1.5,0.9166666666666666) node[scale=1.2000000000000002,inner sep=0pt,left=3.8pt] {$t$};
        \draw (1.5,0.9166666666666666) node[scale=1.2000000000000002,inner sep=0pt,right=10pt,above=-3.3pt] {$t^{\unaryminus 1}$};
    \end{scope}
    \draw[dashed, thick] (0,0) -- (0,1);
    \draw[dashed, thick] (4,1) -- (4,0);
    \draw[] (1,0.3333333333333333) -- (3,0.3333333333333333);
    \draw[] (1,0.6666666666666666) -- (3,0.6666666666666666);
    
    \draw (2,0.3333333333333333) node[scale=1.2000000000000002,inner sep=0pt,shape=circle,minimum size=3pt,fill=black] {};
    \draw (2,0.3333333333333333) node[scale=1.2000000000000002,inner sep=0pt,above=3pt] {$t$};
    \draw (2,0.3333333333333333) node[scale=1.2000000000000002,inner sep=0pt,below=5pt,right=-2.3pt] {$t^{\unaryminus 1}$};
    
    \draw (2,0.6666666666666666) node[scale=1.2000000000000002,inner sep=0pt,shape=circle,minimum size=3pt,fill=black] {};
    \draw (2,0.6666666666666666) node[scale=1.2000000000000002,inner sep=0pt,below=3pt] {$t$};
    \draw (2,0.6666666666666666) node[scale=1.2000000000000002,inner sep=0pt,above=7pt,right=-2.3pt] {$t^{\unaryminus 1}$};
\end{scope}

\draw (2.5, 0.5) node[scale=2.5,inner sep=0pt] {$\mathrm{\sqsupset}$};
\draw (5.5, 0.5) node[scale=2.5,inner sep=0pt] {$=$};


\end{tikzpicture}

%% file: images/D_15_5.tikz
\begin{tikzpicture}[xscale=1.5,yscale=6]
\providecommand{\unaryminus}{\scalebox{0.5}[0.75]{\( - \)}}
\fontsize{10}{12}\selectfont
\newcommand*{\DrN}{%
\draw[thick] (0,1) -- (10,1);
\draw[thick] (10,0) -- (0,0);
\draw[dashed, thick] (0,0) -- (0,1);
\draw[dashed, thick] (10,1) -- (10,0);
\draw (0,0) node[scale=1.2000000000000002,inner sep=0pt,shape=isosceles triangle,minimum size=3pt,fill=black] {};
\draw (0,1) node[scale=1.2000000000000002,inner sep=0pt,shape=isosceles triangle,minimum size=3pt,fill=black] {};
\draw (1,0) node[scale=1.2000000000000002,inner sep=0pt,shape=isosceles triangle,minimum size=3pt,fill=black] {};
\draw (1,1) node[scale=1.2000000000000002,inner sep=0pt,shape=isosceles triangle,minimum size=3pt,fill=black] {};
\draw (2,0) node[scale=1.2000000000000002,inner sep=0pt,shape=isosceles triangle,minimum size=3pt,fill=black] {};
\draw (2,1) node[scale=1.2000000000000002,inner sep=0pt,shape=isosceles triangle,minimum size=3pt,fill=black] {};
\draw (3,0) node[scale=1.2000000000000002,inner sep=0pt,shape=isosceles triangle,minimum size=3pt,fill=black] {};
\draw (3,1) node[scale=1.2000000000000002,inner sep=0pt,shape=isosceles triangle,minimum size=3pt,fill=black] {};
\draw (4,0) node[scale=1.2000000000000002,inner sep=0pt,shape=isosceles triangle,minimum size=3pt,fill=black] {};
\draw (4,1) node[scale=1.2000000000000002,inner sep=0pt,shape=isosceles triangle,minimum size=3pt,fill=black] {};
\draw (5,0) node[scale=1.2000000000000002,inner sep=0pt,shape=isosceles triangle,minimum size=3pt,fill=black] {};
\draw (5,1) node[scale=1.2000000000000002,inner sep=0pt,shape=isosceles triangle,minimum size=3pt,fill=black] {};
\draw (6,0) node[scale=1.2000000000000002,inner sep=0pt,shape=isosceles triangle,minimum size=3pt,fill=black] {};
\draw (6,1) node[scale=1.2000000000000002,inner sep=0pt,shape=isosceles triangle,minimum size=3pt,fill=black] {};
\draw (7,0) node[scale=1.2000000000000002,inner sep=0pt,shape=isosceles triangle,minimum size=3pt,fill=black] {};
\draw (7,1) node[scale=1.2000000000000002,inner sep=0pt,shape=isosceles triangle,minimum size=3pt,fill=black] {};
\draw (8,0) node[scale=1.2000000000000002,inner sep=0pt,shape=isosceles triangle,minimum size=3pt,fill=black] {};
\draw (8,1) node[scale=1.2000000000000002,inner sep=0pt,shape=isosceles triangle,minimum size=3pt,fill=black] {};
\draw (9,0) node[scale=1.2000000000000002,inner sep=0pt,shape=isosceles triangle,minimum size=3pt,fill=black] {};
\draw (9,1) node[scale=1.2000000000000002,inner sep=0pt,shape=isosceles triangle,minimum size=3pt,fill=black] {};
\draw (10,0) node[scale=1.2000000000000002,inner sep=0pt,shape=isosceles triangle,minimum size=3pt,fill=black] {};
\draw (10,1) node[scale=1.2000000000000002,inner sep=0pt,shape=isosceles triangle,minimum size=3pt,fill=black] {};
\draw (0.5,0) node[scale=1.2000000000000002,inner sep=0pt,minimum size=14pt,below=0pt] {$e_{2}$};
\draw (0.5,1) node[scale=1.2000000000000002,inner sep=0pt,minimum size=14pt,above=0pt] {$e_{1}$};
\draw (1.5,0) node[scale=1.2000000000000002,inner sep=0pt,minimum size=14pt,below=0pt] {$e_{1}$};
\draw (1.5,1) node[scale=1.2000000000000002,inner sep=0pt,minimum size=14pt,above=0pt] {$e_{2}$};
\draw (2.5,0) node[scale=1.2000000000000002,inner sep=0pt,minimum size=14pt,below=0pt] {$e_{4}$};
\draw (2.5,1) node[scale=1.2000000000000002,inner sep=0pt,minimum size=14pt,above=0pt] {$e_{3}$};
\draw (3.5,0) node[scale=1.2000000000000002,inner sep=0pt,minimum size=14pt,below=0pt] {$e_{3}$};
\draw (3.5,1) node[scale=1.2000000000000002,inner sep=0pt,minimum size=14pt,above=0pt] {$e_{4}$};
\draw (4.5,0) node[scale=1.2000000000000002,inner sep=0pt,minimum size=14pt,below=0pt] {$e_{6}$};
\draw (4.5,1) node[scale=1.2000000000000002,inner sep=0pt,minimum size=14pt,above=0pt] {$e_{5}$};
\draw (5.5,0) node[scale=1.2000000000000002,inner sep=0pt,minimum size=14pt,below=0pt] {$e_{5}$};
\draw (5.5,1) node[scale=1.2000000000000002,inner sep=0pt,minimum size=14pt,above=0pt] {$e_{6}$};
\draw (6.5,0) node[scale=1.2000000000000002,inner sep=0pt,minimum size=14pt,below=0pt] {$e_{8}$};
\draw (6.5,1) node[scale=1.2000000000000002,inner sep=0pt,minimum size=14pt,above=0pt] {$e_{7}$};
\draw (7.5,0) node[scale=1.2000000000000002,inner sep=0pt,minimum size=14pt,below=0pt] {$e_{7}$};
\draw (7.5,1) node[scale=1.2000000000000002,inner sep=0pt,minimum size=14pt,above=0pt] {$e_{8}$};
\draw (8.5,0) node[scale=1.2000000000000002,inner sep=0pt,minimum size=14pt,below=0pt] {$e_{10}$};
\draw (8.5,1) node[scale=1.2000000000000002,inner sep=0pt,minimum size=14pt,above=0pt] {$e_{9}$};
\draw (9.5,0) node[scale=1.2000000000000002,inner sep=0pt,minimum size=14pt,below=0pt] {$e_{9}$};
\draw (9.5,1) node[scale=1.2000000000000002,inner sep=0pt,minimum size=14pt,above=0pt] {$e_{10}$};
\draw (1,0.3333333333333333) node[scale=1.2000000000000002,inner sep=0pt,shape=circle,minimum size=3pt,fill=black] {};
\draw (1,0.3333333333333333) node[scale=1.2000000000000002,inner sep=0pt,below=9.895453501482391pt] {$A^{-}$};
\draw (1,0.6666666666666666) node[scale=1.2000000000000002,inner sep=0pt,shape=circle,minimum size=3pt,fill=black] {};
\draw (1,0.6666666666666666) node[scale=1.2000000000000002,inner sep=0pt,above=14.895453501482391pt,right=-4pt] {$A^{+}$};
\draw[] (1,0.3333333333333333) -- (0.5,0.16666666666666666);
\draw[] (0.5,0.16666666666666666) -- (0.5,0);
\draw[] (1,0.6666666666666666) -- (0.5,0.8333333333333334);
\draw[] (0.5,0.8333333333333334) -- (0.5,1);
\draw (0.5,0.9166666666666666) node[scale=1.2000000000000002,inner sep=0pt,shape=circle,minimum size=3pt,fill=black] {};
\draw (0.5,0.9166666666666666) node[scale=1.2000000000000002,inner sep=0pt,left=3.8pt] {$t$};
\draw (0.5,0.9166666666666666) node[scale=1.2000000000000002,inner sep=0pt,right=10pt,above=-3.3pt] {$t^{\unaryminus 1}$};
\draw[] (1,0.3333333333333333) -- (1.5,0.16666666666666666);
\draw[] (1.5,0.16666666666666666) -- (1.5,0);
\draw[] (1,0.6666666666666666) -- (1.5,0.8333333333333334);
\draw[] (1.5,0.8333333333333334) -- (1.5,1);
\draw (1.5,0.9166666666666666) node[scale=1.2000000000000002,inner sep=0pt,shape=circle,minimum size=3pt,fill=black] {};
\draw (1.5,0.9166666666666666) node[scale=1.2000000000000002,inner sep=0pt,left=3.8pt] {$t$};
\draw (1.5,0.9166666666666666) node[scale=1.2000000000000002,inner sep=0pt,right=10pt,above=-3.3pt] {$t^{\unaryminus 1}$};
\draw[] (1,0.3333333333333333) -- (2.5,0.16666666666666666);
\draw[] (2.5,0.16666666666666666) -- (2.5,0);
\draw[] (1,0.6666666666666666) -- (2.5,0.8333333333333334);
\draw[] (2.5,0.8333333333333334) -- (2.5,1);
\draw (2.5,0.9166666666666666) node[scale=1.2000000000000002,inner sep=0pt,shape=circle,minimum size=3pt,fill=black] {};
\draw (2.5,0.9166666666666666) node[scale=1.2000000000000002,inner sep=0pt,left=3.8pt] {$t$};
\draw (2.5,0.9166666666666666) node[scale=1.2000000000000002,inner sep=0pt,right=10pt,above=-3.3pt] {$t^{\unaryminus 1}$};
\draw[] (1,0.3333333333333333) -- (3.5,0.16666666666666666);
\draw[] (3.5,0.16666666666666666) -- (3.5,0);
\draw[] (1,0.6666666666666666) -- (3.5,0.8333333333333334);
\draw[] (3.5,0.8333333333333334) -- (3.5,1);
\draw (3.5,0.9166666666666666) node[scale=1.2000000000000002,inner sep=0pt,shape=circle,minimum size=3pt,fill=black] {};
\draw (3.5,0.9166666666666666) node[scale=1.2000000000000002,inner sep=0pt,left=3.8pt] {$t$};
\draw (3.5,0.9166666666666666) node[scale=1.2000000000000002,inner sep=0pt,right=10pt,above=-3.3pt] {$t^{\unaryminus 1}$};
\draw (7,0.3333333333333333) node[scale=1.2000000000000002,inner sep=0pt,shape=circle,minimum size=3pt,fill=black] {};
\draw (7,0.3333333333333333) node[scale=1.2000000000000002,inner sep=0pt,below=9.895453501482391pt] {$A^{-}$};
\draw (7,0.6666666666666666) node[scale=1.2000000000000002,inner sep=0pt,shape=circle,minimum size=3pt,fill=black] {};
\draw (7,0.6666666666666666) node[scale=1.2000000000000002,inner sep=0pt,above=14.895453501482391pt,right=-4pt] {$A^{+}$};
\draw[] (7,0.3333333333333333) -- (6.5,0.16666666666666666);
\draw[] (6.5,0.16666666666666666) -- (6.5,0);
\draw[] (7,0.6666666666666666) -- (6.5,0.8333333333333334);
\draw[] (6.5,0.8333333333333334) -- (6.5,1);
\draw (6.5,0.9166666666666666) node[scale=1.2000000000000002,inner sep=0pt,shape=circle,minimum size=3pt,fill=black] {};
\draw (6.5,0.9166666666666666) node[scale=1.2000000000000002,inner sep=0pt,left=3.8pt] {$t$};
\draw (6.5,0.9166666666666666) node[scale=1.2000000000000002,inner sep=0pt,right=10pt,above=-3.3pt] {$t^{\unaryminus 1}$};
\draw[] (7,0.3333333333333333) -- (7.5,0.16666666666666666);
\draw[] (7.5,0.16666666666666666) -- (7.5,0);
\draw[] (7,0.6666666666666666) -- (7.5,0.8333333333333334);
\draw[] (7.5,0.8333333333333334) -- (7.5,1);
\draw (7.5,0.9166666666666666) node[scale=1.2000000000000002,inner sep=0pt,shape=circle,minimum size=3pt,fill=black] {};
\draw (7.5,0.9166666666666666) node[scale=1.2000000000000002,inner sep=0pt,left=3.8pt] {$t$};
\draw (7.5,0.9166666666666666) node[scale=1.2000000000000002,inner sep=0pt,right=10pt,above=-3.3pt] {$t^{\unaryminus 1}$};
\draw[] (7,0.3333333333333333) -- (8.5,0.16666666666666666);
\draw[] (8.5,0.16666666666666666) -- (8.5,0);
\draw[] (7,0.6666666666666666) -- (8.5,0.8333333333333334);
\draw[] (8.5,0.8333333333333334) -- (8.5,1);
\draw (8.5,0.9166666666666666) node[scale=1.2000000000000002,inner sep=0pt,shape=circle,minimum size=3pt,fill=black] {};
\draw (8.5,0.9166666666666666) node[scale=1.2000000000000002,inner sep=0pt,left=3.8pt] {$t$};
\draw (8.5,0.9166666666666666) node[scale=1.2000000000000002,inner sep=0pt,right=10pt,above=-3.3pt] {$t^{\unaryminus 1}$};
\draw[] (7,0.3333333333333333) -- (9.5,0.16666666666666666);
\draw[] (9.5,0.16666666666666666) -- (9.5,0);
\draw[] (7,0.6666666666666666) -- (9.5,0.8333333333333334);
\draw[] (9.5,0.8333333333333334) -- (9.5,1);
\draw (9.5,0.9166666666666666) node[scale=1.2000000000000002,inner sep=0pt,shape=circle,minimum size=3pt,fill=black] {};
\draw (9.5,0.9166666666666666) node[scale=1.2000000000000002,inner sep=0pt,left=3.8pt] {$t$};
\draw (9.5,0.9166666666666666) node[scale=1.2000000000000002,inner sep=0pt,right=10pt,above=-3.3pt] {$t^{\unaryminus 1}$};
\draw (5,0.3333333333333333) node[scale=1.2000000000000002,inner sep=0pt,shape=circle,minimum size=3pt,fill=black] {};
\draw (5,0.3333333333333333) node[scale=1.2000000000000002,inner sep=0pt,below=9.895453501482391pt] {$A^{+}$};
\draw (5,0.6666666666666666) node[scale=1.2000000000000002,inner sep=0pt,shape=circle,minimum size=3pt,fill=black] {};
\draw (5,0.6666666666666666) node[scale=1.2000000000000002,inner sep=0pt,above=14.895453501482391pt,right=-4pt] {$A^{-}$};
\draw[] (5,0.3333333333333333) -- (5,0.6666666666666666);
\draw (5,0.5) node[scale=1.2000000000000002,inner sep=0pt,shape=circle,minimum size=3pt,fill=black] {};
\draw (5,0.5) node[scale=1.2000000000000002,inner sep=0pt,left=3.8pt] {$t$};
\draw (5,0.5) node[scale=1.2000000000000002,inner sep=0pt,right=10pt,above=-3.3pt] {$t^{\unaryminus 1}$};
\draw[] (5,0.6666666666666666) -- (4.5,0.8333333333333334);
\draw[] (4.5,0.8333333333333334) -- (4.5,1);
\draw[] (5,0.3333333333333333) -- (4.5,0.16666666666666666);
\draw[] (4.5,0.16666666666666666) -- (4.5,0);
\draw (4.5,0.9166666666666666) node[scale=1.2000000000000002,inner sep=0pt,shape=circle,minimum size=3pt,fill=black] {};
\draw (4.5,0.9166666666666666) node[scale=1.2000000000000002,inner sep=0pt,right=3.8pt] {$t$};
\draw (4.5,0.9166666666666666) node[scale=1.2000000000000002,inner sep=0pt,left=8pt,above=-3.3pt] {$t^{\unaryminus 1}$};
\draw[] (5,0.6666666666666666) -- (5.5,0.8333333333333334);
\draw[] (5.5,0.8333333333333334) -- (5.5,1);
\draw[] (5,0.3333333333333333) -- (5.5,0.16666666666666666);
\draw[] (5.5,0.16666666666666666) -- (5.5,0);
\draw (5.5,0.9166666666666666) node[scale=1.2000000000000002,inner sep=0pt,shape=circle,minimum size=3pt,fill=black] {};
\draw (5.5,0.9166666666666666) node[scale=1.2000000000000002,inner sep=0pt,right=3.8pt] {$t$};
\draw (5.5,0.9166666666666666) node[scale=1.2000000000000002,inner sep=0pt,left=8pt,above=-3.3pt] {$t^{\unaryminus 1}$};
\draw[] (1,0.6666666666666666) -- (5,0.6666666666666666);
\draw[] (1,0.3333333333333333) -- (5,0.3333333333333333);
\draw[] (7,0.6666666666666666) -- (5,0.6666666666666666);
\draw[] (7,0.3333333333333333) -- (5,0.3333333333333333);
\draw (3.0,0.6666666666666666) node[scale=1.2000000000000002,inner sep=0pt,shape=circle,minimum size=3pt,fill=black] {};
\draw (3.0,0.6666666666666666) node[scale=1.2000000000000002,inner sep=0pt,above=3pt] {$t$};
\draw (3.0,0.6666666666666666) node[scale=1.2000000000000002,inner sep=0pt,below=5pt,right=-2.3pt] {$t^{\unaryminus 1}$};
\draw (3.0,0.3333333333333333) node[scale=1.2000000000000002,inner sep=0pt,shape=circle,minimum size=3pt,fill=black] {};
\draw (3.0,0.3333333333333333) node[scale=1.2000000000000002,inner sep=0pt,below=3pt] {$t$};
\draw (3.0,0.3333333333333333) node[scale=1.2000000000000002,inner sep=0pt,above=7pt,right=-2.3pt] {$t^{\unaryminus 1}$};
\draw (6.0,0.3333333333333333) node[scale=1.2000000000000002,inner sep=0pt,shape=circle,minimum size=3pt,fill=black] {};
\draw (6.0,0.3333333333333333) node[scale=1.2000000000000002,inner sep=0pt,above=3pt] {$t$};
\draw (6.0,0.3333333333333333) node[scale=1.2000000000000002,inner sep=0pt,below=5pt,right=-2.3pt] {$t^{\unaryminus 1}$};
\draw (6.0,0.6666666666666666) node[scale=1.2000000000000002,inner sep=0pt,shape=circle,minimum size=3pt,fill=black] {};
\draw (6.0,0.6666666666666666) node[scale=1.2000000000000002,inner sep=0pt,below=3pt] {$t$};
\draw (6.0,0.6666666666666666) node[scale=1.2000000000000002,inner sep=0pt,above=7pt,right=-2.3pt] {$t^{\unaryminus 1}$};
}


\begin{scope}[shift={(0,0)}]
    \DrN
\end{scope}

\end{tikzpicture}

%% file: images/D_25_5.tikz
\begin{tikzpicture}[xscale=1.5,yscale=6]
\providecommand{\unaryminus}{\scalebox{0.5}[0.75]{\( - \)}}
\fontsize{10}{12}\selectfont
\newcommand*{\DrN}{%
\draw[thick] (0,1) -- (16,1);
\draw[thick] (16,0) -- (0,0);
\draw[dashed, thick] (0,0) -- (0,1);
\draw[dashed, thick] (16,1) -- (16,0);
\draw (0,0) node[scale=1.2000000000000002,inner sep=0pt,shape=isosceles triangle,minimum size=3pt,fill=black] {};
\draw (0,1) node[scale=1.2000000000000002,inner sep=0pt,shape=isosceles triangle,minimum size=3pt,fill=black] {};
\draw (1,0) node[scale=1.2000000000000002,inner sep=0pt,shape=isosceles triangle,minimum size=3pt,fill=black] {};
\draw (1,1) node[scale=1.2000000000000002,inner sep=0pt,shape=isosceles triangle,minimum size=3pt,fill=black] {};
\draw (2,0) node[scale=1.2000000000000002,inner sep=0pt,shape=isosceles triangle,minimum size=3pt,fill=black] {};
\draw (2,1) node[scale=1.2000000000000002,inner sep=0pt,shape=isosceles triangle,minimum size=3pt,fill=black] {};
\draw (3,0) node[scale=1.2000000000000002,inner sep=0pt,shape=isosceles triangle,minimum size=3pt,fill=black] {};
\draw (3,1) node[scale=1.2000000000000002,inner sep=0pt,shape=isosceles triangle,minimum size=3pt,fill=black] {};
\draw (4,0) node[scale=1.2000000000000002,inner sep=0pt,shape=isosceles triangle,minimum size=3pt,fill=black] {};
\draw (4,1) node[scale=1.2000000000000002,inner sep=0pt,shape=isosceles triangle,minimum size=3pt,fill=black] {};
\draw (5,0) node[scale=1.2000000000000002,inner sep=0pt,shape=isosceles triangle,minimum size=3pt,fill=black] {};
\draw (5,1) node[scale=1.2000000000000002,inner sep=0pt,shape=isosceles triangle,minimum size=3pt,fill=black] {};
\draw (6,0) node[scale=1.2000000000000002,inner sep=0pt,shape=isosceles triangle,minimum size=3pt,fill=black] {};
\draw (6,1) node[scale=1.2000000000000002,inner sep=0pt,shape=isosceles triangle,minimum size=3pt,fill=black] {};
\draw (7,0) node[scale=1.2000000000000002,inner sep=0pt,shape=isosceles triangle,minimum size=3pt,fill=black] {};
\draw (7,1) node[scale=1.2000000000000002,inner sep=0pt,shape=isosceles triangle,minimum size=3pt,fill=black] {};
\draw (8,0) node[scale=1.2000000000000002,inner sep=0pt,shape=isosceles triangle,minimum size=3pt,fill=black] {};
\draw (8,1) node[scale=1.2000000000000002,inner sep=0pt,shape=isosceles triangle,minimum size=3pt,fill=black] {};
\draw (9,0) node[scale=1.2000000000000002,inner sep=0pt,shape=isosceles triangle,minimum size=3pt,fill=black] {};
\draw (9,1) node[scale=1.2000000000000002,inner sep=0pt,shape=isosceles triangle,minimum size=3pt,fill=black] {};
\draw (10,0) node[scale=1.2000000000000002,inner sep=0pt,shape=isosceles triangle,minimum size=3pt,fill=black] {};
\draw (10,1) node[scale=1.2000000000000002,inner sep=0pt,shape=isosceles triangle,minimum size=3pt,fill=black] {};
\draw (11,0) node[scale=1.2000000000000002,inner sep=0pt,shape=isosceles triangle,minimum size=3pt,fill=black] {};
\draw (11,1) node[scale=1.2000000000000002,inner sep=0pt,shape=isosceles triangle,minimum size=3pt,fill=black] {};
\draw (12,0) node[scale=1.2000000000000002,inner sep=0pt,shape=isosceles triangle,minimum size=3pt,fill=black] {};
\draw (12,1) node[scale=1.2000000000000002,inner sep=0pt,shape=isosceles triangle,minimum size=3pt,fill=black] {};
\draw (13,0) node[scale=1.2000000000000002,inner sep=0pt,shape=isosceles triangle,minimum size=3pt,fill=black] {};
\draw (13,1) node[scale=1.2000000000000002,inner sep=0pt,shape=isosceles triangle,minimum size=3pt,fill=black] {};
\draw (14,0) node[scale=1.2000000000000002,inner sep=0pt,shape=isosceles triangle,minimum size=3pt,fill=black] {};
\draw (14,1) node[scale=1.2000000000000002,inner sep=0pt,shape=isosceles triangle,minimum size=3pt,fill=black] {};
\draw (15,0) node[scale=1.2000000000000002,inner sep=0pt,shape=isosceles triangle,minimum size=3pt,fill=black] {};
\draw (15,1) node[scale=1.2000000000000002,inner sep=0pt,shape=isosceles triangle,minimum size=3pt,fill=black] {};
\draw (16,0) node[scale=1.2000000000000002,inner sep=0pt,shape=isosceles triangle,minimum size=3pt,fill=black] {};
\draw (16,1) node[scale=1.2000000000000002,inner sep=0pt,shape=isosceles triangle,minimum size=3pt,fill=black] {};
\draw (0.5,0) node[scale=1.2000000000000002,inner sep=0pt,minimum size=14pt,below=0pt] {$e_{2}$};
\draw (0.5,1) node[scale=1.2000000000000002,inner sep=0pt,minimum size=14pt,above=0pt] {$e_{1}$};
\draw (1.5,0) node[scale=1.2000000000000002,inner sep=0pt,minimum size=14pt,below=0pt] {$e_{1}$};
\draw (1.5,1) node[scale=1.2000000000000002,inner sep=0pt,minimum size=14pt,above=0pt] {$e_{2}$};
\draw (2.5,0) node[scale=1.2000000000000002,inner sep=0pt,minimum size=14pt,below=0pt] {$e_{4}$};
\draw (2.5,1) node[scale=1.2000000000000002,inner sep=0pt,minimum size=14pt,above=0pt] {$e_{3}$};
\draw (3.5,0) node[scale=1.2000000000000002,inner sep=0pt,minimum size=14pt,below=0pt] {$e_{3}$};
\draw (3.5,1) node[scale=1.2000000000000002,inner sep=0pt,minimum size=14pt,above=0pt] {$e_{4}$};
\draw (4.5,0) node[scale=1.2000000000000002,inner sep=0pt,minimum size=14pt,below=0pt] {$e_{6}$};
\draw (4.5,1) node[scale=1.2000000000000002,inner sep=0pt,minimum size=14pt,above=0pt] {$e_{5}$};
\draw (5.5,0) node[scale=1.2000000000000002,inner sep=0pt,minimum size=14pt,below=0pt] {$e_{5}$};
\draw (5.5,1) node[scale=1.2000000000000002,inner sep=0pt,minimum size=14pt,above=0pt] {$e_{6}$};
\draw (6.5,0) node[scale=1.2000000000000002,inner sep=0pt,minimum size=14pt,below=0pt] {$e_{8}$};
\draw (6.5,1) node[scale=1.2000000000000002,inner sep=0pt,minimum size=14pt,above=0pt] {$e_{7}$};
\draw (7.5,0) node[scale=1.2000000000000002,inner sep=0pt,minimum size=14pt,below=0pt] {$e_{7}$};
\draw (7.5,1) node[scale=1.2000000000000002,inner sep=0pt,minimum size=14pt,above=0pt] {$e_{8}$};
\draw (8.5,0) node[scale=1.2000000000000002,inner sep=0pt,minimum size=14pt,below=0pt] {$e_{10}$};
\draw (8.5,1) node[scale=1.2000000000000002,inner sep=0pt,minimum size=14pt,above=0pt] {$e_{9}$};
\draw (9.5,0) node[scale=1.2000000000000002,inner sep=0pt,minimum size=14pt,below=0pt] {$e_{9}$};
\draw (9.5,1) node[scale=1.2000000000000002,inner sep=0pt,minimum size=14pt,above=0pt] {$e_{10}$};
\draw (10.5,0) node[scale=1.2000000000000002,inner sep=0pt,minimum size=14pt,below=0pt] {$e_{12}$};
\draw (10.5,1) node[scale=1.2000000000000002,inner sep=0pt,minimum size=14pt,above=0pt] {$e_{11}$};
\draw (11.5,0) node[scale=1.2000000000000002,inner sep=0pt,minimum size=14pt,below=0pt] {$e_{11}$};
\draw (11.5,1) node[scale=1.2000000000000002,inner sep=0pt,minimum size=14pt,above=0pt] {$e_{12}$};
\draw (12.5,0) node[scale=1.2000000000000002,inner sep=0pt,minimum size=14pt,below=0pt] {$e_{14}$};
\draw (12.5,1) node[scale=1.2000000000000002,inner sep=0pt,minimum size=14pt,above=0pt] {$e_{13}$};
\draw (13.5,0) node[scale=1.2000000000000002,inner sep=0pt,minimum size=14pt,below=0pt] {$e_{13}$};
\draw (13.5,1) node[scale=1.2000000000000002,inner sep=0pt,minimum size=14pt,above=0pt] {$e_{14}$};
\draw (14.5,0) node[scale=1.2000000000000002,inner sep=0pt,minimum size=14pt,below=0pt] {$e_{16}$};
\draw (14.5,1) node[scale=1.2000000000000002,inner sep=0pt,minimum size=14pt,above=0pt] {$e_{15}$};
\draw (15.5,0) node[scale=1.2000000000000002,inner sep=0pt,minimum size=14pt,below=0pt] {$e_{15}$};
\draw (15.5,1) node[scale=1.2000000000000002,inner sep=0pt,minimum size=14pt,above=0pt] {$e_{16}$};
\draw (1,0.3333333333333333) node[scale=1.2000000000000002,inner sep=0pt,shape=circle,minimum size=3pt,fill=black] {};
\draw (1,0.3333333333333333) node[scale=1.2000000000000002,inner sep=0pt,below=9.895453501482391pt] {$A^{-}$};
\draw (1,0.6666666666666666) node[scale=1.2000000000000002,inner sep=0pt,shape=circle,minimum size=3pt,fill=black] {};
\draw (1,0.6666666666666666) node[scale=1.2000000000000002,inner sep=0pt,above=14.895453501482391pt,right=-4pt] {$A^{+}$};
\draw[] (1,0.3333333333333333) -- (0.5,0.16666666666666666);
\draw[] (0.5,0.16666666666666666) -- (0.5,0);
\draw[] (1,0.6666666666666666) -- (0.5,0.8333333333333334);
\draw[] (0.5,0.8333333333333334) -- (0.5,1);
\draw (0.5,0.9166666666666666) node[scale=1.2000000000000002,inner sep=0pt,shape=circle,minimum size=3pt,fill=black] {};
\draw (0.5,0.9166666666666666) node[scale=1.2000000000000002,inner sep=0pt,left=3.8pt] {$t$};
\draw (0.5,0.9166666666666666) node[scale=1.2000000000000002,inner sep=0pt,right=10pt,above=-3.3pt] {$t^{\unaryminus 1}$};
\draw[] (1,0.3333333333333333) -- (1.5,0.16666666666666666);
\draw[] (1.5,0.16666666666666666) -- (1.5,0);
\draw[] (1,0.6666666666666666) -- (1.5,0.8333333333333334);
\draw[] (1.5,0.8333333333333334) -- (1.5,1);
\draw (1.5,0.9166666666666666) node[scale=1.2000000000000002,inner sep=0pt,shape=circle,minimum size=3pt,fill=black] {};
\draw (1.5,0.9166666666666666) node[scale=1.2000000000000002,inner sep=0pt,left=3.8pt] {$t$};
\draw (1.5,0.9166666666666666) node[scale=1.2000000000000002,inner sep=0pt,right=10pt,above=-3.3pt] {$t^{\unaryminus 1}$};
\draw[] (1,0.3333333333333333) -- (2.5,0.16666666666666666);
\draw[] (2.5,0.16666666666666666) -- (2.5,0);
\draw[] (1,0.6666666666666666) -- (2.5,0.8333333333333334);
\draw[] (2.5,0.8333333333333334) -- (2.5,1);
\draw (2.5,0.9166666666666666) node[scale=1.2000000000000002,inner sep=0pt,shape=circle,minimum size=3pt,fill=black] {};
\draw (2.5,0.9166666666666666) node[scale=1.2000000000000002,inner sep=0pt,left=3.8pt] {$t$};
\draw (2.5,0.9166666666666666) node[scale=1.2000000000000002,inner sep=0pt,right=10pt,above=-3.3pt] {$t^{\unaryminus 1}$};
\draw[] (1,0.3333333333333333) -- (3.5,0.16666666666666666);
\draw[] (3.5,0.16666666666666666) -- (3.5,0);
\draw[] (1,0.6666666666666666) -- (3.5,0.8333333333333334);
\draw[] (3.5,0.8333333333333334) -- (3.5,1);
\draw (3.5,0.9166666666666666) node[scale=1.2000000000000002,inner sep=0pt,shape=circle,minimum size=3pt,fill=black] {};
\draw (3.5,0.9166666666666666) node[scale=1.2000000000000002,inner sep=0pt,left=3.8pt] {$t$};
\draw (3.5,0.9166666666666666) node[scale=1.2000000000000002,inner sep=0pt,right=10pt,above=-3.3pt] {$t^{\unaryminus 1}$};
\draw (7,0.3333333333333333) node[scale=1.2000000000000002,inner sep=0pt,shape=circle,minimum size=3pt,fill=black] {};
\draw (7,0.3333333333333333) node[scale=1.2000000000000002,inner sep=0pt,below=9.895453501482391pt] {$A^{-}$};
\draw (7,0.6666666666666666) node[scale=1.2000000000000002,inner sep=0pt,shape=circle,minimum size=3pt,fill=black] {};
\draw (7,0.6666666666666666) node[scale=1.2000000000000002,inner sep=0pt,above=14.895453501482391pt,right=-4pt] {$A^{+}$};
\draw[] (7,0.3333333333333333) -- (6.5,0.16666666666666666);
\draw[] (6.5,0.16666666666666666) -- (6.5,0);
\draw[] (7,0.6666666666666666) -- (6.5,0.8333333333333334);
\draw[] (6.5,0.8333333333333334) -- (6.5,1);
\draw (6.5,0.9166666666666666) node[scale=1.2000000000000002,inner sep=0pt,shape=circle,minimum size=3pt,fill=black] {};
\draw (6.5,0.9166666666666666) node[scale=1.2000000000000002,inner sep=0pt,left=3.8pt] {$t$};
\draw (6.5,0.9166666666666666) node[scale=1.2000000000000002,inner sep=0pt,right=10pt,above=-3.3pt] {$t^{\unaryminus 1}$};
\draw[] (7,0.3333333333333333) -- (7.5,0.16666666666666666);
\draw[] (7.5,0.16666666666666666) -- (7.5,0);
\draw[] (7,0.6666666666666666) -- (7.5,0.8333333333333334);
\draw[] (7.5,0.8333333333333334) -- (7.5,1);
\draw (7.5,0.9166666666666666) node[scale=1.2000000000000002,inner sep=0pt,shape=circle,minimum size=3pt,fill=black] {};
\draw (7.5,0.9166666666666666) node[scale=1.2000000000000002,inner sep=0pt,left=3.8pt] {$t$};
\draw (7.5,0.9166666666666666) node[scale=1.2000000000000002,inner sep=0pt,right=10pt,above=-3.3pt] {$t^{\unaryminus 1}$};
\draw[] (7,0.3333333333333333) -- (8.5,0.16666666666666666);
\draw[] (8.5,0.16666666666666666) -- (8.5,0);
\draw[] (7,0.6666666666666666) -- (8.5,0.8333333333333334);
\draw[] (8.5,0.8333333333333334) -- (8.5,1);
\draw (8.5,0.9166666666666666) node[scale=1.2000000000000002,inner sep=0pt,shape=circle,minimum size=3pt,fill=black] {};
\draw (8.5,0.9166666666666666) node[scale=1.2000000000000002,inner sep=0pt,left=3.8pt] {$t$};
\draw (8.5,0.9166666666666666) node[scale=1.2000000000000002,inner sep=0pt,right=10pt,above=-3.3pt] {$t^{\unaryminus 1}$};
\draw[] (7,0.3333333333333333) -- (9.5,0.16666666666666666);
\draw[] (9.5,0.16666666666666666) -- (9.5,0);
\draw[] (7,0.6666666666666666) -- (9.5,0.8333333333333334);
\draw[] (9.5,0.8333333333333334) -- (9.5,1);
\draw (9.5,0.9166666666666666) node[scale=1.2000000000000002,inner sep=0pt,shape=circle,minimum size=3pt,fill=black] {};
\draw (9.5,0.9166666666666666) node[scale=1.2000000000000002,inner sep=0pt,left=3.8pt] {$t$};
\draw (9.5,0.9166666666666666) node[scale=1.2000000000000002,inner sep=0pt,right=10pt,above=-3.3pt] {$t^{\unaryminus 1}$};
\draw (13,0.3333333333333333) node[scale=1.2000000000000002,inner sep=0pt,shape=circle,minimum size=3pt,fill=black] {};
\draw (13,0.3333333333333333) node[scale=1.2000000000000002,inner sep=0pt,below=9.895453501482391pt] {$A^{-}$};
\draw (13,0.6666666666666666) node[scale=1.2000000000000002,inner sep=0pt,shape=circle,minimum size=3pt,fill=black] {};
\draw (13,0.6666666666666666) node[scale=1.2000000000000002,inner sep=0pt,above=14.895453501482391pt,right=-4pt] {$A^{+}$};
\draw[] (13,0.3333333333333333) -- (12.5,0.16666666666666666);
\draw[] (12.5,0.16666666666666666) -- (12.5,0);
\draw[] (13,0.6666666666666666) -- (12.5,0.8333333333333334);
\draw[] (12.5,0.8333333333333334) -- (12.5,1);
\draw (12.5,0.9166666666666666) node[scale=1.2000000000000002,inner sep=0pt,shape=circle,minimum size=3pt,fill=black] {};
\draw (12.5,0.9166666666666666) node[scale=1.2000000000000002,inner sep=0pt,left=3.8pt] {$t$};
\draw (12.5,0.9166666666666666) node[scale=1.2000000000000002,inner sep=0pt,right=10pt,above=-3.3pt] {$t^{\unaryminus 1}$};
\draw[] (13,0.3333333333333333) -- (13.5,0.16666666666666666);
\draw[] (13.5,0.16666666666666666) -- (13.5,0);
\draw[] (13,0.6666666666666666) -- (13.5,0.8333333333333334);
\draw[] (13.5,0.8333333333333334) -- (13.5,1);
\draw (13.5,0.9166666666666666) node[scale=1.2000000000000002,inner sep=0pt,shape=circle,minimum size=3pt,fill=black] {};
\draw (13.5,0.9166666666666666) node[scale=1.2000000000000002,inner sep=0pt,left=3.8pt] {$t$};
\draw (13.5,0.9166666666666666) node[scale=1.2000000000000002,inner sep=0pt,right=10pt,above=-3.3pt] {$t^{\unaryminus 1}$};
\draw[] (13,0.3333333333333333) -- (14.5,0.16666666666666666);
\draw[] (14.5,0.16666666666666666) -- (14.5,0);
\draw[] (13,0.6666666666666666) -- (14.5,0.8333333333333334);
\draw[] (14.5,0.8333333333333334) -- (14.5,1);
\draw (14.5,0.9166666666666666) node[scale=1.2000000000000002,inner sep=0pt,shape=circle,minimum size=3pt,fill=black] {};
\draw (14.5,0.9166666666666666) node[scale=1.2000000000000002,inner sep=0pt,left=3.8pt] {$t$};
\draw (14.5,0.9166666666666666) node[scale=1.2000000000000002,inner sep=0pt,right=10pt,above=-3.3pt] {$t^{\unaryminus 1}$};
\draw[] (13,0.3333333333333333) -- (15.5,0.16666666666666666);
\draw[] (15.5,0.16666666666666666) -- (15.5,0);
\draw[] (13,0.6666666666666666) -- (15.5,0.8333333333333334);
\draw[] (15.5,0.8333333333333334) -- (15.5,1);
\draw (15.5,0.9166666666666666) node[scale=1.2000000000000002,inner sep=0pt,shape=circle,minimum size=3pt,fill=black] {};
\draw (15.5,0.9166666666666666) node[scale=1.2000000000000002,inner sep=0pt,left=3.8pt] {$t$};
\draw (15.5,0.9166666666666666) node[scale=1.2000000000000002,inner sep=0pt,right=10pt,above=-3.3pt] {$t^{\unaryminus 1}$};
\draw (5,0.3333333333333333) node[scale=1.2000000000000002,inner sep=0pt,shape=circle,minimum size=3pt,fill=black] {};
\draw (5,0.3333333333333333) node[scale=1.2000000000000002,inner sep=0pt,below=9.895453501482391pt] {$A^{+}$};
\draw (5,0.6666666666666666) node[scale=1.2000000000000002,inner sep=0pt,shape=circle,minimum size=3pt,fill=black] {};
\draw (5,0.6666666666666666) node[scale=1.2000000000000002,inner sep=0pt,above=14.895453501482391pt,right=-4pt] {$A^{-}$};
\draw[] (5,0.3333333333333333) -- (5,0.6666666666666666);
\draw (5,0.5) node[scale=1.2000000000000002,inner sep=0pt,shape=circle,minimum size=3pt,fill=black] {};
\draw (5,0.5) node[scale=1.2000000000000002,inner sep=0pt,left=3.8pt] {$t$};
\draw (5,0.5) node[scale=1.2000000000000002,inner sep=0pt,right=10pt,above=-3.3pt] {$t^{\unaryminus 1}$};
\draw[] (5,0.6666666666666666) -- (4.5,0.8333333333333334);
\draw[] (4.5,0.8333333333333334) -- (4.5,1);
\draw[] (5,0.3333333333333333) -- (4.5,0.16666666666666666);
\draw[] (4.5,0.16666666666666666) -- (4.5,0);
\draw (4.5,0.9166666666666666) node[scale=1.2000000000000002,inner sep=0pt,shape=circle,minimum size=3pt,fill=black] {};
\draw (4.5,0.9166666666666666) node[scale=1.2000000000000002,inner sep=0pt,right=3.8pt] {$t$};
\draw (4.5,0.9166666666666666) node[scale=1.2000000000000002,inner sep=0pt,left=8pt,above=-3.3pt] {$t^{\unaryminus 1}$};
\draw[] (5,0.6666666666666666) -- (5.5,0.8333333333333334);
\draw[] (5.5,0.8333333333333334) -- (5.5,1);
\draw[] (5,0.3333333333333333) -- (5.5,0.16666666666666666);
\draw[] (5.5,0.16666666666666666) -- (5.5,0);
\draw (5.5,0.9166666666666666) node[scale=1.2000000000000002,inner sep=0pt,shape=circle,minimum size=3pt,fill=black] {};
\draw (5.5,0.9166666666666666) node[scale=1.2000000000000002,inner sep=0pt,right=3.8pt] {$t$};
\draw (5.5,0.9166666666666666) node[scale=1.2000000000000002,inner sep=0pt,left=8pt,above=-3.3pt] {$t^{\unaryminus 1}$};
\draw (11,0.3333333333333333) node[scale=1.2000000000000002,inner sep=0pt,shape=circle,minimum size=3pt,fill=black] {};
\draw (11,0.3333333333333333) node[scale=1.2000000000000002,inner sep=0pt,below=9.895453501482391pt] {$A^{+}$};
\draw (11,0.6666666666666666) node[scale=1.2000000000000002,inner sep=0pt,shape=circle,minimum size=3pt,fill=black] {};
\draw (11,0.6666666666666666) node[scale=1.2000000000000002,inner sep=0pt,above=14.895453501482391pt,right=-4pt] {$A^{-}$};
\draw[] (11,0.3333333333333333) -- (11,0.6666666666666666);
\draw (11,0.5) node[scale=1.2000000000000002,inner sep=0pt,shape=circle,minimum size=3pt,fill=black] {};
\draw (11,0.5) node[scale=1.2000000000000002,inner sep=0pt,left=3.8pt] {$t$};
\draw (11,0.5) node[scale=1.2000000000000002,inner sep=0pt,right=10pt,above=-3.3pt] {$t^{\unaryminus 1}$};
\draw[] (11,0.6666666666666666) -- (10.5,0.8333333333333334);
\draw[] (10.5,0.8333333333333334) -- (10.5,1);
\draw[] (11,0.3333333333333333) -- (10.5,0.16666666666666666);
\draw[] (10.5,0.16666666666666666) -- (10.5,0);
\draw (10.5,0.9166666666666666) node[scale=1.2000000000000002,inner sep=0pt,shape=circle,minimum size=3pt,fill=black] {};
\draw (10.5,0.9166666666666666) node[scale=1.2000000000000002,inner sep=0pt,right=3.8pt] {$t$};
\draw (10.5,0.9166666666666666) node[scale=1.2000000000000002,inner sep=0pt,left=8pt,above=-3.3pt] {$t^{\unaryminus 1}$};
\draw[] (11,0.6666666666666666) -- (11.5,0.8333333333333334);
\draw[] (11.5,0.8333333333333334) -- (11.5,1);
\draw[] (11,0.3333333333333333) -- (11.5,0.16666666666666666);
\draw[] (11.5,0.16666666666666666) -- (11.5,0);
\draw (11.5,0.9166666666666666) node[scale=1.2000000000000002,inner sep=0pt,shape=circle,minimum size=3pt,fill=black] {};
\draw (11.5,0.9166666666666666) node[scale=1.2000000000000002,inner sep=0pt,right=3.8pt] {$t$};
\draw (11.5,0.9166666666666666) node[scale=1.2000000000000002,inner sep=0pt,left=8pt,above=-3.3pt] {$t^{\unaryminus 1}$};
\draw[] (1,0.6666666666666666) -- (5,0.6666666666666666);
\draw[] (1,0.3333333333333333) -- (5,0.3333333333333333);
\draw[] (13,0.6666666666666666) -- (11,0.6666666666666666);
\draw[] (13,0.3333333333333333) -- (11,0.3333333333333333);
\draw (3.0,0.6666666666666666) node[scale=1.2000000000000002,inner sep=0pt,shape=circle,minimum size=3pt,fill=black] {};
\draw (3.0,0.6666666666666666) node[scale=1.2000000000000002,inner sep=0pt,above=3pt] {$t$};
\draw (3.0,0.6666666666666666) node[scale=1.2000000000000002,inner sep=0pt,below=5pt,right=-2.3pt] {$t^{\unaryminus 1}$};
\draw (3.0,0.3333333333333333) node[scale=1.2000000000000002,inner sep=0pt,shape=circle,minimum size=3pt,fill=black] {};
\draw (3.0,0.3333333333333333) node[scale=1.2000000000000002,inner sep=0pt,below=3pt] {$t$};
\draw (3.0,0.3333333333333333) node[scale=1.2000000000000002,inner sep=0pt,above=7pt,right=-2.3pt] {$t^{\unaryminus 1}$};
\draw (12.0,0.3333333333333333) node[scale=1.2000000000000002,inner sep=0pt,shape=circle,minimum size=3pt,fill=black] {};
\draw (12.0,0.3333333333333333) node[scale=1.2000000000000002,inner sep=0pt,above=3pt] {$t$};
\draw (12.0,0.3333333333333333) node[scale=1.2000000000000002,inner sep=0pt,below=5pt,right=-2.3pt] {$t^{\unaryminus 1}$};
\draw (12.0,0.6666666666666666) node[scale=1.2000000000000002,inner sep=0pt,shape=circle,minimum size=3pt,fill=black] {};
\draw (12.0,0.6666666666666666) node[scale=1.2000000000000002,inner sep=0pt,below=3pt] {$t$};
\draw (12.0,0.6666666666666666) node[scale=1.2000000000000002,inner sep=0pt,above=7pt,right=-2.3pt] {$t^{\unaryminus 1}$};
\draw[] (5,0.6666666666666666) -- (7,0.6666666666666666);
\draw[] (7,0.3333333333333333) -- (11,0.3333333333333333);
\draw[] (5,0.3333333333333333) -- (11,0.6666666666666666);
\draw (6.0,0.6666666666666666) node[scale=1.2000000000000002,inner sep=0pt,shape=circle,minimum size=3pt,fill=black] {};
\draw (6.0,0.6666666666666666) node[scale=1.2000000000000002,inner sep=0pt,below=3pt] {$t$};
\draw (6.0,0.6666666666666666) node[scale=1.2000000000000002,inner sep=0pt,above=7pt,right=-2.3pt] {$t^{\unaryminus 1}$};
\draw (9.0,0.3333333333333333) node[scale=1.2000000000000002,inner sep=0pt,shape=circle,minimum size=3pt,fill=black] {};
\draw (9.0,0.3333333333333333) node[scale=1.2000000000000002,inner sep=0pt,below=3pt] {$t$};
\draw (9.0,0.3333333333333333) node[scale=1.2000000000000002,inner sep=0pt,above=7pt,right=-2.3pt] {$t^{\unaryminus 1}$};
\draw (8,0.5) node[scale=1.2000000000000002,inner sep=0pt,shape=circle,minimum size=3pt,fill=black] {};
\draw (8,0.5) node[scale=1.2000000000000002,inner sep=0pt,above=3pt] {$t$};
\draw (8,0.5) node[scale=1.2000000000000002,inner sep=0pt,below=5pt,right=-2.3pt] {$t^{\unaryminus 1}$};
}


\begin{scope}[shift={(0,0)}]
    \DrN
\end{scope}

\end{tikzpicture}

%% file: images/D_1_1_minus.tikz
\begin{tikzpicture}[xscale=3.0,yscale=6]
\providecommand{\unaryminus}{\scalebox{0.5}[0.75]{\( - \)}}
\fontsize{10}{12}\selectfont
\newcommand*{\DrN}{%
\draw[dashed, thick] (1,0) .. controls(0.25, 0.5) .. (1,1);
\draw[dashed, thick] (1,0) .. controls(1.75, 0.5) .. (1,1);
\draw (1,0) node[scale=1.2000000000000002,inner sep=0pt,shape=isosceles triangle,minimum size=3pt,fill=black] {};
\draw (1,1) node[scale=1.2000000000000002,inner sep=0pt,shape=isosceles triangle,minimum size=3pt,fill=black] {};
\draw (1,0.3333333333333333) node[scale=1.2000000000000002,inner sep=0pt,shape=circle,minimum size=3pt,fill=black] {};
\draw (1,0.3333333333333333) node[scale=1.2000000000000002,inner sep=0pt,below=5.3665631459994945pt] {$A^{+}$};
\draw (1,0.6666666666666666) node[scale=1.2000000000000002,inner sep=0pt,shape=circle,minimum size=3pt,fill=black] {};
\draw (1,0.6666666666666666) node[scale=1.2000000000000002,inner sep=0pt,above=10.366563145999494pt,right=-4pt] {$A^{-}$};
 (1.5,0.8333333333333334);
\draw[] (1,0.3333333333333333) -- (1,0.6666666666666666);
\draw (1,0.5) node[scale=1.2000000000000002,inner sep=0pt,shape=circle,minimum size=3pt,fill=black] {};
\draw (1,0.5) node[scale=1.2000000000000002,inner sep=0pt,left=3.8pt] {$t$};
\draw (1,0.5) node[scale=1.2000000000000002,inner sep=0pt,right=10pt,above=-3.3pt] {$t^{\unaryminus 1}$};
\draw (1,0) node[scale=1.2000000000000002,inner sep=0pt,minimum size=14pt,below=0pt] {};
\draw (1,1) node[scale=1.2000000000000002,inner sep=0pt,minimum size=14pt,above=0pt] {};
}

\begin{scope}[shift={(0,0)}]
    \DrN
\end{scope}

\end{tikzpicture}

%% file: images/D_15_3.tikz
\begin{tikzpicture}[xscale=2.0,yscale=6]
\providecommand{\unaryminus}{\scalebox{0.5}[0.75]{\( - \)}}
\fontsize{10}{12}\selectfont
\newcommand*{\DrN}{%
\draw[thick] (0,1) -- (6,1);
\draw[thick] (6,0) -- (0,0);
\draw[dashed, thick] (0,0) -- (0,1);
\draw[dashed, thick] (6,1) -- (6,0);
\draw (0,0) node[scale=1.2000000000000002,inner sep=0pt,shape=isosceles triangle,minimum size=3pt,fill=black] {};
\draw (0,1) node[scale=1.2000000000000002,inner sep=0pt,shape=isosceles triangle,minimum size=3pt,fill=black] {};
\draw (1,0) node[scale=1.2000000000000002,inner sep=0pt,shape=isosceles triangle,minimum size=3pt,fill=black] {};
\draw (1,1) node[scale=1.2000000000000002,inner sep=0pt,shape=isosceles triangle,minimum size=3pt,fill=black] {};
\draw (2,0) node[scale=1.2000000000000002,inner sep=0pt,shape=isosceles triangle,minimum size=3pt,fill=black] {};
\draw (2,1) node[scale=1.2000000000000002,inner sep=0pt,shape=isosceles triangle,minimum size=3pt,fill=black] {};
\draw (3,0) node[scale=1.2000000000000002,inner sep=0pt,shape=isosceles triangle,minimum size=3pt,fill=black] {};
\draw (3,1) node[scale=1.2000000000000002,inner sep=0pt,shape=isosceles triangle,minimum size=3pt,fill=black] {};
\draw (4,0) node[scale=1.2000000000000002,inner sep=0pt,shape=isosceles triangle,minimum size=3pt,fill=black] {};
\draw (4,1) node[scale=1.2000000000000002,inner sep=0pt,shape=isosceles triangle,minimum size=3pt,fill=black] {};
\draw (5,0) node[scale=1.2000000000000002,inner sep=0pt,shape=isosceles triangle,minimum size=3pt,fill=black] {};
\draw (5,1) node[scale=1.2000000000000002,inner sep=0pt,shape=isosceles triangle,minimum size=3pt,fill=black] {};
\draw (6,0) node[scale=1.2000000000000002,inner sep=0pt,shape=isosceles triangle,minimum size=3pt,fill=black] {};
\draw (6,1) node[scale=1.2000000000000002,inner sep=0pt,shape=isosceles triangle,minimum size=3pt,fill=black] {};
\draw (0.5,0) node[scale=1.2000000000000002,inner sep=0pt,minimum size=14pt,below=0pt] {$e_{2}$};
\draw (0.5,1) node[scale=1.2000000000000002,inner sep=0pt,minimum size=14pt,above=0pt] {$e_{1}$};
\draw (1.5,0) node[scale=1.2000000000000002,inner sep=0pt,minimum size=14pt,below=0pt] {$e_{1}$};
\draw (1.5,1) node[scale=1.2000000000000002,inner sep=0pt,minimum size=14pt,above=0pt] {$e_{2}$};
\draw (2.5,0) node[scale=1.2000000000000002,inner sep=0pt,minimum size=14pt,below=0pt] {$e_{4}$};
\draw (2.5,1) node[scale=1.2000000000000002,inner sep=0pt,minimum size=14pt,above=0pt] {$e_{3}$};
\draw (3.5,0) node[scale=1.2000000000000002,inner sep=0pt,minimum size=14pt,below=0pt] {$e_{3}$};
\draw (3.5,1) node[scale=1.2000000000000002,inner sep=0pt,minimum size=14pt,above=0pt] {$e_{4}$};
\draw (4.5,0) node[scale=1.2000000000000002,inner sep=0pt,minimum size=14pt,below=0pt] {$e_{6}$};
\draw (4.5,1) node[scale=1.2000000000000002,inner sep=0pt,minimum size=14pt,above=0pt] {$e_{5}$};
\draw (5.5,0) node[scale=1.2000000000000002,inner sep=0pt,minimum size=14pt,below=0pt] {$e_{5}$};
\draw (5.5,1) node[scale=1.2000000000000002,inner sep=0pt,minimum size=14pt,above=0pt] {$e_{6}$};
\draw (1,0.3333333333333333) node[scale=1.2000000000000002,inner sep=0pt,shape=circle,minimum size=3pt,fill=black] {};
\draw (1,0.3333333333333333) node[scale=1.2000000000000002,inner sep=0pt,below=7.589466384404111pt] {$A^{-}$};
\draw (1,0.6666666666666666) node[scale=1.2000000000000002,inner sep=0pt,shape=circle,minimum size=3pt,fill=black] {};
\draw (1,0.6666666666666666) node[scale=1.2000000000000002,inner sep=0pt,above=12.589466384404112pt,right=-4pt] {$A^{+}$};
\draw[] (1,0.3333333333333333) -- (0.5,0.16666666666666666);
\draw[] (0.5,0.16666666666666666) -- (0.5,0);
\draw[] (1,0.6666666666666666) -- (0.5,0.8333333333333334);
\draw[] (0.5,0.8333333333333334) -- (0.5,1);
\draw (0.5,0.9166666666666666) node[scale=1.2000000000000002,inner sep=0pt,shape=circle,minimum size=3pt,fill=black] {};
\draw (0.5,0.9166666666666666) node[scale=1.2000000000000002,inner sep=0pt,left=3.8pt] {$t$};
\draw (0.5,0.9166666666666666) node[scale=1.2000000000000002,inner sep=0pt,right=10pt,above=-3.3pt] {$t^{\unaryminus 1}$};
\draw[] (1,0.3333333333333333) -- (1.5,0.16666666666666666);
\draw[] (1.5,0.16666666666666666) -- (1.5,0);
\draw[] (1,0.6666666666666666) -- (1.5,0.8333333333333334);
\draw[] (1.5,0.8333333333333334) -- (1.5,1);
\draw (1.5,0.9166666666666666) node[scale=1.2000000000000002,inner sep=0pt,shape=circle,minimum size=3pt,fill=black] {};
\draw (1.5,0.9166666666666666) node[scale=1.2000000000000002,inner sep=0pt,left=3.8pt] {$t$};
\draw (1.5,0.9166666666666666) node[scale=1.2000000000000002,inner sep=0pt,right=10pt,above=-3.3pt] {$t^{\unaryminus 1}$};
\draw (3,0.3333333333333333) node[scale=1.2000000000000002,inner sep=0pt,shape=circle,minimum size=3pt,fill=black] {};
\draw (3,0.3333333333333333) node[scale=1.2000000000000002,inner sep=0pt,below=7.589466384404111pt] {$A^{-}$};
\draw (3,0.6666666666666666) node[scale=1.2000000000000002,inner sep=0pt,shape=circle,minimum size=3pt,fill=black] {};
\draw (3,0.6666666666666666) node[scale=1.2000000000000002,inner sep=0pt,above=12.589466384404112pt,right=-4pt] {$A^{+}$};
\draw[] (3,0.3333333333333333) -- (2.5,0.16666666666666666);
\draw[] (2.5,0.16666666666666666) -- (2.5,0);
\draw[] (3,0.6666666666666666) -- (2.5,0.8333333333333334);
\draw[] (2.5,0.8333333333333334) -- (2.5,1);
\draw (2.5,0.9166666666666666) node[scale=1.2000000000000002,inner sep=0pt,shape=circle,minimum size=3pt,fill=black] {};
\draw (2.5,0.9166666666666666) node[scale=1.2000000000000002,inner sep=0pt,left=3.8pt] {$t$};
\draw (2.5,0.9166666666666666) node[scale=1.2000000000000002,inner sep=0pt,right=10pt,above=-3.3pt] {$t^{\unaryminus 1}$};
\draw[] (3,0.3333333333333333) -- (3.5,0.16666666666666666);
\draw[] (3.5,0.16666666666666666) -- (3.5,0);
\draw[] (3,0.6666666666666666) -- (3.5,0.8333333333333334);
\draw[] (3.5,0.8333333333333334) -- (3.5,1);
\draw (3.5,0.9166666666666666) node[scale=1.2000000000000002,inner sep=0pt,shape=circle,minimum size=3pt,fill=black] {};
\draw (3.5,0.9166666666666666) node[scale=1.2000000000000002,inner sep=0pt,left=3.8pt] {$t$};
\draw (3.5,0.9166666666666666) node[scale=1.2000000000000002,inner sep=0pt,right=10pt,above=-3.3pt] {$t^{\unaryminus 1}$};
\draw (5,0.3333333333333333) node[scale=1.2000000000000002,inner sep=0pt,shape=circle,minimum size=3pt,fill=black] {};
\draw (5,0.3333333333333333) node[scale=1.2000000000000002,inner sep=0pt,below=7.589466384404111pt] {$A^{-}$};
\draw (5,0.6666666666666666) node[scale=1.2000000000000002,inner sep=0pt,shape=circle,minimum size=3pt,fill=black] {};
\draw (5,0.6666666666666666) node[scale=1.2000000000000002,inner sep=0pt,above=12.589466384404112pt,right=-4pt] {$A^{+}$};
\draw[] (5,0.3333333333333333) -- (4.5,0.16666666666666666);
\draw[] (4.5,0.16666666666666666) -- (4.5,0);
\draw[] (5,0.6666666666666666) -- (4.5,0.8333333333333334);
\draw[] (4.5,0.8333333333333334) -- (4.5,1);
\draw (4.5,0.9166666666666666) node[scale=1.2000000000000002,inner sep=0pt,shape=circle,minimum size=3pt,fill=black] {};
\draw (4.5,0.9166666666666666) node[scale=1.2000000000000002,inner sep=0pt,left=3.8pt] {$t$};
\draw (4.5,0.9166666666666666) node[scale=1.2000000000000002,inner sep=0pt,right=10pt,above=-3.3pt] {$t^{\unaryminus 1}$};
\draw[] (5,0.3333333333333333) -- (5.5,0.16666666666666666);
\draw[] (5.5,0.16666666666666666) -- (5.5,0);
\draw[] (5,0.6666666666666666) -- (5.5,0.8333333333333334);
\draw[] (5.5,0.8333333333333334) -- (5.5,1);
\draw (5.5,0.9166666666666666) node[scale=1.2000000000000002,inner sep=0pt,shape=circle,minimum size=3pt,fill=black] {};
\draw (5.5,0.9166666666666666) node[scale=1.2000000000000002,inner sep=0pt,left=3.8pt] {$t$};
\draw (5.5,0.9166666666666666) node[scale=1.2000000000000002,inner sep=0pt,right=10pt,above=-3.3pt] {$t^{\unaryminus 1}$};
\draw (2,0.3333333333333333) node[scale=1.2000000000000002,inner sep=0pt,shape=circle,minimum size=3pt,fill=black] {};
\draw (2,0.3333333333333333) node[scale=1.2000000000000002,inner sep=0pt,below=7.589466384404111pt] {$A^{+}$};
\draw (2,0.6666666666666666) node[scale=1.2000000000000002,inner sep=0pt,shape=circle,minimum size=3pt,fill=black] {};
\draw (2,0.6666666666666666) node[scale=1.2000000000000002,inner sep=0pt,above=12.589466384404112pt,right=-4pt] {$A^{-}$};
\draw[] (2,0.3333333333333333) -- (2,0.6666666666666666);
\draw (2,0.5) node[scale=1.2000000000000002,inner sep=0pt,shape=circle,minimum size=3pt,fill=black] {};
\draw (2,0.5) node[scale=1.2000000000000002,inner sep=0pt,left=3.8pt] {$t$};
\draw (2,0.5) node[scale=1.2000000000000002,inner sep=0pt,right=10pt,above=-3.3pt] {$t^{\unaryminus 1}$};
\draw (4,0.3333333333333333) node[scale=1.2000000000000002,inner sep=0pt,shape=circle,minimum size=3pt,fill=black] {};
\draw (4,0.3333333333333333) node[scale=1.2000000000000002,inner sep=0pt,below=7.589466384404111pt] {$A^{+}$};
\draw (4,0.6666666666666666) node[scale=1.2000000000000002,inner sep=0pt,shape=circle,minimum size=3pt,fill=black] {};
\draw (4,0.6666666666666666) node[scale=1.2000000000000002,inner sep=0pt,above=12.589466384404112pt,right=-4pt] {$A^{-}$};
\draw[] (4,0.3333333333333333) -- (4,0.6666666666666666);
\draw (4,0.5) node[scale=1.2000000000000002,inner sep=0pt,shape=circle,minimum size=3pt,fill=black] {};
\draw (4,0.5) node[scale=1.2000000000000002,inner sep=0pt,left=3.8pt] {$t$};
\draw (4,0.5) node[scale=1.2000000000000002,inner sep=0pt,right=10pt,above=-3.3pt] {$t^{\unaryminus 1}$};
\draw[] (1,0.6666666666666666) -- (2,0.6666666666666666);
\draw[] (1,0.3333333333333333) -- (2,0.3333333333333333);
\draw[] (4,0.6666666666666666) -- (5,0.6666666666666666);
\draw[] (4,0.3333333333333333) -- (5,0.3333333333333333);
\draw (1.5,0.6666666666666666) node[scale=1.2000000000000002,inner sep=0pt,shape=circle,minimum size=3pt,fill=black] {};
\draw (1.5,0.6666666666666666) node[scale=1.2000000000000002,inner sep=0pt,above=3pt] {$t$};
\draw (1.5,0.6666666666666666) node[scale=1.2000000000000002,inner sep=0pt,below=5pt,right=-2.3pt] {$t^{\unaryminus 1}$};
\draw (1.5,0.3333333333333333) node[scale=1.2000000000000002,inner sep=0pt,shape=circle,minimum size=3pt,fill=black] {};
\draw (1.5,0.3333333333333333) node[scale=1.2000000000000002,inner sep=0pt,below=3pt] {$t$};
\draw (1.5,0.3333333333333333) node[scale=1.2000000000000002,inner sep=0pt,above=7pt,right=-2.3pt] {$t^{\unaryminus 1}$};
\draw (4.5,0.3333333333333333) node[scale=1.2000000000000002,inner sep=0pt,shape=circle,minimum size=3pt,fill=black] {};
\draw (4.5,0.3333333333333333) node[scale=1.2000000000000002,inner sep=0pt,above=3pt] {$t$};
\draw (4.5,0.3333333333333333) node[scale=1.2000000000000002,inner sep=0pt,below=5pt,right=-2.3pt] {$t^{\unaryminus 1}$};
\draw (4.5,0.6666666666666666) node[scale=1.2000000000000002,inner sep=0pt,shape=circle,minimum size=3pt,fill=black] {};
\draw (4.5,0.6666666666666666) node[scale=1.2000000000000002,inner sep=0pt,below=3pt] {$t$};
\draw (4.5,0.6666666666666666) node[scale=1.2000000000000002,inner sep=0pt,above=7pt,right=-2.3pt] {$t^{\unaryminus 1}$};
\draw[] (2,0.6666666666666666) -- (3,0.6666666666666666);
\draw[] (3,0.3333333333333333) -- (4,0.3333333333333333);
\draw[] (2,0.3333333333333333) -- (4,0.6666666666666666);
\draw (2.5,0.6666666666666666) node[scale=1.2000000000000002,inner sep=0pt,shape=circle,minimum size=3pt,fill=black] {};
\draw (2.5,0.6666666666666666) node[scale=1.2000000000000002,inner sep=0pt,below=3pt] {$t$};
\draw (2.5,0.6666666666666666) node[scale=1.2000000000000002,inner sep=0pt,above=7pt,right=-2.3pt] {$t^{\unaryminus 1}$};
\draw (3.5,0.3333333333333333) node[scale=1.2000000000000002,inner sep=0pt,shape=circle,minimum size=3pt,fill=black] {};
\draw (3.5,0.3333333333333333) node[scale=1.2000000000000002,inner sep=0pt,below=3pt] {$t$};
\draw (3.5,0.3333333333333333) node[scale=1.2000000000000002,inner sep=0pt,above=7pt,right=-2.3pt] {$t^{\unaryminus 1}$};
\draw (3,0.5) node[scale=1.2000000000000002,inner sep=0pt,shape=circle,minimum size=3pt,fill=black] {};
\draw (3,0.5) node[scale=1.2000000000000002,inner sep=0pt,above=3pt] {$t$};
\draw (3,0.5) node[scale=1.2000000000000002,inner sep=0pt,below=5pt,right=-2.3pt] {$t^{\unaryminus 1}$};
}


\begin{scope}[shift={(0,0)}]
    \DrN
\end{scope}

\end{tikzpicture}

%% file: images/D_+2.tikz
\begin{tikzpicture}[xscale=1.5,yscale=6]
\providecommand{\unaryminus}{\scalebox{0.5}[0.75]{\( - \)}}
\fontsize{10}{12}\selectfont
\draw[thick] (0,1) -- (2,1);
\draw[thick] (2,0) -- (0,0);
\draw[dashed, thick] (0,0) -- (0,1);
\draw[dashed, thick] (2,1) -- (2,0);
\draw (0,0) node[scale=1.2000000000000002,inner sep=0pt,shape=isosceles triangle,minimum size=3pt,fill=black] {};
\draw (0,1) node[scale=1.2000000000000002,inner sep=0pt,shape=isosceles triangle,minimum size=3pt,fill=black] {};
\draw (1,0) node[scale=1.2000000000000002,inner sep=0pt,shape=isosceles triangle,minimum size=3pt,fill=black] {};
\draw (1,1) node[scale=1.2000000000000002,inner sep=0pt,shape=isosceles triangle,minimum size=3pt,fill=black] {};
\draw (2,0) node[scale=1.2000000000000002,inner sep=0pt,shape=isosceles triangle,minimum size=3pt,fill=black] {};
\draw (2,1) node[scale=1.2000000000000002,inner sep=0pt,shape=isosceles triangle,minimum size=3pt,fill=black] {};
\draw (0.5,0) node[scale=1.2000000000000002,inner sep=0pt,minimum size=14pt,below=0pt] {$e_{2}$};
\draw (0.5,1) node[scale=1.2000000000000002,inner sep=0pt,minimum size=14pt,above=0pt] {$e_{1}$};
\draw (1.5,0) node[scale=1.2000000000000002,inner sep=0pt,minimum size=14pt,below=0pt] {$e_{1}$};
\draw (1.5,1) node[scale=1.2000000000000002,inner sep=0pt,minimum size=14pt,above=0pt] {$e_{2}$};
\draw[] (0.25,1) -- (0.25,0.3333333333333333);
\draw (1,0.3333333333333333) node[scale=1.2000000000000002,inner sep=0pt,shape=circle,minimum size=3pt,fill=black] {};
\draw (1,0.3333333333333333) node[scale=1.2000000000000002,inner sep=0pt,right=6pt,above=3pt] {$a$};
\draw (1,0.3333333333333333) node[scale=1.2000000000000002,inner sep=0pt,right=0pt,below=6.123105625617663pt] {$a$};
\draw (1,0.3333333333333333) node[scale=1.2000000000000002,inner sep=0pt,left=7pt,above=3pt] {$a^{\unaryminus 1}$};
\draw (1,0.3333333333333333) node[scale=1.2000000000000002,inner sep=0pt,left=12.80776406404415pt,below=1pt] {$a^{\unaryminus 1}$};
\draw (1,0.6666666666666666) node[scale=1.2000000000000002,inner sep=0pt,shape=circle,minimum size=3pt,fill=black] {};
\draw (1,0.6666666666666666) node[scale=1.2000000000000002,inner sep=0pt,right=10.05397531527947pt,above=2pt] {$t$};
\draw (1,0.6666666666666666) node[scale=1.2000000000000002,inner sep=0pt,left=0pt,above=7.123105625617663pt] {$t$};
\draw (1,0.6666666666666666) node[scale=1.2000000000000002,inner sep=0pt,right=9pt,below=2.5pt] {$t^{\unaryminus 1}$};
\draw (1,0.6666666666666666) node[scale=1.2000000000000002,inner sep=0pt,left=7pt,below=2.5pt] {$t^{\unaryminus 1}$};
\draw[] (1,0.3333333333333333) -- (1,0.6666666666666666);
\draw[] (1,0.6666666666666666) -- (0.5,0.8333333333333334);
\draw[] (0.5,0.8333333333333334) -- (0.5,1);
\draw[] (1,0.3333333333333333) -- (0.5,0.16666666666666666);
\draw[] (0.5,0.16666666666666666) -- (0.5,0);
\draw[] (1,0.6666666666666666) -- (1.5,0.8333333333333334);
\draw[] (1.5,0.8333333333333334) -- (1.5,1);
\draw[] (1,0.3333333333333333) -- (1.5,0.16666666666666666);
\draw[] (1.5,0.16666666666666666) -- (1.5,0);
\draw[] (0.25,0.3333333333333333) -- (1,0.3333333333333333);
\draw[] (1.75,0.6666666666666666) -- (1,0.6666666666666666);
\draw[] (0.75,0) -- (0.75,0.041666666666666664);
\draw[] (0.75,0.041666666666666664) -- (1.25,0.041666666666666664);
\draw[] (1.25,0.041666666666666664) -- (1.25,0);
\draw[] (1.75,0.6666666666666666) -- (1.75,1);

\end{tikzpicture}

%% file: images/D_+4.tikz
\begin{tikzpicture}[xscale=1.5,yscale=6]
\providecommand{\unaryminus}{\scalebox{0.5}[0.75]{\( - \)}}
\fontsize{10}{12}\selectfont
\draw[thick] (0,1) -- (4,1);
\draw[thick] (4,0) -- (0,0);
\draw[dashed, thick] (0,0) -- (0,1);
\draw[dashed, thick] (4,1) -- (4,0);
\draw (0,0) node[scale=1.2000000000000002,inner sep=0pt,shape=isosceles triangle,minimum size=3pt,fill=black] {};
\draw (0,1) node[scale=1.2000000000000002,inner sep=0pt,shape=isosceles triangle,minimum size=3pt,fill=black] {};
\draw (1,0) node[scale=1.2000000000000002,inner sep=0pt,shape=isosceles triangle,minimum size=3pt,fill=black] {};
\draw (1,1) node[scale=1.2000000000000002,inner sep=0pt,shape=isosceles triangle,minimum size=3pt,fill=black] {};
\draw (2,0) node[scale=1.2000000000000002,inner sep=0pt,shape=isosceles triangle,minimum size=3pt,fill=black] {};
\draw (2,1) node[scale=1.2000000000000002,inner sep=0pt,shape=isosceles triangle,minimum size=3pt,fill=black] {};
\draw (3,0) node[scale=1.2000000000000002,inner sep=0pt,shape=isosceles triangle,minimum size=3pt,fill=black] {};
\draw (3,1) node[scale=1.2000000000000002,inner sep=0pt,shape=isosceles triangle,minimum size=3pt,fill=black] {};
\draw (4,0) node[scale=1.2000000000000002,inner sep=0pt,shape=isosceles triangle,minimum size=3pt,fill=black] {};
\draw (4,1) node[scale=1.2000000000000002,inner sep=0pt,shape=isosceles triangle,minimum size=3pt,fill=black] {};
\draw (0.5,0) node[scale=1.2000000000000002,inner sep=0pt,minimum size=14pt,below=0pt] {$e_{2}$};
\draw (0.5,1) node[scale=1.2000000000000002,inner sep=0pt,minimum size=14pt,above=0pt] {$e_{1}$};
\draw (1.5,0) node[scale=1.2000000000000002,inner sep=0pt,minimum size=14pt,below=0pt] {$e_{1}$};
\draw (1.5,1) node[scale=1.2000000000000002,inner sep=0pt,minimum size=14pt,above=0pt] {$e_{2}$};
\draw (2.5,0) node[scale=1.2000000000000002,inner sep=0pt,minimum size=14pt,below=0pt] {$e_{4}$};
\draw (2.5,1) node[scale=1.2000000000000002,inner sep=0pt,minimum size=14pt,above=0pt] {$e_{3}$};
\draw (3.5,0) node[scale=1.2000000000000002,inner sep=0pt,minimum size=14pt,below=0pt] {$e_{3}$};
\draw (3.5,1) node[scale=1.2000000000000002,inner sep=0pt,minimum size=14pt,above=0pt] {$e_{4}$};
\draw[] (0.25,1) -- (0.25,0.3333333333333333);
\draw (1,0.3333333333333333) node[scale=1.2000000000000002,inner sep=0pt,shape=circle,minimum size=3pt,fill=black] {};
\draw (1,0.3333333333333333) node[scale=1.2000000000000002,inner sep=0pt,right=6pt,above=3pt] {$a$};
\draw (1,0.3333333333333333) node[scale=1.2000000000000002,inner sep=0pt,right=0pt,below=6.123105625617663pt] {$a$};
\draw (1,0.3333333333333333) node[scale=1.2000000000000002,inner sep=0pt,left=7pt,above=3pt] {$a^{\unaryminus 1}$};
\draw (1,0.3333333333333333) node[scale=1.2000000000000002,inner sep=0pt,left=12.80776406404415pt,below=1pt] {$a^{\unaryminus 1}$};
\draw (1,0.6666666666666666) node[scale=1.2000000000000002,inner sep=0pt,shape=circle,minimum size=3pt,fill=black] {};
\draw (1,0.6666666666666666) node[scale=1.2000000000000002,inner sep=0pt,right=10.05397531527947pt,above=2pt] {$t$};
\draw (1,0.6666666666666666) node[scale=1.2000000000000002,inner sep=0pt,left=0pt,above=7.123105625617663pt] {$t$};
\draw (1,0.6666666666666666) node[scale=1.2000000000000002,inner sep=0pt,right=9pt,below=2.5pt] {$t^{\unaryminus 1}$};
\draw (1,0.6666666666666666) node[scale=1.2000000000000002,inner sep=0pt,left=7pt,below=2.5pt] {$t^{\unaryminus 1}$};
\draw[] (1,0.3333333333333333) -- (1,0.6666666666666666);
\draw[] (1,0.6666666666666666) -- (0.5,0.8333333333333334);
\draw[] (0.5,0.8333333333333334) -- (0.5,1);
\draw[] (1,0.3333333333333333) -- (0.5,0.16666666666666666);
\draw[] (0.5,0.16666666666666666) -- (0.5,0);
\draw[] (1,0.6666666666666666) -- (1.5,0.8333333333333334);
\draw[] (1.5,0.8333333333333334) -- (1.5,1);
\draw[] (1,0.3333333333333333) -- (1.5,0.16666666666666666);
\draw[] (1.5,0.16666666666666666) -- (1.5,0);
\draw[] (0.25,0.3333333333333333) -- (1,0.3333333333333333);
\draw[] (1.75,0.6666666666666666) -- (1,0.6666666666666666);
\draw[] (0.75,0) -- (0.75,0.041666666666666664);
\draw[] (0.75,0.041666666666666664) -- (1.25,0.041666666666666664);
\draw[] (1.25,0.041666666666666664) -- (1.25,0);
\draw (3,0.3333333333333333) node[scale=1.2000000000000002,inner sep=0pt,shape=circle,minimum size=3pt,fill=black] {};
\draw (3,0.3333333333333333) node[scale=1.2000000000000002,inner sep=0pt,right=6pt,above=3pt] {$a$};
\draw (3,0.3333333333333333) node[scale=1.2000000000000002,inner sep=0pt,right=0pt,below=6.123105625617663pt] {$a$};
\draw (3,0.3333333333333333) node[scale=1.2000000000000002,inner sep=0pt,left=7pt,above=3pt] {$a^{\unaryminus 1}$};
\draw (3,0.3333333333333333) node[scale=1.2000000000000002,inner sep=0pt,left=12.80776406404415pt,below=1pt] {$a^{\unaryminus 1}$};
\draw (3,0.6666666666666666) node[scale=1.2000000000000002,inner sep=0pt,shape=circle,minimum size=3pt,fill=black] {};
\draw (3,0.6666666666666666) node[scale=1.2000000000000002,inner sep=0pt,right=10.05397531527947pt,above=2pt] {$t$};
\draw (3,0.6666666666666666) node[scale=1.2000000000000002,inner sep=0pt,left=0pt,above=7.123105625617663pt] {$t$};
\draw (3,0.6666666666666666) node[scale=1.2000000000000002,inner sep=0pt,right=9pt,below=2.5pt] {$t^{\unaryminus 1}$};
\draw (3,0.6666666666666666) node[scale=1.2000000000000002,inner sep=0pt,left=7pt,below=2.5pt] {$t^{\unaryminus 1}$};
\draw[] (3,0.3333333333333333) -- (3,0.6666666666666666);
\draw[] (3,0.6666666666666666) -- (2.5,0.8333333333333334);
\draw[] (2.5,0.8333333333333334) -- (2.5,1);
\draw[] (3,0.3333333333333333) -- (2.5,0.16666666666666666);
\draw[] (2.5,0.16666666666666666) -- (2.5,0);
\draw[] (3,0.6666666666666666) -- (3.5,0.8333333333333334);
\draw[] (3.5,0.8333333333333334) -- (3.5,1);
\draw[] (3,0.3333333333333333) -- (3.5,0.16666666666666666);
\draw[] (3.5,0.16666666666666666) -- (3.5,0);
\draw[] (2.25,0.3333333333333333) -- (3,0.3333333333333333);
\draw[] (3.75,0.6666666666666666) -- (3,0.6666666666666666);
\draw[] (2.75,0) -- (2.75,0.041666666666666664);
\draw[] (2.75,0.041666666666666664) -- (3.25,0.041666666666666664);
\draw[] (3.25,0.041666666666666664) -- (3.25,0);
\draw[] (1.75,0.6666666666666666) -- (2.25,0.3333333333333333);
\draw[] (1.75,1) -- (1.75,0.9583333333333334);
\draw[] (1.75,0.9583333333333334) -- (2.25,0.9583333333333334);
\draw[] (2.25,0.9583333333333334) -- (2.25,1);
\draw[] (3.75,0.6666666666666666) -- (3.75,1);

\end{tikzpicture}

%% file: images/D_7_5.tikz
\begin{tikzpicture}[xscale=1.5,yscale=6]
\providecommand{\unaryminus}{\scalebox{0.5}[0.75]{\( - \)}}
\fontsize{10}{12}\selectfont
\newcommand*{\DrN}{%
\draw[thick] (0,1) -- (4,1);
\draw[thick] (4,0) -- (0,0);
\draw[dashed, thick] (0,0) -- (0,1);
\draw (0,0) node[scale=1.2000000000000002,inner sep=0pt,shape=isosceles triangle,minimum size=3pt,fill=black] {};
\draw (0,1) node[scale=1.2000000000000002,inner sep=0pt,shape=isosceles triangle,minimum size=3pt,fill=black] {};
\draw (1,0) node[scale=1.2000000000000002,inner sep=0pt,shape=isosceles triangle,minimum size=3pt,fill=black] {};
\draw (1,1) node[scale=1.2000000000000002,inner sep=0pt,shape=isosceles triangle,minimum size=3pt,fill=black] {};
\draw (2,0) node[scale=1.2000000000000002,inner sep=0pt,shape=isosceles triangle,minimum size=3pt,fill=black] {};
\draw (2,1) node[scale=1.2000000000000002,inner sep=0pt,shape=isosceles triangle,minimum size=3pt,fill=black] {};
\draw (3,0) node[scale=1.2000000000000002,inner sep=0pt,shape=isosceles triangle,minimum size=3pt,fill=black] {};
\draw (3,1) node[scale=1.2000000000000002,inner sep=0pt,shape=isosceles triangle,minimum size=3pt,fill=black] {};
\draw (4,0) node[scale=1.2000000000000002,inner sep=0pt,shape=isosceles triangle,minimum size=3pt,fill=black] {};
\draw (4,1) node[scale=1.2000000000000002,inner sep=0pt,shape=isosceles triangle,minimum size=3pt,fill=black] {};
\draw (0.5,0) node[scale=1.2000000000000002,inner sep=0pt,minimum size=14pt,below=0pt] {$e_{2}$};
\draw (0.5,1) node[scale=1.2000000000000002,inner sep=0pt,minimum size=14pt,above=0pt] {$e_{1}$};
\draw (1.5,0) node[scale=1.2000000000000002,inner sep=0pt,minimum size=14pt,below=0pt] {$e_{1}$};
\draw (1.5,1) node[scale=1.2000000000000002,inner sep=0pt,minimum size=14pt,above=0pt] {$e_{2}$};
\draw (2.5,0) node[scale=1.2000000000000002,inner sep=0pt,minimum size=14pt,below=0pt] {$e_{4}$};
\draw (2.5,1) node[scale=1.2000000000000002,inner sep=0pt,minimum size=14pt,above=0pt] {$e_{3}$};
\draw (3.5,0) node[scale=1.2000000000000002,inner sep=0pt,minimum size=14pt,below=0pt] {$e_{3}$};
\draw (3.5,1) node[scale=1.2000000000000002,inner sep=0pt,minimum size=14pt,above=0pt] {$e_{4}$};
\draw (1,0.3333333333333333) node[scale=1.2000000000000002,inner sep=0pt,shape=circle,minimum size=3pt,fill=black] {};
\draw (1,0.3333333333333333) node[scale=1.2000000000000002,inner sep=0pt,below=9.895453501482391pt] {$A^{-}$};
\draw (1,0.6666666666666666) node[scale=1.2000000000000002,inner sep=0pt,shape=circle,minimum size=3pt,fill=black] {};
\draw (1,0.6666666666666666) node[scale=1.2000000000000002,inner sep=0pt,above=14.895453501482391pt,right=-4pt] {$A^{+}$};
\draw[] (1,0.3333333333333333) -- (0.5,0.16666666666666666);
\draw[] (0.5,0.16666666666666666) -- (0.5,0);
\draw[] (1,0.6666666666666666) -- (0.5,0.8333333333333334);
\draw[] (0.5,0.8333333333333334) -- (0.5,1);
\draw (0.5,0.9166666666666666) node[scale=1.2000000000000002,inner sep=0pt,shape=circle,minimum size=3pt,fill=black] {};
\draw (0.5,0.9166666666666666) node[scale=1.2000000000000002,inner sep=0pt,left=3.8pt] {$t$};
\draw (0.5,0.9166666666666666) node[scale=1.2000000000000002,inner sep=0pt,right=10pt,above=-3.3pt] {$t^{\unaryminus 1}$};
\draw[] (1,0.3333333333333333) -- (1.5,0.16666666666666666);
\draw[] (1.5,0.16666666666666666) -- (1.5,0);
\draw[] (1,0.6666666666666666) -- (1.5,0.8333333333333334);
\draw[] (1.5,0.8333333333333334) -- (1.5,1);
\draw (1.5,0.9166666666666666) node[scale=1.2000000000000002,inner sep=0pt,shape=circle,minimum size=3pt,fill=black] {};
\draw (1.5,0.9166666666666666) node[scale=1.2000000000000002,inner sep=0pt,left=3.8pt] {$t$};
\draw (1.5,0.9166666666666666) node[scale=1.2000000000000002,inner sep=0pt,right=10pt,above=-3.3pt] {$t^{\unaryminus 1}$};
\draw[] (1,0.3333333333333333) -- (2.5,0.16666666666666666);
\draw[] (2.5,0.16666666666666666) -- (2.5,0);
\draw[] (1,0.6666666666666666) -- (2.5,0.8333333333333334);
\draw[] (2.5,0.8333333333333334) -- (2.5,1);
\draw (2.5,0.9166666666666666) node[scale=1.2000000000000002,inner sep=0pt,shape=circle,minimum size=3pt,fill=black] {};
\draw (2.5,0.9166666666666666) node[scale=1.2000000000000002,inner sep=0pt,left=3.8pt] {$t$};
\draw (2.5,0.9166666666666666) node[scale=1.2000000000000002,inner sep=0pt,right=10pt,above=-3.3pt] {$t^{\unaryminus 1}$};
\draw[] (1,0.3333333333333333) -- (3.5,0.16666666666666666);
\draw[] (3.5,0.16666666666666666) -- (3.5,0);
\draw[] (4,0.9583333333333334) -- (3.5,0.9583333333333334);
\draw[] (3.5,0.9583333333333334) -- (3.5,1);
\draw[] (1,0.6666666666666666) -- (3.75,0.6666666666666666);
\draw[] (3.75,0.6666666666666666) -- (4,0.5);
\draw (3.5,0.6666666666666666) node[scale=1.2000000000000002,inner sep=0pt,shape=circle,minimum size=3pt,fill=black] {};
\draw (3.5,0.6666666666666666) node[scale=1.2000000000000002,inner sep=0pt,above=3pt] {$t$};
\draw (3.5,0.6666666666666666) node[scale=1.2000000000000002,inner sep=0pt,below=5pt,right=-2.3pt] {$t^{\unaryminus 1}$};
\draw[] (1,0.3333333333333333) -- (1,0.6666666666666666);
\draw (1,0.5) node[scale=1.2000000000000002,inner sep=0pt,shape=circle,minimum size=3pt,fill=black] {};
\draw (1,0.5) node[scale=1.2000000000000002,inner sep=0pt,right=3.8pt] {$t$};
\draw (1,0.5) node[scale=1.2000000000000002,inner sep=0pt,left=8pt,above=-3.3pt] {$t^{\unaryminus 1}$};
}
\newcommand*{\Dplustwos}{%
\draw[thick] (0,1) -- (2,1);
\draw[thick] (2,0) -- (0,0);
\draw[dashed, thick] (2,1) -- (2,0);
\draw (0,0) node[scale=1.2000000000000002,inner sep=0pt,shape=isosceles triangle,minimum size=3pt,fill=black] {};
\draw (0,1) node[scale=1.2000000000000002,inner sep=0pt,shape=isosceles triangle,minimum size=3pt,fill=black] {};
\draw (1,0) node[scale=1.2000000000000002,inner sep=0pt,shape=isosceles triangle,minimum size=3pt,fill=black] {};
\draw (1,1) node[scale=1.2000000000000002,inner sep=0pt,shape=isosceles triangle,minimum size=3pt,fill=black] {};
\draw (2,0) node[scale=1.2000000000000002,inner sep=0pt,shape=isosceles triangle,minimum size=3pt,fill=black] {};
\draw (2,1) node[scale=1.2000000000000002,inner sep=0pt,shape=isosceles triangle,minimum size=3pt,fill=black] {};
\draw (0.5,0) node[scale=1.2000000000000002,inner sep=0pt,minimum size=14pt,below=0pt] {$e_{6}$};
\draw (0.5,1) node[scale=1.2000000000000002,inner sep=0pt,minimum size=14pt,above=0pt] {$e_{5}$};
\draw (1.5,0) node[scale=1.2000000000000002,inner sep=0pt,minimum size=14pt,below=0pt] {$e_{5}$};
\draw (1.5,1) node[scale=1.2000000000000002,inner sep=0pt,minimum size=14pt,above=0pt] {$e_{6}$};
\draw[] (0,0.9583333333333334) -- (0.25,0.9583333333333334);
\draw[] (0.25,0.9583333333333334) -- (0.25,1);
\draw[] (0,0.5) -- (0.25,0.3333333333333333);
\draw (1,0.3333333333333333) node[scale=1.2000000000000002,inner sep=0pt,shape=circle,minimum size=3pt,fill=black] {};
\draw (1,0.3333333333333333) node[scale=1.2000000000000002,inner sep=0pt,right=6pt,above=3pt] {$a$};
\draw (1,0.3333333333333333) node[scale=1.2000000000000002,inner sep=0pt,right=0pt,below=6.123105625617663pt] {$a$};
\draw (1,0.3333333333333333) node[scale=1.2000000000000002,inner sep=0pt,left=7pt,above=3pt] {$a^{\unaryminus 1}$};
\draw (1,0.3333333333333333) node[scale=1.2000000000000002,inner sep=0pt,left=12.80776406404415pt,below=1pt] {$a^{\unaryminus 1}$};
\draw (1,0.6666666666666666) node[scale=1.2000000000000002,inner sep=0pt,shape=circle,minimum size=3pt,fill=black] {};
\draw (1,0.6666666666666666) node[scale=1.2000000000000002,inner sep=0pt,right=10.05397531527947pt,above=2pt] {$t$};
\draw (1,0.6666666666666666) node[scale=1.2000000000000002,inner sep=0pt,left=0pt,above=7.123105625617663pt] {$t$};
\draw (1,0.6666666666666666) node[scale=1.2000000000000002,inner sep=0pt,right=9pt,below=2.5pt] {$t^{\unaryminus 1}$};
\draw (1,0.6666666666666666) node[scale=1.2000000000000002,inner sep=0pt,left=7pt,below=2.5pt] {$t^{\unaryminus 1}$};
\draw[] (1,0.3333333333333333) -- (1,0.6666666666666666);
\draw[] (1,0.6666666666666666) -- (0.5,0.8333333333333334);
\draw[] (0.5,0.8333333333333334) -- (0.5,1);
\draw[] (1,0.3333333333333333) -- (0.5,0.16666666666666666);
\draw[] (0.5,0.16666666666666666) -- (0.5,0);
\draw[] (1,0.6666666666666666) -- (1.5,0.8333333333333334);
\draw[] (1.5,0.8333333333333334) -- (1.5,1);
\draw[] (1,0.3333333333333333) -- (1.5,0.16666666666666666);
\draw[] (1.5,0.16666666666666666) -- (1.5,0);
\draw[] (0.25,0.3333333333333333) -- (1,0.3333333333333333);
\draw[] (1.75,0.6666666666666666) -- (1,0.6666666666666666);
\draw[] (0.75,0) -- (0.75,0.041666666666666664);
\draw[] (0.75,0.041666666666666664) -- (1.25,0.041666666666666664);
\draw[] (1.25,0.041666666666666664) -- (1.25,0);
\draw[] (1.75,0.6666666666666666) -- (1.75,1);
}


\begin{scope}[shift={(0,0)}]
    \DrN
\end{scope}

\begin{scope}[shift={(4,0)}]
    \Dplustwos
\end{scope}

\end{tikzpicture}

%% file: images/D_+N+1_3.tikz
\begin{tikzpicture}[xscale=2.0,yscale=6]
\providecommand{\unaryminus}{\scalebox{0.5}[0.75]{\( - \)}}
\fontsize{10}{12}\selectfont
\draw[thick] (0,1) -- (2,1);
\draw[thick] (2,0) -- (0,0);
\draw[dashed, thick] (0,0) -- (0,1);
\draw[dashed, thick] (2,1) -- (2,0);
\draw (0,0) node[scale=1.2000000000000002,inner sep=0pt,shape=isosceles triangle,minimum size=3pt,fill=black] {};
\draw (0,1) node[scale=1.2000000000000002,inner sep=0pt,shape=isosceles triangle,minimum size=3pt,fill=black] {};
\draw (1,0) node[scale=1.2000000000000002,inner sep=0pt,shape=isosceles triangle,minimum size=3pt,fill=black] {};
\draw (1,1) node[scale=1.2000000000000002,inner sep=0pt,shape=isosceles triangle,minimum size=3pt,fill=black] {};
\draw (2,0) node[scale=1.2000000000000002,inner sep=0pt,shape=isosceles triangle,minimum size=3pt,fill=black] {};
\draw (2,1) node[scale=1.2000000000000002,inner sep=0pt,shape=isosceles triangle,minimum size=3pt,fill=black] {};
\draw (0.5,0) node[scale=1.2000000000000002,inner sep=0pt,minimum size=14pt,below=0pt] {$e_{2}$};
\draw (0.5,1) node[scale=1.2000000000000002,inner sep=0pt,minimum size=14pt,above=0pt] {$e_{1}$};
\draw (1.5,0) node[scale=1.2000000000000002,inner sep=0pt,minimum size=14pt,below=0pt] {$e_{1}$};
\draw (1.5,1) node[scale=1.2000000000000002,inner sep=0pt,minimum size=14pt,above=0pt] {$e_{2}$};
\draw (1,0.3333333333333333) node[scale=1.2000000000000002,inner sep=0pt,shape=circle,minimum size=3pt,fill=black] {};
\draw (1,0.3333333333333333) node[scale=1.2000000000000002,inner sep=0pt,below=7.589466384404111pt] {$A^{+}$};
\draw (1,0.6666666666666666) node[scale=1.2000000000000002,inner sep=0pt,shape=circle,minimum size=3pt,fill=black] {};
\draw (1,0.6666666666666666) node[scale=1.2000000000000002,inner sep=0pt,above=12.589466384404112pt,right=-4pt] {$A^{-}$};
\draw[] (1,0.3333333333333333) -- (1,0.6666666666666666);
\draw[] (1,0.3333333333333333) -- (0.5,0.16666666666666666);
\draw[] (0.5,0.16666666666666666) -- (0.5,0);
\draw[] (1,0.6666666666666666) -- (0.5,0.8333333333333334);
\draw[] (0.5,0.8333333333333334) -- (0.5,1);
\draw (0.5,0.9166666666666666) node[scale=1.2000000000000002,inner sep=0pt,shape=circle,minimum size=3pt,fill=black] {};
\draw (0.5,0.9166666666666666) node[scale=1.2000000000000002,inner sep=0pt,right=3.8pt] {$t$};
\draw (0.5,0.9166666666666666) node[scale=1.2000000000000002,inner sep=0pt,left=8pt,above=-3.3pt] {$t^{\unaryminus 1}$};
\draw[] (1,0.3333333333333333) -- (1.5,0.16666666666666666);
\draw[] (1.5,0.16666666666666666) -- (1.5,0);
\draw[] (1,0.6666666666666666) -- (1.5,0.8333333333333334);
\draw[] (1.5,0.8333333333333334) -- (1.5,1);
\draw (1.5,0.9166666666666666) node[scale=1.2000000000000002,inner sep=0pt,shape=circle,minimum size=3pt,fill=black] {};
\draw (1.5,0.9166666666666666) node[scale=1.2000000000000002,inner sep=0pt,right=3.8pt] {$t$};
\draw (1.5,0.9166666666666666) node[scale=1.2000000000000002,inner sep=0pt,left=8pt,above=-3.3pt] {$t^{\unaryminus 1}$};
\draw (1,0.5) node[scale=1.2000000000000002,inner sep=0pt,shape=circle,minimum size=3pt,fill=black] {};
\draw (1,0.5) node[scale=1.2000000000000002,inner sep=0pt,left=4pt,above=3pt] {$t$};
\draw (1,0.5) node[scale=1.2000000000000002,inner sep=0pt,left=4pt,below=3pt] {$t$};
\draw (1,0.5) node[scale=1.2000000000000002,inner sep=0pt,right=8.5pt,above=3pt] {$t^{\unaryminus 1}$};
\draw (1,0.5) node[scale=1.2000000000000002,inner sep=0pt,right=8.5pt,below=1pt] {$t^{\unaryminus 1}$};
\draw[] (1,0.5) -- (0.25,0.5);
\draw[] (0.25,0.5) -- (0.25,0);
\draw[] (1,0.5) -- (1.75,0.5);
\draw[] (1.75,0.5) -- (1.75,0);
\draw[] (1.25,1) -- (1.25,0.9583333333333334);
\draw[] (1.25,0.9583333333333334) -- (0.75,0.9583333333333334);
\draw[] (0.75,0.9583333333333334) -- (0.75,1);
\draw (1.5,0.5) node[scale=1.2000000000000002,inner sep=0pt,shape=circle,minimum size=3pt,fill=black] {};
\draw (1.5,0.5) node[scale=1.2000000000000002,inner sep=0pt,below=3pt] {$a$};
\draw (1.5,0.5) node[scale=1.2000000000000002,inner sep=0pt,above=7.5pt,right=-3pt] {$a^{\unaryminus 1}$};

\end{tikzpicture}

%% file: images/D_+N+1_5.tikz
\begin{tikzpicture}[xscale=1.5,yscale=6]
\providecommand{\unaryminus}{\scalebox{0.5}[0.75]{\( - \)}}
\fontsize{10}{12}\selectfont
\draw[thick] (0,1) -- (4,1);
\draw[thick] (4,0) -- (0,0);
\draw[dashed, thick] (0,0) -- (0,1);
\draw[dashed, thick] (4,1) -- (4,0);
\draw (0,0) node[scale=1.2000000000000002,inner sep=0pt,shape=isosceles triangle,minimum size=3pt,fill=black] {};
\draw (0,1) node[scale=1.2000000000000002,inner sep=0pt,shape=isosceles triangle,minimum size=3pt,fill=black] {};
\draw (1,0) node[scale=1.2000000000000002,inner sep=0pt,shape=isosceles triangle,minimum size=3pt,fill=black] {};
\draw (1,1) node[scale=1.2000000000000002,inner sep=0pt,shape=isosceles triangle,minimum size=3pt,fill=black] {};
\draw (2,0) node[scale=1.2000000000000002,inner sep=0pt,shape=isosceles triangle,minimum size=3pt,fill=black] {};
\draw (2,1) node[scale=1.2000000000000002,inner sep=0pt,shape=isosceles triangle,minimum size=3pt,fill=black] {};
\draw (3,0) node[scale=1.2000000000000002,inner sep=0pt,shape=isosceles triangle,minimum size=3pt,fill=black] {};
\draw (3,1) node[scale=1.2000000000000002,inner sep=0pt,shape=isosceles triangle,minimum size=3pt,fill=black] {};
\draw (4,0) node[scale=1.2000000000000002,inner sep=0pt,shape=isosceles triangle,minimum size=3pt,fill=black] {};
\draw (4,1) node[scale=1.2000000000000002,inner sep=0pt,shape=isosceles triangle,minimum size=3pt,fill=black] {};
\draw (0.5,0) node[scale=1.2000000000000002,inner sep=0pt,minimum size=14pt,below=0pt] {$e_{2}$};
\draw (0.5,1) node[scale=1.2000000000000002,inner sep=0pt,minimum size=14pt,above=0pt] {$e_{1}$};
\draw (1.5,0) node[scale=1.2000000000000002,inner sep=0pt,minimum size=14pt,below=0pt] {$e_{1}$};
\draw (1.5,1) node[scale=1.2000000000000002,inner sep=0pt,minimum size=14pt,above=0pt] {$e_{2}$};
\draw (2.5,0) node[scale=1.2000000000000002,inner sep=0pt,minimum size=14pt,below=0pt] {$e_{4}$};
\draw (2.5,1) node[scale=1.2000000000000002,inner sep=0pt,minimum size=14pt,above=0pt] {$e_{3}$};
\draw (3.5,0) node[scale=1.2000000000000002,inner sep=0pt,minimum size=14pt,below=0pt] {$e_{3}$};
\draw (3.5,1) node[scale=1.2000000000000002,inner sep=0pt,minimum size=14pt,above=0pt] {$e_{4}$};
\draw (1,0.3333333333333333) node[scale=1.2000000000000002,inner sep=0pt,shape=circle,minimum size=3pt,fill=black] {};
\draw (1,0.3333333333333333) node[scale=1.2000000000000002,inner sep=0pt,below=9.895453501482391pt] {$A^{+}$};
\draw (1,0.6666666666666666) node[scale=1.2000000000000002,inner sep=0pt,shape=circle,minimum size=3pt,fill=black] {};
\draw (1,0.6666666666666666) node[scale=1.2000000000000002,inner sep=0pt,above=14.895453501482391pt,right=-4pt] {$A^{-}$};
\draw[] (1,0.3333333333333333) -- (1,0.6666666666666666);
\draw[] (1,0.3333333333333333) -- (0.5,0.16666666666666666);
\draw[] (0.5,0.16666666666666666) -- (0.5,0);
\draw[] (1,0.6666666666666666) -- (0.5,0.8333333333333334);
\draw[] (0.5,0.8333333333333334) -- (0.5,1);
\draw (0.5,0.9166666666666666) node[scale=1.2000000000000002,inner sep=0pt,shape=circle,minimum size=3pt,fill=black] {};
\draw (0.5,0.9166666666666666) node[scale=1.2000000000000002,inner sep=0pt,right=3.8pt] {$t$};
\draw (0.5,0.9166666666666666) node[scale=1.2000000000000002,inner sep=0pt,left=8pt,above=-3.3pt] {$t^{\unaryminus 1}$};
\draw[] (1,0.3333333333333333) -- (1.5,0.16666666666666666);
\draw[] (1.5,0.16666666666666666) -- (1.5,0);
\draw[] (1,0.6666666666666666) -- (1.5,0.8333333333333334);
\draw[] (1.5,0.8333333333333334) -- (1.5,1);
\draw (1.5,0.9166666666666666) node[scale=1.2000000000000002,inner sep=0pt,shape=circle,minimum size=3pt,fill=black] {};
\draw (1.5,0.9166666666666666) node[scale=1.2000000000000002,inner sep=0pt,right=3.8pt] {$t$};
\draw (1.5,0.9166666666666666) node[scale=1.2000000000000002,inner sep=0pt,left=8pt,above=-3.3pt] {$t^{\unaryminus 1}$};
\draw[] (1,0.3333333333333333) -- (2.5,0.16666666666666666);
\draw[] (2.5,0.16666666666666666) -- (2.5,0);
\draw[] (1,0.6666666666666666) -- (2.5,0.8333333333333334);
\draw[] (2.5,0.8333333333333334) -- (2.5,1);
\draw (2.5,0.9166666666666666) node[scale=1.2000000000000002,inner sep=0pt,shape=circle,minimum size=3pt,fill=black] {};
\draw (2.5,0.9166666666666666) node[scale=1.2000000000000002,inner sep=0pt,right=3.8pt] {$t$};
\draw (2.5,0.9166666666666666) node[scale=1.2000000000000002,inner sep=0pt,left=8pt,above=-3.3pt] {$t^{\unaryminus 1}$};
\draw[] (1,0.3333333333333333) -- (3.5,0.16666666666666666);
\draw[] (3.5,0.16666666666666666) -- (3.5,0);
\draw[] (1,0.6666666666666666) -- (3.5,0.8333333333333334);
\draw[] (3.5,0.8333333333333334) -- (3.5,1);
\draw (3.5,0.9166666666666666) node[scale=1.2000000000000002,inner sep=0pt,shape=circle,minimum size=3pt,fill=black] {};
\draw (3.5,0.9166666666666666) node[scale=1.2000000000000002,inner sep=0pt,right=3.8pt] {$t$};
\draw (3.5,0.9166666666666666) node[scale=1.2000000000000002,inner sep=0pt,left=8pt,above=-3.3pt] {$t^{\unaryminus 1}$};
\draw (1,0.5) node[scale=1.2000000000000002,inner sep=0pt,shape=circle,minimum size=3pt,fill=black] {};
\draw (1,0.5) node[scale=1.2000000000000002,inner sep=0pt,left=4pt,above=3pt] {$t$};
\draw (1,0.5) node[scale=1.2000000000000002,inner sep=0pt,left=4pt,below=3pt] {$t$};
\draw (1,0.5) node[scale=1.2000000000000002,inner sep=0pt,right=8.5pt,above=3pt] {$t^{\unaryminus 1}$};
\draw (1,0.5) node[scale=1.2000000000000002,inner sep=0pt,right=8.5pt,below=1pt] {$t^{\unaryminus 1}$};
\draw[] (1,0.5) -- (0.25,0.5);
\draw[] (0.25,0.5) -- (0.25,0);
\draw[] (1,0.5) -- (3.75,0.5);
\draw[] (3.75,0.5) -- (3.75,0);
\draw[] (1.25,1) -- (1.25,0.9583333333333334);
\draw[] (1.25,0.9583333333333334) -- (0.75,0.9583333333333334);
\draw[] (0.75,0.9583333333333334) -- (0.75,1);
\draw[] (3.25,1) -- (3.25,0.9583333333333334);
\draw[] (3.25,0.9583333333333334) -- (2.75,0.9583333333333334);
\draw[] (2.75,0.9583333333333334) -- (2.75,1);
\draw[] (1.75,0) -- (1.75,0.041666666666666664);
\draw[] (1.75,0.041666666666666664) -- (2.25,0.041666666666666664);
\draw[] (2.25,0.041666666666666664) -- (2.25,0);
\draw (2.5,0.5) node[scale=1.2000000000000002,inner sep=0pt,shape=circle,minimum size=3pt,fill=black] {};
\draw (2.5,0.5) node[scale=1.2000000000000002,inner sep=0pt,below=3pt] {$a$};
\draw (2.5,0.5) node[scale=1.2000000000000002,inner sep=0pt,above=7.5pt,right=-3pt] {$a^{\unaryminus 1}$};

\end{tikzpicture}

%% file: images/D_11_5.tikz
\begin{tikzpicture}[xscale=1.5,yscale=6]
\providecommand{\unaryminus}{\scalebox{0.5}[0.75]{\( - \)}}
\fontsize{10}{12}\selectfont
\newcommand*{\DrN}{%
\draw[thick] (0,1) -- (4,1);
\draw[thick] (4,0) -- (0,0);
\draw[dashed, thick] (4,1) -- (4,0);
\draw (0,0) node[scale=1.2000000000000002,inner sep=0pt,shape=isosceles triangle,minimum size=3pt,fill=black] {};
\draw (0,1) node[scale=1.2000000000000002,inner sep=0pt,shape=isosceles triangle,minimum size=3pt,fill=black] {};
\draw (1,0) node[scale=1.2000000000000002,inner sep=0pt,shape=isosceles triangle,minimum size=3pt,fill=black] {};
\draw (1,1) node[scale=1.2000000000000002,inner sep=0pt,shape=isosceles triangle,minimum size=3pt,fill=black] {};
\draw (2,0) node[scale=1.2000000000000002,inner sep=0pt,shape=isosceles triangle,minimum size=3pt,fill=black] {};
\draw (2,1) node[scale=1.2000000000000002,inner sep=0pt,shape=isosceles triangle,minimum size=3pt,fill=black] {};
\draw (3,0) node[scale=1.2000000000000002,inner sep=0pt,shape=isosceles triangle,minimum size=3pt,fill=black] {};
\draw (3,1) node[scale=1.2000000000000002,inner sep=0pt,shape=isosceles triangle,minimum size=3pt,fill=black] {};
\draw (4,0) node[scale=1.2000000000000002,inner sep=0pt,shape=isosceles triangle,minimum size=3pt,fill=black] {};
\draw (4,1) node[scale=1.2000000000000002,inner sep=0pt,shape=isosceles triangle,minimum size=3pt,fill=black] {};
\draw (0.5,0) node[scale=1.2000000000000002,inner sep=0pt,minimum size=14pt,below=0pt] {$e_{6}$};
\draw (0.5,1) node[scale=1.2000000000000002,inner sep=0pt,minimum size=14pt,above=0pt] {$e_{5}$};
\draw (1.5,0) node[scale=1.2000000000000002,inner sep=0pt,minimum size=14pt,below=0pt] {$e_{5}$};
\draw (1.5,1) node[scale=1.2000000000000002,inner sep=0pt,minimum size=14pt,above=0pt] {$e_{6}$};
\draw (2.5,0) node[scale=1.2000000000000002,inner sep=0pt,minimum size=14pt,below=0pt] {$e_{8}$};
\draw (2.5,1) node[scale=1.2000000000000002,inner sep=0pt,minimum size=14pt,above=0pt] {$e_{7}$};
\draw (3.5,0) node[scale=1.2000000000000002,inner sep=0pt,minimum size=14pt,below=0pt] {$e_{7}$};
\draw (3.5,1) node[scale=1.2000000000000002,inner sep=0pt,minimum size=14pt,above=0pt] {$e_{8}$};
\draw (1,0.3333333333333333) node[scale=1.2000000000000002,inner sep=0pt,shape=circle,minimum size=3pt,fill=black] {};
\draw (1,0.3333333333333333) node[scale=1.2000000000000002,inner sep=0pt,below=9.895453501482391pt] {$A^{-}$};
\draw (1,0.6666666666666666) node[scale=1.2000000000000002,inner sep=0pt,shape=circle,minimum size=3pt,fill=black] {};
\draw (1,0.6666666666666666) node[scale=1.2000000000000002,inner sep=0pt,above=14.895453501482391pt,right=-4pt] {$A^{+}$};
\draw[] (1,0.3333333333333333) -- (0.5,0.16666666666666666);
\draw[] (0.5,0.16666666666666666) -- (0.5,0);
\draw[] (0,0.8333333333333334) -- (0.5,0.8333333333333334);
\draw[] (0.5,0.8333333333333334) -- (0.5,1);
\draw[] (0,0.6666666666666666) -- (1,0.6666666666666666);
\draw (0.5,0.9166666666666666) node[scale=1.2000000000000002,inner sep=0pt,shape=circle,minimum size=3pt,fill=black] {};
\draw (0.5,0.9166666666666666) node[scale=1.2000000000000002,inner sep=0pt,left=3.8pt] {$t$};
\draw (0.5,0.9166666666666666) node[scale=1.2000000000000002,inner sep=0pt,right=10pt,above=-3.3pt] {$t^{\unaryminus 1}$};
\draw[] (1,0.3333333333333333) -- (1.5,0.16666666666666666);
\draw[] (1.5,0.16666666666666666) -- (1.5,0);
\draw[] (1,0.6666666666666666) -- (1.5,0.8333333333333334);
\draw[] (1.5,0.8333333333333334) -- (1.5,1);
\draw (1.5,0.9166666666666666) node[scale=1.2000000000000002,inner sep=0pt,shape=circle,minimum size=3pt,fill=black] {};
\draw (1.5,0.9166666666666666) node[scale=1.2000000000000002,inner sep=0pt,left=3.8pt] {$t$};
\draw (1.5,0.9166666666666666) node[scale=1.2000000000000002,inner sep=0pt,right=10pt,above=-3.3pt] {$t^{\unaryminus 1}$};
\draw[] (1,0.3333333333333333) -- (2.5,0.16666666666666666);
\draw[] (2.5,0.16666666666666666) -- (2.5,0);
\draw[] (1,0.6666666666666666) -- (2.5,0.8333333333333334);
\draw[] (2.5,0.8333333333333334) -- (2.5,1);
\draw (2.5,0.9166666666666666) node[scale=1.2000000000000002,inner sep=0pt,shape=circle,minimum size=3pt,fill=black] {};
\draw (2.5,0.9166666666666666) node[scale=1.2000000000000002,inner sep=0pt,left=3.8pt] {$t$};
\draw (2.5,0.9166666666666666) node[scale=1.2000000000000002,inner sep=0pt,right=10pt,above=-3.3pt] {$t^{\unaryminus 1}$};
\draw[] (1,0.3333333333333333) -- (3.5,0.16666666666666666);
\draw[] (3.5,0.16666666666666666) -- (3.5,0);
\draw[] (1,0.6666666666666666) -- (3.5,0.8333333333333334);
\draw[] (3.5,0.8333333333333334) -- (3.5,1);
\draw (3.5,0.9166666666666666) node[scale=1.2000000000000002,inner sep=0pt,shape=circle,minimum size=3pt,fill=black] {};
\draw (3.5,0.9166666666666666) node[scale=1.2000000000000002,inner sep=0pt,left=3.8pt] {$t$};
\draw (3.5,0.9166666666666666) node[scale=1.2000000000000002,inner sep=0pt,right=10pt,above=-3.3pt] {$t^{\unaryminus 1}$};
\draw[] (1,0.3333333333333333) -- (1,0.6666666666666666);
\draw (1,0.5) node[scale=1.2000000000000002,inner sep=0pt,shape=circle,minimum size=3pt,fill=black] {};
\draw (1,0.5) node[scale=1.2000000000000002,inner sep=0pt,right=3.8pt] {$t$};
\draw (1,0.5) node[scale=1.2000000000000002,inner sep=0pt,left=8pt,above=-3.3pt] {$t^{\unaryminus 1}$};
}
\newcommand*{\DplusNplusone}{%
\draw[thick] (0,1) -- (4,1);
\draw[thick] (4,0) -- (0,0);
\draw[dashed, thick] (0,0) -- (0,1);
\draw (0,0) node[scale=1.2000000000000002,inner sep=0pt,shape=isosceles triangle,minimum size=3pt,fill=black] {};
\draw (0,1) node[scale=1.2000000000000002,inner sep=0pt,shape=isosceles triangle,minimum size=3pt,fill=black] {};
\draw (1,0) node[scale=1.2000000000000002,inner sep=0pt,shape=isosceles triangle,minimum size=3pt,fill=black] {};
\draw (1,1) node[scale=1.2000000000000002,inner sep=0pt,shape=isosceles triangle,minimum size=3pt,fill=black] {};
\draw (2,0) node[scale=1.2000000000000002,inner sep=0pt,shape=isosceles triangle,minimum size=3pt,fill=black] {};
\draw (2,1) node[scale=1.2000000000000002,inner sep=0pt,shape=isosceles triangle,minimum size=3pt,fill=black] {};
\draw (3,0) node[scale=1.2000000000000002,inner sep=0pt,shape=isosceles triangle,minimum size=3pt,fill=black] {};
\draw (3,1) node[scale=1.2000000000000002,inner sep=0pt,shape=isosceles triangle,minimum size=3pt,fill=black] {};
\draw (4,0) node[scale=1.2000000000000002,inner sep=0pt,shape=isosceles triangle,minimum size=3pt,fill=black] {};
\draw (4,1) node[scale=1.2000000000000002,inner sep=0pt,shape=isosceles triangle,minimum size=3pt,fill=black] {};
\draw (0.5,0) node[scale=1.2000000000000002,inner sep=0pt,minimum size=14pt,below=0pt] {$e_{2}$};
\draw (0.5,1) node[scale=1.2000000000000002,inner sep=0pt,minimum size=14pt,above=0pt] {$e_{1}$};
\draw (1.5,0) node[scale=1.2000000000000002,inner sep=0pt,minimum size=14pt,below=0pt] {$e_{1}$};
\draw (1.5,1) node[scale=1.2000000000000002,inner sep=0pt,minimum size=14pt,above=0pt] {$e_{2}$};
\draw (2.5,0) node[scale=1.2000000000000002,inner sep=0pt,minimum size=14pt,below=0pt] {$e_{4}$};
\draw (2.5,1) node[scale=1.2000000000000002,inner sep=0pt,minimum size=14pt,above=0pt] {$e_{3}$};
\draw (3.5,0) node[scale=1.2000000000000002,inner sep=0pt,minimum size=14pt,below=0pt] {$e_{3}$};
\draw (3.5,1) node[scale=1.2000000000000002,inner sep=0pt,minimum size=14pt,above=0pt] {$e_{4}$};
\draw (1,0.3333333333333333) node[scale=1.2000000000000002,inner sep=0pt,shape=circle,minimum size=3pt,fill=black] {};
\draw (1,0.3333333333333333) node[scale=1.2000000000000002,inner sep=0pt,below=9.895453501482391pt] {$A^{+}$};
\draw (1,0.6666666666666666) node[scale=1.2000000000000002,inner sep=0pt,shape=circle,minimum size=3pt,fill=black] {};
\draw (1,0.6666666666666666) node[scale=1.2000000000000002,inner sep=0pt,above=14.895453501482391pt,right=-4pt] {$A^{-}$};
\draw[] (1,0.3333333333333333) -- (1,0.6666666666666666);
\draw[] (1,0.3333333333333333) -- (0.5,0.16666666666666666);
\draw[] (0.5,0.16666666666666666) -- (0.5,0);
\draw[] (1,0.6666666666666666) -- (0.5,0.8333333333333334);
\draw[] (0.5,0.8333333333333334) -- (0.5,1);
\draw (0.5,0.9166666666666666) node[scale=1.2000000000000002,inner sep=0pt,shape=circle,minimum size=3pt,fill=black] {};
\draw (0.5,0.9166666666666666) node[scale=1.2000000000000002,inner sep=0pt,right=3.8pt] {$t$};
\draw (0.5,0.9166666666666666) node[scale=1.2000000000000002,inner sep=0pt,left=8pt,above=-3.3pt] {$t^{\unaryminus 1}$};
\draw[] (1,0.3333333333333333) -- (1.5,0.16666666666666666);
\draw[] (1.5,0.16666666666666666) -- (1.5,0);
\draw[] (1,0.6666666666666666) -- (1.5,0.8333333333333334);
\draw[] (1.5,0.8333333333333334) -- (1.5,1);
\draw (1.5,0.9166666666666666) node[scale=1.2000000000000002,inner sep=0pt,shape=circle,minimum size=3pt,fill=black] {};
\draw (1.5,0.9166666666666666) node[scale=1.2000000000000002,inner sep=0pt,right=3.8pt] {$t$};
\draw (1.5,0.9166666666666666) node[scale=1.2000000000000002,inner sep=0pt,left=8pt,above=-3.3pt] {$t^{\unaryminus 1}$};
\draw[] (1,0.3333333333333333) -- (2.5,0.16666666666666666);
\draw[] (2.5,0.16666666666666666) -- (2.5,0);
\draw[] (1,0.6666666666666666) -- (2.5,0.8333333333333334);
\draw[] (2.5,0.8333333333333334) -- (2.5,1);
\draw (2.5,0.9166666666666666) node[scale=1.2000000000000002,inner sep=0pt,shape=circle,minimum size=3pt,fill=black] {};
\draw (2.5,0.9166666666666666) node[scale=1.2000000000000002,inner sep=0pt,right=3.8pt] {$t$};
\draw (2.5,0.9166666666666666) node[scale=1.2000000000000002,inner sep=0pt,left=8pt,above=-3.3pt] {$t^{\unaryminus 1}$};
\draw[] (1,0.3333333333333333) -- (3.5,0.16666666666666666);
\draw[] (3.5,0.16666666666666666) -- (3.5,0);
\draw[] (4,0.8333333333333334) -- (3.5,0.8333333333333334);
\draw[] (3.5,0.8333333333333334) -- (3.5,1);
\draw[] (1,0.6666666666666666) -- (4,0.6666666666666666);
\draw (3.5,0.6666666666666666) node[scale=1.2000000000000002,inner sep=0pt,shape=circle,minimum size=3pt,fill=black] {};
\draw (3.5,0.6666666666666666) node[scale=1.2000000000000002,inner sep=0pt,below=3pt] {$t$};
\draw (3.5,0.6666666666666666) node[scale=1.2000000000000002,inner sep=0pt,above=7pt,right=-2.3pt] {$t^{\unaryminus 1}$};
\draw (1,0.5) node[scale=1.2000000000000002,inner sep=0pt,shape=circle,minimum size=3pt,fill=black] {};
\draw (1,0.5) node[scale=1.2000000000000002,inner sep=0pt,left=4pt,above=3pt] {$t$};
\draw (1,0.5) node[scale=1.2000000000000002,inner sep=0pt,left=4pt,below=3pt] {$t$};
\draw (1,0.5) node[scale=1.2000000000000002,inner sep=0pt,right=8.5pt,above=3pt] {$t^{\unaryminus 1}$};
\draw (1,0.5) node[scale=1.2000000000000002,inner sep=0pt,right=8.5pt,below=1pt] {$t^{\unaryminus 1}$};
\draw[] (1,0.5) -- (0.25,0.5);
\draw[] (0.25,0.5) -- (0.25,0);
\draw[] (1,0.5) -- (3.75,0.5);
\draw[] (3.75,0.5) -- (3.75,0);
\draw[] (1.25,1) -- (1.25,0.9583333333333334);
\draw[] (1.25,0.9583333333333334) -- (0.75,0.9583333333333334);
\draw[] (0.75,0.9583333333333334) -- (0.75,1);
\draw[] (3.25,1) -- (3.25,0.9583333333333334);
\draw[] (3.25,0.9583333333333334) -- (2.75,0.9583333333333334);
\draw[] (2.75,0.9583333333333334) -- (2.75,1);
\draw[] (1.75,0) -- (1.75,0.041666666666666664);
\draw[] (1.75,0.041666666666666664) -- (2.25,0.041666666666666664);
\draw[] (2.25,0.041666666666666664) -- (2.25,0);
\draw (2.5,0.5) node[scale=1.2000000000000002,inner sep=0pt,shape=circle,minimum size=3pt,fill=black] {};
\draw (2.5,0.5) node[scale=1.2000000000000002,inner sep=0pt,below=3pt] {$a$};
\draw (2.5,0.5) node[scale=1.2000000000000002,inner sep=0pt,above=7.5pt,right=-3pt] {$a^{\unaryminus 1}$};
}


\begin{scope}[shift={(0,0)}]
    \DplusNplusone
\end{scope}

\begin{scope}[shift={(4,0)}]
    \DrN
\end{scope}

\end{tikzpicture}

%% file: images/D_23_5.tikz
\begin{tikzpicture}[xscale=1.5,yscale=6]
\providecommand{\unaryminus}{\scalebox{0.5}[0.75]{\( - \)}}
\fontsize{10}{12}\selectfont
\newcommand*{\DrN}{%
\draw[thick] (0,1) -- (10,1);
\draw[thick] (10,0) -- (0,0);
\draw (0,0) node[scale=1.2000000000000002,inner sep=0pt,shape=isosceles triangle,minimum size=3pt,fill=black] {};
\draw (0,1) node[scale=1.2000000000000002,inner sep=0pt,shape=isosceles triangle,minimum size=3pt,fill=black] {};
\draw (1,0) node[scale=1.2000000000000002,inner sep=0pt,shape=isosceles triangle,minimum size=3pt,fill=black] {};
\draw (1,1) node[scale=1.2000000000000002,inner sep=0pt,shape=isosceles triangle,minimum size=3pt,fill=black] {};
\draw (2,0) node[scale=1.2000000000000002,inner sep=0pt,shape=isosceles triangle,minimum size=3pt,fill=black] {};
\draw (2,1) node[scale=1.2000000000000002,inner sep=0pt,shape=isosceles triangle,minimum size=3pt,fill=black] {};
\draw (3,0) node[scale=1.2000000000000002,inner sep=0pt,shape=isosceles triangle,minimum size=3pt,fill=black] {};
\draw (3,1) node[scale=1.2000000000000002,inner sep=0pt,shape=isosceles triangle,minimum size=3pt,fill=black] {};
\draw (4,0) node[scale=1.2000000000000002,inner sep=0pt,shape=isosceles triangle,minimum size=3pt,fill=black] {};
\draw (4,1) node[scale=1.2000000000000002,inner sep=0pt,shape=isosceles triangle,minimum size=3pt,fill=black] {};
\draw (5,0) node[scale=1.2000000000000002,inner sep=0pt,shape=isosceles triangle,minimum size=3pt,fill=black] {};
\draw (5,1) node[scale=1.2000000000000002,inner sep=0pt,shape=isosceles triangle,minimum size=3pt,fill=black] {};
\draw (6,0) node[scale=1.2000000000000002,inner sep=0pt,shape=isosceles triangle,minimum size=3pt,fill=black] {};
\draw (6,1) node[scale=1.2000000000000002,inner sep=0pt,shape=isosceles triangle,minimum size=3pt,fill=black] {};
\draw (7,0) node[scale=1.2000000000000002,inner sep=0pt,shape=isosceles triangle,minimum size=3pt,fill=black] {};
\draw (7,1) node[scale=1.2000000000000002,inner sep=0pt,shape=isosceles triangle,minimum size=3pt,fill=black] {};
\draw (8,0) node[scale=1.2000000000000002,inner sep=0pt,shape=isosceles triangle,minimum size=3pt,fill=black] {};
\draw (8,1) node[scale=1.2000000000000002,inner sep=0pt,shape=isosceles triangle,minimum size=3pt,fill=black] {};
\draw (9,0) node[scale=1.2000000000000002,inner sep=0pt,shape=isosceles triangle,minimum size=3pt,fill=black] {};
\draw (9,1) node[scale=1.2000000000000002,inner sep=0pt,shape=isosceles triangle,minimum size=3pt,fill=black] {};
\draw (10,0) node[scale=1.2000000000000002,inner sep=0pt,shape=isosceles triangle,minimum size=3pt,fill=black] {};
\draw (10,1) node[scale=1.2000000000000002,inner sep=0pt,shape=isosceles triangle,minimum size=3pt,fill=black] {};
\draw (0.5,0) node[scale=1.2000000000000002,inner sep=0pt,minimum size=14pt,below=0pt] {$e_{6}$};
\draw (0.5,1) node[scale=1.2000000000000002,inner sep=0pt,minimum size=14pt,above=0pt] {$e_{5}$};
\draw (1.5,0) node[scale=1.2000000000000002,inner sep=0pt,minimum size=14pt,below=0pt] {$e_{5}$};
\draw (1.5,1) node[scale=1.2000000000000002,inner sep=0pt,minimum size=14pt,above=0pt] {$e_{6}$};
\draw (2.5,0) node[scale=1.2000000000000002,inner sep=0pt,minimum size=14pt,below=0pt] {$e_{8}$};
\draw (2.5,1) node[scale=1.2000000000000002,inner sep=0pt,minimum size=14pt,above=0pt] {$e_{7}$};
\draw (3.5,0) node[scale=1.2000000000000002,inner sep=0pt,minimum size=14pt,below=0pt] {$e_{7}$};
\draw (3.5,1) node[scale=1.2000000000000002,inner sep=0pt,minimum size=14pt,above=0pt] {$e_{8}$};
\draw (4.5,0) node[scale=1.2000000000000002,inner sep=0pt,minimum size=14pt,below=0pt] {$e_{10}$};
\draw (4.5,1) node[scale=1.2000000000000002,inner sep=0pt,minimum size=14pt,above=0pt] {$e_{9}$};
\draw (5.5,0) node[scale=1.2000000000000002,inner sep=0pt,minimum size=14pt,below=0pt] {$e_{9}$};
\draw (5.5,1) node[scale=1.2000000000000002,inner sep=0pt,minimum size=14pt,above=0pt] {$e_{10}$};
\draw (6.5,0) node[scale=1.2000000000000002,inner sep=0pt,minimum size=14pt,below=0pt] {$e_{12}$};
\draw (6.5,1) node[scale=1.2000000000000002,inner sep=0pt,minimum size=14pt,above=0pt] {$e_{11}$};
\draw (7.5,0) node[scale=1.2000000000000002,inner sep=0pt,minimum size=14pt,below=0pt] {$e_{11}$};
\draw (7.5,1) node[scale=1.2000000000000002,inner sep=0pt,minimum size=14pt,above=0pt] {$e_{12}$};
\draw (8.5,0) node[scale=1.2000000000000002,inner sep=0pt,minimum size=14pt,below=0pt] {$e_{14}$};
\draw (8.5,1) node[scale=1.2000000000000002,inner sep=0pt,minimum size=14pt,above=0pt] {$e_{13}$};
\draw (9.5,0) node[scale=1.2000000000000002,inner sep=0pt,minimum size=14pt,below=0pt] {$e_{13}$};
\draw (9.5,1) node[scale=1.2000000000000002,inner sep=0pt,minimum size=14pt,above=0pt] {$e_{14}$};
\draw (1,0.3333333333333333) node[scale=1.2000000000000002,inner sep=0pt,shape=circle,minimum size=3pt,fill=black] {};
\draw (1,0.3333333333333333) node[scale=1.2000000000000002,inner sep=0pt,below=9.895453501482391pt] {$A^{-}$};
\draw (1,0.6666666666666666) node[scale=1.2000000000000002,inner sep=0pt,shape=circle,minimum size=3pt,fill=black] {};
\draw (1,0.6666666666666666) node[scale=1.2000000000000002,inner sep=0pt,above=14.895453501482391pt,right=-4pt] {$A^{+}$};
\draw[] (1,0.3333333333333333) -- (0.5,0.16666666666666666);
\draw[] (0.5,0.16666666666666666) -- (0.5,0);
\draw[] (0,0.8333333333333334) -- (0.5,0.8333333333333334);
\draw[] (0.5,0.8333333333333334) -- (0.5,1);
\draw[] (0,0.6666666666666666) -- (1,0.6666666666666666);
\draw (0.5,0.9166666666666666) node[scale=1.2000000000000002,inner sep=0pt,shape=circle,minimum size=3pt,fill=black] {};
\draw (0.5,0.9166666666666666) node[scale=1.2000000000000002,inner sep=0pt,left=3.8pt] {$t$};
\draw (0.5,0.9166666666666666) node[scale=1.2000000000000002,inner sep=0pt,right=10pt,above=-3.3pt] {$t^{\unaryminus 1}$};
\draw[] (1,0.3333333333333333) -- (1.5,0.16666666666666666);
\draw[] (1.5,0.16666666666666666) -- (1.5,0);
\draw[] (1,0.6666666666666666) -- (1.5,0.8333333333333334);
\draw[] (1.5,0.8333333333333334) -- (1.5,1);
\draw (1.5,0.9166666666666666) node[scale=1.2000000000000002,inner sep=0pt,shape=circle,minimum size=3pt,fill=black] {};
\draw (1.5,0.9166666666666666) node[scale=1.2000000000000002,inner sep=0pt,left=3.8pt] {$t$};
\draw (1.5,0.9166666666666666) node[scale=1.2000000000000002,inner sep=0pt,right=10pt,above=-3.3pt] {$t^{\unaryminus 1}$};
\draw[] (1,0.3333333333333333) -- (2.5,0.16666666666666666);
\draw[] (2.5,0.16666666666666666) -- (2.5,0);
\draw[] (1,0.6666666666666666) -- (2.5,0.8333333333333334);
\draw[] (2.5,0.8333333333333334) -- (2.5,1);
\draw (2.5,0.9166666666666666) node[scale=1.2000000000000002,inner sep=0pt,shape=circle,minimum size=3pt,fill=black] {};
\draw (2.5,0.9166666666666666) node[scale=1.2000000000000002,inner sep=0pt,left=3.8pt] {$t$};
\draw (2.5,0.9166666666666666) node[scale=1.2000000000000002,inner sep=0pt,right=10pt,above=-3.3pt] {$t^{\unaryminus 1}$};
\draw[] (1,0.3333333333333333) -- (3.5,0.16666666666666666);
\draw[] (3.5,0.16666666666666666) -- (3.5,0);
\draw[] (1,0.6666666666666666) -- (3.5,0.8333333333333334);
\draw[] (3.5,0.8333333333333334) -- (3.5,1);
\draw (3.5,0.9166666666666666) node[scale=1.2000000000000002,inner sep=0pt,shape=circle,minimum size=3pt,fill=black] {};
\draw (3.5,0.9166666666666666) node[scale=1.2000000000000002,inner sep=0pt,left=3.8pt] {$t$};
\draw (3.5,0.9166666666666666) node[scale=1.2000000000000002,inner sep=0pt,right=10pt,above=-3.3pt] {$t^{\unaryminus 1}$};
\draw (7,0.3333333333333333) node[scale=1.2000000000000002,inner sep=0pt,shape=circle,minimum size=3pt,fill=black] {};
\draw (7,0.3333333333333333) node[scale=1.2000000000000002,inner sep=0pt,below=9.895453501482391pt] {$A^{-}$};
\draw (7,0.6666666666666666) node[scale=1.2000000000000002,inner sep=0pt,shape=circle,minimum size=3pt,fill=black] {};
\draw (7,0.6666666666666666) node[scale=1.2000000000000002,inner sep=0pt,above=14.895453501482391pt,right=-4pt] {$A^{+}$};
\draw[] (7,0.3333333333333333) -- (6.5,0.16666666666666666);
\draw[] (6.5,0.16666666666666666) -- (6.5,0);
\draw[] (7,0.6666666666666666) -- (6.5,0.8333333333333334);
\draw[] (6.5,0.8333333333333334) -- (6.5,1);
\draw (6.5,0.9166666666666666) node[scale=1.2000000000000002,inner sep=0pt,shape=circle,minimum size=3pt,fill=black] {};
\draw (6.5,0.9166666666666666) node[scale=1.2000000000000002,inner sep=0pt,left=3.8pt] {$t$};
\draw (6.5,0.9166666666666666) node[scale=1.2000000000000002,inner sep=0pt,right=10pt,above=-3.3pt] {$t^{\unaryminus 1}$};
\draw[] (7,0.3333333333333333) -- (7.5,0.16666666666666666);
\draw[] (7.5,0.16666666666666666) -- (7.5,0);
\draw[] (7,0.6666666666666666) -- (7.5,0.8333333333333334);
\draw[] (7.5,0.8333333333333334) -- (7.5,1);
\draw (7.5,0.9166666666666666) node[scale=1.2000000000000002,inner sep=0pt,shape=circle,minimum size=3pt,fill=black] {};
\draw (7.5,0.9166666666666666) node[scale=1.2000000000000002,inner sep=0pt,left=3.8pt] {$t$};
\draw (7.5,0.9166666666666666) node[scale=1.2000000000000002,inner sep=0pt,right=10pt,above=-3.3pt] {$t^{\unaryminus 1}$};
\draw[] (7,0.3333333333333333) -- (8.5,0.16666666666666666);
\draw[] (8.5,0.16666666666666666) -- (8.5,0);
\draw[] (7,0.6666666666666666) -- (8.5,0.8333333333333334);
\draw[] (8.5,0.8333333333333334) -- (8.5,1);
\draw (8.5,0.9166666666666666) node[scale=1.2000000000000002,inner sep=0pt,shape=circle,minimum size=3pt,fill=black] {};
\draw (8.5,0.9166666666666666) node[scale=1.2000000000000002,inner sep=0pt,left=3.8pt] {$t$};
\draw (8.5,0.9166666666666666) node[scale=1.2000000000000002,inner sep=0pt,right=10pt,above=-3.3pt] {$t^{\unaryminus 1}$};
\draw[] (7,0.3333333333333333) -- (9.5,0.16666666666666666);
\draw[] (9.5,0.16666666666666666) -- (9.5,0);
\draw[] (10,0.9583333333333334) -- (9.5,0.9583333333333334);
\draw[] (9.5,0.9583333333333334) -- (9.5,1);
\draw[] (7,0.6666666666666666) -- (9.75,0.6666666666666666);
\draw[] (9.75,0.6666666666666666) -- (10,0.5);
\draw (9.5,0.6666666666666666) node[scale=1.2000000000000002,inner sep=0pt,shape=circle,minimum size=3pt,fill=black] {};
\draw (9.5,0.6666666666666666) node[scale=1.2000000000000002,inner sep=0pt,above=3pt] {$t$};
\draw (9.5,0.6666666666666666) node[scale=1.2000000000000002,inner sep=0pt,below=5pt,right=-2.3pt] {$t^{\unaryminus 1}$};
\draw (5,0.3333333333333333) node[scale=1.2000000000000002,inner sep=0pt,shape=circle,minimum size=3pt,fill=black] {};
\draw (5,0.3333333333333333) node[scale=1.2000000000000002,inner sep=0pt,below=9.895453501482391pt] {$A^{+}$};
\draw (5,0.6666666666666666) node[scale=1.2000000000000002,inner sep=0pt,shape=circle,minimum size=3pt,fill=black] {};
\draw (5,0.6666666666666666) node[scale=1.2000000000000002,inner sep=0pt,above=14.895453501482391pt,right=-4pt] {$A^{-}$};
\draw[] (5,0.3333333333333333) -- (5,0.6666666666666666);
\draw (5,0.5) node[scale=1.2000000000000002,inner sep=0pt,shape=circle,minimum size=3pt,fill=black] {};
\draw (5,0.5) node[scale=1.2000000000000002,inner sep=0pt,left=3.8pt] {$t$};
\draw (5,0.5) node[scale=1.2000000000000002,inner sep=0pt,right=10pt,above=-3.3pt] {$t^{\unaryminus 1}$};
\draw[] (5,0.6666666666666666) -- (4.5,0.8333333333333334);
\draw[] (4.5,0.8333333333333334) -- (4.5,1);
\draw[] (5,0.3333333333333333) -- (4.5,0.16666666666666666);
\draw[] (4.5,0.16666666666666666) -- (4.5,0);
\draw (4.5,0.9166666666666666) node[scale=1.2000000000000002,inner sep=0pt,shape=circle,minimum size=3pt,fill=black] {};
\draw (4.5,0.9166666666666666) node[scale=1.2000000000000002,inner sep=0pt,right=3.8pt] {$t$};
\draw (4.5,0.9166666666666666) node[scale=1.2000000000000002,inner sep=0pt,left=8pt,above=-3.3pt] {$t^{\unaryminus 1}$};
\draw[] (5,0.6666666666666666) -- (5.5,0.8333333333333334);
\draw[] (5.5,0.8333333333333334) -- (5.5,1);
\draw[] (5,0.3333333333333333) -- (5.5,0.16666666666666666);
\draw[] (5.5,0.16666666666666666) -- (5.5,0);
\draw (5.5,0.9166666666666666) node[scale=1.2000000000000002,inner sep=0pt,shape=circle,minimum size=3pt,fill=black] {};
\draw (5.5,0.9166666666666666) node[scale=1.2000000000000002,inner sep=0pt,right=3.8pt] {$t$};
\draw (5.5,0.9166666666666666) node[scale=1.2000000000000002,inner sep=0pt,left=8pt,above=-3.3pt] {$t^{\unaryminus 1}$};
\draw[] (1,0.6666666666666666) -- (5,0.6666666666666666);
\draw[] (1,0.3333333333333333) -- (5,0.3333333333333333);
\draw[] (7,0.6666666666666666) -- (5,0.6666666666666666);
\draw[] (7,0.3333333333333333) -- (5,0.3333333333333333);
\draw (3.0,0.6666666666666666) node[scale=1.2000000000000002,inner sep=0pt,shape=circle,minimum size=3pt,fill=black] {};
\draw (3.0,0.6666666666666666) node[scale=1.2000000000000002,inner sep=0pt,above=3pt] {$t$};
\draw (3.0,0.6666666666666666) node[scale=1.2000000000000002,inner sep=0pt,below=5pt,right=-2.3pt] {$t^{\unaryminus 1}$};
\draw (3.0,0.3333333333333333) node[scale=1.2000000000000002,inner sep=0pt,shape=circle,minimum size=3pt,fill=black] {};
\draw (3.0,0.3333333333333333) node[scale=1.2000000000000002,inner sep=0pt,below=3pt] {$t$};
\draw (3.0,0.3333333333333333) node[scale=1.2000000000000002,inner sep=0pt,above=7pt,right=-2.3pt] {$t^{\unaryminus 1}$};
\draw (6.0,0.3333333333333333) node[scale=1.2000000000000002,inner sep=0pt,shape=circle,minimum size=3pt,fill=black] {};
\draw (6.0,0.3333333333333333) node[scale=1.2000000000000002,inner sep=0pt,above=3pt] {$t$};
\draw (6.0,0.3333333333333333) node[scale=1.2000000000000002,inner sep=0pt,below=5pt,right=-2.3pt] {$t^{\unaryminus 1}$};
\draw (6.0,0.6666666666666666) node[scale=1.2000000000000002,inner sep=0pt,shape=circle,minimum size=3pt,fill=black] {};
\draw (6.0,0.6666666666666666) node[scale=1.2000000000000002,inner sep=0pt,below=3pt] {$t$};
\draw (6.0,0.6666666666666666) node[scale=1.2000000000000002,inner sep=0pt,above=7pt,right=-2.3pt] {$t^{\unaryminus 1}$};
}
\newcommand*{\DplusNplusone}{%
\draw[thick] (0,1) -- (4,1);
\draw[thick] (4,0) -- (0,0);
\draw[dashed, thick] (0,0) -- (0,1);
\draw (0,0) node[scale=1.2000000000000002,inner sep=0pt,shape=isosceles triangle,minimum size=3pt,fill=black] {};
\draw (0,1) node[scale=1.2000000000000002,inner sep=0pt,shape=isosceles triangle,minimum size=3pt,fill=black] {};
\draw (1,0) node[scale=1.2000000000000002,inner sep=0pt,shape=isosceles triangle,minimum size=3pt,fill=black] {};
\draw (1,1) node[scale=1.2000000000000002,inner sep=0pt,shape=isosceles triangle,minimum size=3pt,fill=black] {};
\draw (2,0) node[scale=1.2000000000000002,inner sep=0pt,shape=isosceles triangle,minimum size=3pt,fill=black] {};
\draw (2,1) node[scale=1.2000000000000002,inner sep=0pt,shape=isosceles triangle,minimum size=3pt,fill=black] {};
\draw (3,0) node[scale=1.2000000000000002,inner sep=0pt,shape=isosceles triangle,minimum size=3pt,fill=black] {};
\draw (3,1) node[scale=1.2000000000000002,inner sep=0pt,shape=isosceles triangle,minimum size=3pt,fill=black] {};
\draw (4,0) node[scale=1.2000000000000002,inner sep=0pt,shape=isosceles triangle,minimum size=3pt,fill=black] {};
\draw (4,1) node[scale=1.2000000000000002,inner sep=0pt,shape=isosceles triangle,minimum size=3pt,fill=black] {};
\draw (0.5,0) node[scale=1.2000000000000002,inner sep=0pt,minimum size=14pt,below=0pt] {$e_{2}$};
\draw (0.5,1) node[scale=1.2000000000000002,inner sep=0pt,minimum size=14pt,above=0pt] {$e_{1}$};
\draw (1.5,0) node[scale=1.2000000000000002,inner sep=0pt,minimum size=14pt,below=0pt] {$e_{1}$};
\draw (1.5,1) node[scale=1.2000000000000002,inner sep=0pt,minimum size=14pt,above=0pt] {$e_{2}$};
\draw (2.5,0) node[scale=1.2000000000000002,inner sep=0pt,minimum size=14pt,below=0pt] {$e_{4}$};
\draw (2.5,1) node[scale=1.2000000000000002,inner sep=0pt,minimum size=14pt,above=0pt] {$e_{3}$};
\draw (3.5,0) node[scale=1.2000000000000002,inner sep=0pt,minimum size=14pt,below=0pt] {$e_{3}$};
\draw (3.5,1) node[scale=1.2000000000000002,inner sep=0pt,minimum size=14pt,above=0pt] {$e_{4}$};
\draw (1,0.3333333333333333) node[scale=1.2000000000000002,inner sep=0pt,shape=circle,minimum size=3pt,fill=black] {};
\draw (1,0.3333333333333333) node[scale=1.2000000000000002,inner sep=0pt,below=9.895453501482391pt] {$A^{+}$};
\draw (1,0.6666666666666666) node[scale=1.2000000000000002,inner sep=0pt,shape=circle,minimum size=3pt,fill=black] {};
\draw (1,0.6666666666666666) node[scale=1.2000000000000002,inner sep=0pt,above=14.895453501482391pt,right=-4pt] {$A^{-}$};
\draw[] (1,0.3333333333333333) -- (1,0.6666666666666666);
\draw[] (1,0.3333333333333333) -- (0.5,0.16666666666666666);
\draw[] (0.5,0.16666666666666666) -- (0.5,0);
\draw[] (1,0.6666666666666666) -- (0.5,0.8333333333333334);
\draw[] (0.5,0.8333333333333334) -- (0.5,1);
\draw (0.5,0.9166666666666666) node[scale=1.2000000000000002,inner sep=0pt,shape=circle,minimum size=3pt,fill=black] {};
\draw (0.5,0.9166666666666666) node[scale=1.2000000000000002,inner sep=0pt,right=3.8pt] {$t$};
\draw (0.5,0.9166666666666666) node[scale=1.2000000000000002,inner sep=0pt,left=8pt,above=-3.3pt] {$t^{\unaryminus 1}$};
\draw[] (1,0.3333333333333333) -- (1.5,0.16666666666666666);
\draw[] (1.5,0.16666666666666666) -- (1.5,0);
\draw[] (1,0.6666666666666666) -- (1.5,0.8333333333333334);
\draw[] (1.5,0.8333333333333334) -- (1.5,1);
\draw (1.5,0.9166666666666666) node[scale=1.2000000000000002,inner sep=0pt,shape=circle,minimum size=3pt,fill=black] {};
\draw (1.5,0.9166666666666666) node[scale=1.2000000000000002,inner sep=0pt,right=3.8pt] {$t$};
\draw (1.5,0.9166666666666666) node[scale=1.2000000000000002,inner sep=0pt,left=8pt,above=-3.3pt] {$t^{\unaryminus 1}$};
\draw[] (1,0.3333333333333333) -- (2.5,0.16666666666666666);
\draw[] (2.5,0.16666666666666666) -- (2.5,0);
\draw[] (1,0.6666666666666666) -- (2.5,0.8333333333333334);
\draw[] (2.5,0.8333333333333334) -- (2.5,1);
\draw (2.5,0.9166666666666666) node[scale=1.2000000000000002,inner sep=0pt,shape=circle,minimum size=3pt,fill=black] {};
\draw (2.5,0.9166666666666666) node[scale=1.2000000000000002,inner sep=0pt,right=3.8pt] {$t$};
\draw (2.5,0.9166666666666666) node[scale=1.2000000000000002,inner sep=0pt,left=8pt,above=-3.3pt] {$t^{\unaryminus 1}$};
\draw[] (1,0.3333333333333333) -- (3.5,0.16666666666666666);
\draw[] (3.5,0.16666666666666666) -- (3.5,0);
\draw[] (4,0.8333333333333334) -- (3.5,0.8333333333333334);
\draw[] (3.5,0.8333333333333334) -- (3.5,1);
\draw[] (1,0.6666666666666666) -- (4,0.6666666666666666);
\draw (3.5,0.6666666666666666) node[scale=1.2000000000000002,inner sep=0pt,shape=circle,minimum size=3pt,fill=black] {};
\draw (3.5,0.6666666666666666) node[scale=1.2000000000000002,inner sep=0pt,below=3pt] {$t$};
\draw (3.5,0.6666666666666666) node[scale=1.2000000000000002,inner sep=0pt,above=7pt,right=-2.3pt] {$t^{\unaryminus 1}$};
\draw (1,0.5) node[scale=1.2000000000000002,inner sep=0pt,shape=circle,minimum size=3pt,fill=black] {};
\draw (1,0.5) node[scale=1.2000000000000002,inner sep=0pt,left=4pt,above=3pt] {$t$};
\draw (1,0.5) node[scale=1.2000000000000002,inner sep=0pt,left=4pt,below=3pt] {$t$};
\draw (1,0.5) node[scale=1.2000000000000002,inner sep=0pt,right=8.5pt,above=3pt] {$t^{\unaryminus 1}$};
\draw (1,0.5) node[scale=1.2000000000000002,inner sep=0pt,right=8.5pt,below=1pt] {$t^{\unaryminus 1}$};
\draw[] (1,0.5) -- (0.25,0.5);
\draw[] (0.25,0.5) -- (0.25,0);
\draw[] (1,0.5) -- (3.75,0.5);
\draw[] (3.75,0.5) -- (3.75,0);
\draw[] (1.25,1) -- (1.25,0.9583333333333334);
\draw[] (1.25,0.9583333333333334) -- (0.75,0.9583333333333334);
\draw[] (0.75,0.9583333333333334) -- (0.75,1);
\draw[] (3.25,1) -- (3.25,0.9583333333333334);
\draw[] (3.25,0.9583333333333334) -- (2.75,0.9583333333333334);
\draw[] (2.75,0.9583333333333334) -- (2.75,1);
\draw[] (1.75,0) -- (1.75,0.041666666666666664);
\draw[] (1.75,0.041666666666666664) -- (2.25,0.041666666666666664);
\draw[] (2.25,0.041666666666666664) -- (2.25,0);
\draw (2.5,0.5) node[scale=1.2000000000000002,inner sep=0pt,shape=circle,minimum size=3pt,fill=black] {};
\draw (2.5,0.5) node[scale=1.2000000000000002,inner sep=0pt,below=3pt] {$a$};
\draw (2.5,0.5) node[scale=1.2000000000000002,inner sep=0pt,above=7.5pt,right=-3pt] {$a^{\unaryminus 1}$};
}
\newcommand*{\Dplustwos}{%
\draw[thick] (0,1) -- (2,1);
\draw[thick] (2,0) -- (0,0);
\draw[dashed, thick] (2,1) -- (2,0);
\draw (0,0) node[scale=1.2000000000000002,inner sep=0pt,shape=isosceles triangle,minimum size=3pt,fill=black] {};
\draw (0,1) node[scale=1.2000000000000002,inner sep=0pt,shape=isosceles triangle,minimum size=3pt,fill=black] {};
\draw (1,0) node[scale=1.2000000000000002,inner sep=0pt,shape=isosceles triangle,minimum size=3pt,fill=black] {};
\draw (1,1) node[scale=1.2000000000000002,inner sep=0pt,shape=isosceles triangle,minimum size=3pt,fill=black] {};
\draw (2,0) node[scale=1.2000000000000002,inner sep=0pt,shape=isosceles triangle,minimum size=3pt,fill=black] {};
\draw (2,1) node[scale=1.2000000000000002,inner sep=0pt,shape=isosceles triangle,minimum size=3pt,fill=black] {};
\draw (0.5,0) node[scale=1.2000000000000002,inner sep=0pt,minimum size=14pt,below=0pt] {$e_{16}$};
\draw (0.5,1) node[scale=1.2000000000000002,inner sep=0pt,minimum size=14pt,above=0pt] {$e_{15}$};
\draw (1.5,0) node[scale=1.2000000000000002,inner sep=0pt,minimum size=14pt,below=0pt] {$e_{15}$};
\draw (1.5,1) node[scale=1.2000000000000002,inner sep=0pt,minimum size=14pt,above=0pt] {$e_{16}$};
\draw[] (0,0.9583333333333334) -- (0.25,0.9583333333333334);
\draw[] (0.25,0.9583333333333334) -- (0.25,1);
\draw[] (0,0.5) -- (0.25,0.3333333333333333);
\draw (1,0.3333333333333333) node[scale=1.2000000000000002,inner sep=0pt,shape=circle,minimum size=3pt,fill=black] {};
\draw (1,0.3333333333333333) node[scale=1.2000000000000002,inner sep=0pt,right=6pt,above=3pt] {$a$};
\draw (1,0.3333333333333333) node[scale=1.2000000000000002,inner sep=0pt,right=0pt,below=6.123105625617663pt] {$a$};
\draw (1,0.3333333333333333) node[scale=1.2000000000000002,inner sep=0pt,left=7pt,above=3pt] {$a^{\unaryminus 1}$};
\draw (1,0.3333333333333333) node[scale=1.2000000000000002,inner sep=0pt,left=12.80776406404415pt,below=1pt] {$a^{\unaryminus 1}$};
\draw (1,0.6666666666666666) node[scale=1.2000000000000002,inner sep=0pt,shape=circle,minimum size=3pt,fill=black] {};
\draw (1,0.6666666666666666) node[scale=1.2000000000000002,inner sep=0pt,right=10.05397531527947pt,above=2pt] {$t$};
\draw (1,0.6666666666666666) node[scale=1.2000000000000002,inner sep=0pt,left=0pt,above=7.123105625617663pt] {$t$};
\draw (1,0.6666666666666666) node[scale=1.2000000000000002,inner sep=0pt,right=9pt,below=2.5pt] {$t^{\unaryminus 1}$};
\draw (1,0.6666666666666666) node[scale=1.2000000000000002,inner sep=0pt,left=7pt,below=2.5pt] {$t^{\unaryminus 1}$};
\draw[] (1,0.3333333333333333) -- (1,0.6666666666666666);
\draw[] (1,0.6666666666666666) -- (0.5,0.8333333333333334);
\draw[] (0.5,0.8333333333333334) -- (0.5,1);
\draw[] (1,0.3333333333333333) -- (0.5,0.16666666666666666);
\draw[] (0.5,0.16666666666666666) -- (0.5,0);
\draw[] (1,0.6666666666666666) -- (1.5,0.8333333333333334);
\draw[] (1.5,0.8333333333333334) -- (1.5,1);
\draw[] (1,0.3333333333333333) -- (1.5,0.16666666666666666);
\draw[] (1.5,0.16666666666666666) -- (1.5,0);
\draw[] (0.25,0.3333333333333333) -- (1,0.3333333333333333);
\draw[] (1.75,0.6666666666666666) -- (1,0.6666666666666666);
\draw[] (0.75,0) -- (0.75,0.041666666666666664);
\draw[] (0.75,0.041666666666666664) -- (1.25,0.041666666666666664);
\draw[] (1.25,0.041666666666666664) -- (1.25,0);
\draw[] (1.75,0.6666666666666666) -- (1.75,1);
}


\begin{scope}[shift={(0,0)}]
    \DplusNplusone
\end{scope}

\begin{scope}[shift={(4,0)}]
    \DrN
\end{scope}

\begin{scope}[shift={(14,0)}]
    \Dplustwos
\end{scope}

\end{tikzpicture}

%% file: images/D_2_2.tikz
\begin{tikzpicture}[xscale=1.5,yscale=6]
\providecommand{\unaryminus}{\scalebox{0.5}[0.75]{\( - \)}}
\fontsize{10}{12}\selectfont
\newcommand*{\DrN}{%
\draw[thick] (0,1) -- (2,1);
\draw[thick] (2,0) -- (0,0);
\draw[dashed, thick] (0,0) -- (0,1);
\draw[dashed, thick] (2,1) -- (2,0);
\draw (0,0) node[scale=1.2000000000000002,inner sep=0pt,shape=isosceles triangle,minimum size=3pt,fill=black] {};
\draw (0,1) node[scale=1.2000000000000002,inner sep=0pt,shape=isosceles triangle,minimum size=3pt,fill=black] {};
\draw (1,0) node[scale=1.2000000000000002,inner sep=0pt,shape=isosceles triangle,minimum size=3pt,fill=black] {};
\draw (1,1) node[scale=1.2000000000000002,inner sep=0pt,shape=isosceles triangle,minimum size=3pt,fill=black] {};
\draw (2,0) node[scale=1.2000000000000002,inner sep=0pt,shape=isosceles triangle,minimum size=3pt,fill=black] {};
\draw (2,1) node[scale=1.2000000000000002,inner sep=0pt,shape=isosceles triangle,minimum size=3pt,fill=black] {};
\draw (0.5,0) node[scale=1.2000000000000002,inner sep=0pt,minimum size=14pt,below=0pt] {$e_{2}$};
\draw (0.5,1) node[scale=1.2000000000000002,inner sep=0pt,minimum size=14pt,above=0pt] {$e_{1}$};
\draw (1.5,0) node[scale=1.2000000000000002,inner sep=0pt,minimum size=14pt,below=0pt] {$e_{1}$};
\draw (1.5,1) node[scale=1.2000000000000002,inner sep=0pt,minimum size=14pt,above=0pt] {$e_{2}$};
\draw (1.5,0.3333333333333333) node[scale=1.2000000000000002,inner sep=0pt,shape=circle,minimum size=3pt,fill=black] {};
\draw (1.5,0.3333333333333333) node[scale=1.2000000000000002,inner sep=0pt,below=6pt,left=1pt] {$A^{-}$};
\draw (1.5,0.6666666666666666) node[scale=1.2000000000000002,inner sep=0pt,shape=circle,minimum size=3pt,fill=black] {};
\draw (1.5,0.6666666666666666) node[scale=1.2000000000000002,inner sep=0pt,above=7pt,left=1pt] {$A^{+}$};
\draw[] (0.5,1) -- (0.5,0.6666666666666666);
\draw[] (0.5,0.6666666666666666) -- (1,0.5);
\draw[] (1,0.5) -- (1.5,0.6666666666666666);
\draw[] (1.5,0.6666666666666666) -- (1.5,1);
\draw[] (0.5,0) -- (0.5,0.3333333333333333);
\draw[] (0.5,0.3333333333333333) -- (1,0.5);
\draw[] (1,0.5) -- (1.5,0.3333333333333333);
\draw[] (1.5,0.3333333333333333) -- (1.5,0);
\draw (1,0.5) node[scale=1.2000000000000002,inner sep=0pt,shape=circle,minimum size=3pt,fill=black] {};
\draw (1,0.5) node[scale=1.2000000000000002,inner sep=0pt,left=6pt] {$t$};
\draw (1,0.5) node[scale=1.2000000000000002,inner sep=0pt,right=6pt] {$t$};
\draw (1,0.5) node[scale=1.2000000000000002,inner sep=0pt,right=2pt,above=8.246211251235327pt] {$t^{\unaryminus 1}$};
\draw (1,0.5) node[scale=1.2000000000000002,inner sep=0pt,below=8.246211251235327pt] {$t^{\unaryminus 1}$};
\draw[] (1.5,0) -- (1.5,0.16666666666666666);
}


\begin{scope}[shift={(0,0)}]
    \DrN
\end{scope}

\end{tikzpicture}

%% file: images/D_4_4.tikz
\begin{tikzpicture}[xscale=1.5,yscale=6]
\providecommand{\unaryminus}{\scalebox{0.5}[0.75]{\( - \)}}
\fontsize{10}{12}\selectfont
\newcommand*{\DrN}{%
\draw[thick] (0,1) -- (4,1);
\draw[thick] (4,0) -- (0,0);
\draw[dashed, thick] (0,0) -- (0,1);
\draw[dashed, thick] (4,1) -- (4,0);
\draw (0,0) node[scale=1.2000000000000002,inner sep=0pt,shape=isosceles triangle,minimum size=3pt,fill=black] {};
\draw (0,1) node[scale=1.2000000000000002,inner sep=0pt,shape=isosceles triangle,minimum size=3pt,fill=black] {};
\draw (1,0) node[scale=1.2000000000000002,inner sep=0pt,shape=isosceles triangle,minimum size=3pt,fill=black] {};
\draw (1,1) node[scale=1.2000000000000002,inner sep=0pt,shape=isosceles triangle,minimum size=3pt,fill=black] {};
\draw (2,0) node[scale=1.2000000000000002,inner sep=0pt,shape=isosceles triangle,minimum size=3pt,fill=black] {};
\draw (2,1) node[scale=1.2000000000000002,inner sep=0pt,shape=isosceles triangle,minimum size=3pt,fill=black] {};
\draw (3,0) node[scale=1.2000000000000002,inner sep=0pt,shape=isosceles triangle,minimum size=3pt,fill=black] {};
\draw (3,1) node[scale=1.2000000000000002,inner sep=0pt,shape=isosceles triangle,minimum size=3pt,fill=black] {};
\draw (4,0) node[scale=1.2000000000000002,inner sep=0pt,shape=isosceles triangle,minimum size=3pt,fill=black] {};
\draw (4,1) node[scale=1.2000000000000002,inner sep=0pt,shape=isosceles triangle,minimum size=3pt,fill=black] {};
\draw (0.5,0) node[scale=1.2000000000000002,inner sep=0pt,minimum size=14pt,below=0pt] {$e_{2}$};
\draw (0.5,1) node[scale=1.2000000000000002,inner sep=0pt,minimum size=14pt,above=0pt] {$e_{1}$};
\draw (1.5,0) node[scale=1.2000000000000002,inner sep=0pt,minimum size=14pt,below=0pt] {$e_{1}$};
\draw (1.5,1) node[scale=1.2000000000000002,inner sep=0pt,minimum size=14pt,above=0pt] {$e_{2}$};
\draw (2.5,0) node[scale=1.2000000000000002,inner sep=0pt,minimum size=14pt,below=0pt] {$e_{4}$};
\draw (2.5,1) node[scale=1.2000000000000002,inner sep=0pt,minimum size=14pt,above=0pt] {$e_{3}$};
\draw (3.5,0) node[scale=1.2000000000000002,inner sep=0pt,minimum size=14pt,below=0pt] {$e_{3}$};
\draw (3.5,1) node[scale=1.2000000000000002,inner sep=0pt,minimum size=14pt,above=0pt] {$e_{4}$};
\draw (1.5,0.3333333333333333) node[scale=1.2000000000000002,inner sep=0pt,shape=circle,minimum size=3pt,fill=black] {};
\draw (1.5,0.3333333333333333) node[scale=1.2000000000000002,inner sep=0pt,below=6pt,left=1pt] {$A^{-}$};
\draw (1.5,0.6666666666666666) node[scale=1.2000000000000002,inner sep=0pt,shape=circle,minimum size=3pt,fill=black] {};
\draw (1.5,0.6666666666666666) node[scale=1.2000000000000002,inner sep=0pt,above=7pt,left=1pt] {$A^{+}$};
\draw[] (0.5,1) -- (0.5,0.6666666666666666);
\draw[] (0.5,0.6666666666666666) -- (1,0.5);
\draw[] (1,0.5) -- (1.5,0.6666666666666666);
\draw[] (1.5,0.6666666666666666) -- (1.5,1);
\draw[] (0.5,0) -- (0.5,0.3333333333333333);
\draw[] (0.5,0.3333333333333333) -- (1,0.5);
\draw[] (1,0.5) -- (1.5,0.3333333333333333);
\draw[] (1.5,0.3333333333333333) -- (1.5,0);
\draw (1,0.5) node[scale=1.2000000000000002,inner sep=0pt,shape=circle,minimum size=3pt,fill=black] {};
\draw (1,0.5) node[scale=1.2000000000000002,inner sep=0pt,left=6pt] {$t$};
\draw (1,0.5) node[scale=1.2000000000000002,inner sep=0pt,right=6pt] {$t$};
\draw (1,0.5) node[scale=1.2000000000000002,inner sep=0pt,right=2pt,above=8.246211251235327pt] {$t^{\unaryminus 1}$};
\draw (1,0.5) node[scale=1.2000000000000002,inner sep=0pt,below=8.246211251235327pt] {$t^{\unaryminus 1}$};
\draw[] (2.5,1) -- (2.5,0.8333333333333334);
\draw[] (2.5,0.8333333333333334) -- (1.5,0.6666666666666666);
\draw[] (2.5,0) -- (2.5,0.16666666666666666);
\draw[] (2.5,0.16666666666666666) -- (1.5,0.3333333333333333);
\draw (2.5,0.9166666666666666) node[scale=1.2000000000000002,inner sep=0pt,shape=circle,minimum size=3pt,fill=black] {};
\draw (2.5,0.9166666666666666) node[scale=1.2000000000000002,inner sep=0pt,left=3.8pt] {$t$};
\draw (2.5,0.9166666666666666) node[scale=1.2000000000000002,inner sep=0pt,right=10pt,above=-3.3pt] {$t^{\unaryminus 1}$};
\draw[] (3.5,0) -- (3.5,0.16666666666666666);
\draw[] (3.5,0.16666666666666666) -- (1.5,0.3333333333333333);
\draw[] (3.5,1) -- (3.5,0.8333333333333334);
\draw[] (3.5,0.8333333333333334) -- (1.5,0.6666666666666666);
\draw (3.5,0.9166666666666666) node[scale=1.2000000000000002,inner sep=0pt,shape=circle,minimum size=3pt,fill=black] {};
\draw (3.5,0.9166666666666666) node[scale=1.2000000000000002,inner sep=0pt,left=3.8pt] {$t$};
\draw (3.5,0.9166666666666666) node[scale=1.2000000000000002,inner sep=0pt,right=10pt,above=-3.3pt] {$t^{\unaryminus 1}$};
}


\begin{scope}[shift={(0,0)}]
    \DrN
\end{scope}

\end{tikzpicture}

%% file: images/D_8_4.tikz
\begin{tikzpicture}[xscale=1.5,yscale=6]
\providecommand{\unaryminus}{\scalebox{0.5}[0.75]{\( - \)}}
\fontsize{10}{12}\selectfont
\newcommand*{\DrN}{%
\draw[thick] (0,1) -- (6,1);
\draw[thick] (6,0) -- (0,0);
\draw[dashed, thick] (0,0) -- (0,1);
\draw[dashed, thick] (6,1) -- (6,0);
\draw (0,0) node[scale=1.2000000000000002,inner sep=0pt,shape=isosceles triangle,minimum size=3pt,fill=black] {};
\draw (0,1) node[scale=1.2000000000000002,inner sep=0pt,shape=isosceles triangle,minimum size=3pt,fill=black] {};
\draw (1,0) node[scale=1.2000000000000002,inner sep=0pt,shape=isosceles triangle,minimum size=3pt,fill=black] {};
\draw (1,1) node[scale=1.2000000000000002,inner sep=0pt,shape=isosceles triangle,minimum size=3pt,fill=black] {};
\draw (2,0) node[scale=1.2000000000000002,inner sep=0pt,shape=isosceles triangle,minimum size=3pt,fill=black] {};
\draw (2,1) node[scale=1.2000000000000002,inner sep=0pt,shape=isosceles triangle,minimum size=3pt,fill=black] {};
\draw (3,0) node[scale=1.2000000000000002,inner sep=0pt,shape=isosceles triangle,minimum size=3pt,fill=black] {};
\draw (3,1) node[scale=1.2000000000000002,inner sep=0pt,shape=isosceles triangle,minimum size=3pt,fill=black] {};
\draw (4,0) node[scale=1.2000000000000002,inner sep=0pt,shape=isosceles triangle,minimum size=3pt,fill=black] {};
\draw (4,1) node[scale=1.2000000000000002,inner sep=0pt,shape=isosceles triangle,minimum size=3pt,fill=black] {};
\draw (5,0) node[scale=1.2000000000000002,inner sep=0pt,shape=isosceles triangle,minimum size=3pt,fill=black] {};
\draw (5,1) node[scale=1.2000000000000002,inner sep=0pt,shape=isosceles triangle,minimum size=3pt,fill=black] {};
\draw (6,0) node[scale=1.2000000000000002,inner sep=0pt,shape=isosceles triangle,minimum size=3pt,fill=black] {};
\draw (6,1) node[scale=1.2000000000000002,inner sep=0pt,shape=isosceles triangle,minimum size=3pt,fill=black] {};
\draw (0.5,0) node[scale=1.2000000000000002,inner sep=0pt,minimum size=14pt,below=0pt] {$e_{2}$};
\draw (0.5,1) node[scale=1.2000000000000002,inner sep=0pt,minimum size=14pt,above=0pt] {$e_{1}$};
\draw (1.5,0) node[scale=1.2000000000000002,inner sep=0pt,minimum size=14pt,below=0pt] {$e_{1}$};
\draw (1.5,1) node[scale=1.2000000000000002,inner sep=0pt,minimum size=14pt,above=0pt] {$e_{2}$};
\draw (2.5,0) node[scale=1.2000000000000002,inner sep=0pt,minimum size=14pt,below=0pt] {$e_{4}$};
\draw (2.5,1) node[scale=1.2000000000000002,inner sep=0pt,minimum size=14pt,above=0pt] {$e_{3}$};
\draw (3.5,0) node[scale=1.2000000000000002,inner sep=0pt,minimum size=14pt,below=0pt] {$e_{3}$};
\draw (3.5,1) node[scale=1.2000000000000002,inner sep=0pt,minimum size=14pt,above=0pt] {$e_{4}$};
\draw (4.5,0) node[scale=1.2000000000000002,inner sep=0pt,minimum size=14pt,below=0pt] {$e_{6}$};
\draw (4.5,1) node[scale=1.2000000000000002,inner sep=0pt,minimum size=14pt,above=0pt] {$e_{5}$};
\draw (5.5,0) node[scale=1.2000000000000002,inner sep=0pt,minimum size=14pt,below=0pt] {$e_{5}$};
\draw (5.5,1) node[scale=1.2000000000000002,inner sep=0pt,minimum size=14pt,above=0pt] {$e_{6}$};
\draw (1.5,0.3333333333333333) node[scale=1.2000000000000002,inner sep=0pt,shape=circle,minimum size=3pt,fill=black] {};
\draw (1.5,0.3333333333333333) node[scale=1.2000000000000002,inner sep=0pt,below=6pt,left=1pt] {$A^{-}$};
\draw (1.5,0.6666666666666666) node[scale=1.2000000000000002,inner sep=0pt,shape=circle,minimum size=3pt,fill=black] {};
\draw (1.5,0.6666666666666666) node[scale=1.2000000000000002,inner sep=0pt,above=7pt,left=1pt] {$A^{+}$};
\draw[] (0.5,1) -- (0.5,0.6666666666666666);
\draw[] (0.5,0.6666666666666666) -- (1,0.5);
\draw[] (1,0.5) -- (1.5,0.6666666666666666);
\draw[] (1.5,0.6666666666666666) -- (1.5,1);
\draw[] (0.5,0) -- (0.5,0.3333333333333333);
\draw[] (0.5,0.3333333333333333) -- (1,0.5);
\draw[] (1,0.5) -- (1.5,0.3333333333333333);
\draw[] (1.5,0.3333333333333333) -- (1.5,0);
\draw (1,0.5) node[scale=1.2000000000000002,inner sep=0pt,shape=circle,minimum size=3pt,fill=black] {};
\draw (1,0.5) node[scale=1.2000000000000002,inner sep=0pt,left=6pt] {$t$};
\draw (1,0.5) node[scale=1.2000000000000002,inner sep=0pt,right=6pt] {$t$};
\draw (1,0.5) node[scale=1.2000000000000002,inner sep=0pt,right=2pt,above=8.246211251235327pt] {$t^{\unaryminus 1}$};
\draw (1,0.5) node[scale=1.2000000000000002,inner sep=0pt,below=8.246211251235327pt] {$t^{\unaryminus 1}$};
\draw[] (2.5,1) -- (2.5,0.8333333333333334);
\draw[] (2.5,0.8333333333333334) -- (1.5,0.6666666666666666);
\draw[] (2.5,0) -- (2.5,0.16666666666666666);
\draw[] (2.5,0.16666666666666666) -- (1.5,0.3333333333333333);
\draw (2.5,0.9166666666666666) node[scale=1.2000000000000002,inner sep=0pt,shape=circle,minimum size=3pt,fill=black] {};
\draw (2.5,0.9166666666666666) node[scale=1.2000000000000002,inner sep=0pt,left=3.8pt] {$t$};
\draw (2.5,0.9166666666666666) node[scale=1.2000000000000002,inner sep=0pt,right=10pt,above=-3.3pt] {$t^{\unaryminus 1}$};
\draw[] (3.5,0) -- (3.5,0.16666666666666666);
\draw[] (3.5,1) -- (3.5,0.8333333333333334);
\draw[] (3.5,0.8333333333333334) -- (1.5,0.6666666666666666);
\draw (3.5,0.9166666666666666) node[scale=1.2000000000000002,inner sep=0pt,shape=circle,minimum size=3pt,fill=black] {};
\draw (3.5,0.9166666666666666) node[scale=1.2000000000000002,inner sep=0pt,left=3.8pt] {$t$};
\draw (3.5,0.9166666666666666) node[scale=1.2000000000000002,inner sep=0pt,right=10pt,above=-3.3pt] {$t^{\unaryminus 1}$};
\draw[] (1.5,0.3333333333333333) -- (3,0.3333333333333333);
\draw[] (3,0.3333333333333333) -- (4,0.5);
\draw[] (3.5,0.16666666666666666) -- (4,0.2222222222222222);
\draw (5,0.3333333333333333) node[scale=1.2000000000000002,inner sep=0pt,shape=circle,minimum size=3pt,fill=black] {};
\draw (5,0.3333333333333333) node[scale=1.2000000000000002,inner sep=0pt,below=9.895453501482391pt] {$A^{-}$};
\draw (5,0.6666666666666666) node[scale=1.2000000000000002,inner sep=0pt,shape=circle,minimum size=3pt,fill=black] {};
\draw (5,0.6666666666666666) node[scale=1.2000000000000002,inner sep=0pt,above=14.895453501482391pt,right=-4pt] {$A^{+}$};
\draw[] (5,0.3333333333333333) -- (5,0.6666666666666666);
\draw (5,0.5) node[scale=1.2000000000000002,inner sep=0pt,shape=circle,minimum size=3pt,fill=black] {};
\draw (5,0.5) node[scale=1.2000000000000002,inner sep=0pt,right=3.8pt] {$t$};
\draw (5,0.5) node[scale=1.2000000000000002,inner sep=0pt,left=8pt,above=-3.3pt] {$t^{\unaryminus 1}$};
\draw[] (5,0.3333333333333333) -- (4.5,0.16666666666666666);
\draw[] (4.5,0.16666666666666666) -- (4.5,0);
\draw[] (5,0.6666666666666666) -- (4.5,0.8333333333333334);
\draw[] (4.5,0.8333333333333334) -- (4.5,1);
\draw (4.5,0.9166666666666666) node[scale=1.2000000000000002,inner sep=0pt,shape=circle,minimum size=3pt,fill=black] {};
\draw (4.5,0.9166666666666666) node[scale=1.2000000000000002,inner sep=0pt,left=3.8pt] {$t$};
\draw (4.5,0.9166666666666666) node[scale=1.2000000000000002,inner sep=0pt,right=10pt,above=-3.3pt] {$t^{\unaryminus 1}$};
\draw[] (5,0.6666666666666666) -- (4,0.5);
\draw[] (5,0.3333333333333333) -- (4,0.2222222222222222);
\draw (4,0.5) node[scale=1.2000000000000002,inner sep=0pt,shape=circle,minimum size=3pt,fill=black] {};
\draw (4,0.5) node[scale=1.2000000000000002,inner sep=0pt,right=2.496150883013531pt,below=3.7442263245202967pt] {$t$};
\draw (4,0.5) node[scale=1.2000000000000002,inner sep=0pt,left=0.49615088301353083pt,above=3.7442263245202967pt,] {$t^{\unaryminus 1}$};
\draw[] (5,0.3333333333333333) -- (5.5,0.16666666666666666);
\draw[] (5.5,0.16666666666666666) -- (5.5,0);
\draw[] (5,0.6666666666666666) -- (5.5,0.8333333333333334);
\draw[] (5.5,0.8333333333333334) -- (5.5,1);
\draw (5.5,0.9166666666666666) node[scale=1.2000000000000002,inner sep=0pt,shape=circle,minimum size=3pt,fill=black] {};
\draw (5.5,0.9166666666666666) node[scale=1.2000000000000002,inner sep=0pt,left=3.8pt] {$t$};
\draw (5.5,0.9166666666666666) node[scale=1.2000000000000002,inner sep=0pt,right=10pt,above=-3.3pt] {$t^{\unaryminus 1}$};
}


\begin{scope}[shift={(0,0)}]
    \DrN
\end{scope}

\end{tikzpicture}

%% file: images/D_6_4.tikz
\begin{tikzpicture}[xscale=1.5,yscale=6]
\providecommand{\unaryminus}{\scalebox{0.5}[0.75]{\( - \)}}
\fontsize{10}{12}\selectfont
\newcommand*{\DrN}{%
\draw[thick] (0,1) -- (4,1);
\draw[thick] (4,0) -- (0,0);
\draw[dashed, thick] (0,0) -- (0,1);
\draw (0,0) node[scale=1.2000000000000002,inner sep=0pt,shape=isosceles triangle,minimum size=3pt,fill=black] {};
\draw (0,1) node[scale=1.2000000000000002,inner sep=0pt,shape=isosceles triangle,minimum size=3pt,fill=black] {};
\draw (1,0) node[scale=1.2000000000000002,inner sep=0pt,shape=isosceles triangle,minimum size=3pt,fill=black] {};
\draw (1,1) node[scale=1.2000000000000002,inner sep=0pt,shape=isosceles triangle,minimum size=3pt,fill=black] {};
\draw (2,0) node[scale=1.2000000000000002,inner sep=0pt,shape=isosceles triangle,minimum size=3pt,fill=black] {};
\draw (2,1) node[scale=1.2000000000000002,inner sep=0pt,shape=isosceles triangle,minimum size=3pt,fill=black] {};
\draw (3,0) node[scale=1.2000000000000002,inner sep=0pt,shape=isosceles triangle,minimum size=3pt,fill=black] {};
\draw (3,1) node[scale=1.2000000000000002,inner sep=0pt,shape=isosceles triangle,minimum size=3pt,fill=black] {};
\draw (4,0) node[scale=1.2000000000000002,inner sep=0pt,shape=isosceles triangle,minimum size=3pt,fill=black] {};
\draw (4,1) node[scale=1.2000000000000002,inner sep=0pt,shape=isosceles triangle,minimum size=3pt,fill=black] {};
\draw (0.5,0) node[scale=1.2000000000000002,inner sep=0pt,minimum size=14pt,below=0pt] {$e_{2}$};
\draw (0.5,1) node[scale=1.2000000000000002,inner sep=0pt,minimum size=14pt,above=0pt] {$e_{1}$};
\draw (1.5,0) node[scale=1.2000000000000002,inner sep=0pt,minimum size=14pt,below=0pt] {$e_{1}$};
\draw (1.5,1) node[scale=1.2000000000000002,inner sep=0pt,minimum size=14pt,above=0pt] {$e_{2}$};
\draw (2.5,0) node[scale=1.2000000000000002,inner sep=0pt,minimum size=14pt,below=0pt] {$e_{4}$};
\draw (2.5,1) node[scale=1.2000000000000002,inner sep=0pt,minimum size=14pt,above=0pt] {$e_{3}$};
\draw (3.5,0) node[scale=1.2000000000000002,inner sep=0pt,minimum size=14pt,below=0pt] {$e_{3}$};
\draw (3.5,1) node[scale=1.2000000000000002,inner sep=0pt,minimum size=14pt,above=0pt] {$e_{4}$};
\draw (1.5,0.3333333333333333) node[scale=1.2000000000000002,inner sep=0pt,shape=circle,minimum size=3pt,fill=black] {};
\draw (1.5,0.3333333333333333) node[scale=1.2000000000000002,inner sep=0pt,below=6pt,left=1pt] {$A^{-}$};
\draw (1.5,0.6666666666666666) node[scale=1.2000000000000002,inner sep=0pt,shape=circle,minimum size=3pt,fill=black] {};
\draw (1.5,0.6666666666666666) node[scale=1.2000000000000002,inner sep=0pt,above=7pt,left=1pt] {$A^{+}$};
\draw[] (0.5,1) -- (0.5,0.6666666666666666);
\draw[] (0.5,0.6666666666666666) -- (1,0.5);
\draw[] (1,0.5) -- (1.5,0.6666666666666666);
\draw[] (1.5,0.6666666666666666) -- (1.5,1);
\draw[] (0.5,0) -- (0.5,0.3333333333333333);
\draw[] (0.5,0.3333333333333333) -- (1,0.5);
\draw[] (1,0.5) -- (1.5,0.3333333333333333);
\draw[] (1.5,0.3333333333333333) -- (1.5,0);
\draw (1,0.5) node[scale=1.2000000000000002,inner sep=0pt,shape=circle,minimum size=3pt,fill=black] {};
\draw (1,0.5) node[scale=1.2000000000000002,inner sep=0pt,left=6pt] {$t$};
\draw (1,0.5) node[scale=1.2000000000000002,inner sep=0pt,right=6pt] {$t$};
\draw (1,0.5) node[scale=1.2000000000000002,inner sep=0pt,right=2pt,above=8.246211251235327pt] {$t^{\unaryminus 1}$};
\draw (1,0.5) node[scale=1.2000000000000002,inner sep=0pt,below=8.246211251235327pt] {$t^{\unaryminus 1}$};
\draw[] (2.5,1) -- (2.5,0.8333333333333334);
\draw[] (2.5,0.8333333333333334) -- (1.5,0.6666666666666666);
\draw[] (2.5,0) -- (2.5,0.16666666666666666);
\draw[] (2.5,0.16666666666666666) -- (1.5,0.3333333333333333);
\draw (2.5,0.9166666666666666) node[scale=1.2000000000000002,inner sep=0pt,shape=circle,minimum size=3pt,fill=black] {};
\draw (2.5,0.9166666666666666) node[scale=1.2000000000000002,inner sep=0pt,left=3.8pt] {$t$};
\draw (2.5,0.9166666666666666) node[scale=1.2000000000000002,inner sep=0pt,right=10pt,above=-3.3pt] {$t^{\unaryminus 1}$};
\draw[] (3.5,0) -- (3.5,0.16666666666666666);
\draw[] (3.5,0.16666666666666666) -- (1.5,0.3333333333333333);
\draw[] (4,0.9583333333333334) -- (3.5,0.9583333333333334);
\draw[] (3.5,0.9583333333333334) -- (3.5,1);
\draw[] (1.5,0.6666666666666666) -- (3.75,0.6666666666666666);
\draw[] (3.75,0.6666666666666666) -- (4,0.5);
\draw (3.5,0.6666666666666666) node[scale=1.2000000000000002,inner sep=0pt,shape=circle,minimum size=3pt,fill=black] {};
\draw (3.5,0.6666666666666666) node[scale=1.2000000000000002,inner sep=0pt,above=3pt] {$t$};
\draw (3.5,0.6666666666666666) node[scale=1.2000000000000002,inner sep=0pt,below=5pt,right=-2.3pt] {$t^{\unaryminus 1}$};
}
\newcommand*{\Dplustwos}{%
\draw[thick] (0,1) -- (2,1);
\draw[thick] (2,0) -- (0,0);
\draw[dashed, thick] (2,1) -- (2,0);
\draw (0,0) node[scale=1.2000000000000002,inner sep=0pt,shape=isosceles triangle,minimum size=3pt,fill=black] {};
\draw (0,1) node[scale=1.2000000000000002,inner sep=0pt,shape=isosceles triangle,minimum size=3pt,fill=black] {};
\draw (1,0) node[scale=1.2000000000000002,inner sep=0pt,shape=isosceles triangle,minimum size=3pt,fill=black] {};
\draw (1,1) node[scale=1.2000000000000002,inner sep=0pt,shape=isosceles triangle,minimum size=3pt,fill=black] {};
\draw (2,0) node[scale=1.2000000000000002,inner sep=0pt,shape=isosceles triangle,minimum size=3pt,fill=black] {};
\draw (2,1) node[scale=1.2000000000000002,inner sep=0pt,shape=isosceles triangle,minimum size=3pt,fill=black] {};
\draw (0.5,0) node[scale=1.2000000000000002,inner sep=0pt,minimum size=14pt,below=0pt] {$e_{6}$};
\draw (0.5,1) node[scale=1.2000000000000002,inner sep=0pt,minimum size=14pt,above=0pt] {$e_{5}$};
\draw (1.5,0) node[scale=1.2000000000000002,inner sep=0pt,minimum size=14pt,below=0pt] {$e_{5}$};
\draw (1.5,1) node[scale=1.2000000000000002,inner sep=0pt,minimum size=14pt,above=0pt] {$e_{6}$};
\draw[] (0,0.9583333333333334) -- (0.25,0.9583333333333334);
\draw[] (0.25,0.9583333333333334) -- (0.25,1);
\draw[] (0,0.5) -- (0.25,0.3333333333333333);
\draw (1,0.3333333333333333) node[scale=1.2000000000000002,inner sep=0pt,shape=circle,minimum size=3pt,fill=black] {};
\draw (1,0.3333333333333333) node[scale=1.2000000000000002,inner sep=0pt,right=6pt,above=3pt] {$a$};
\draw (1,0.3333333333333333) node[scale=1.2000000000000002,inner sep=0pt,right=0pt,below=6.123105625617663pt] {$a$};
\draw (1,0.3333333333333333) node[scale=1.2000000000000002,inner sep=0pt,left=7pt,above=3pt] {$a^{\unaryminus 1}$};
\draw (1,0.3333333333333333) node[scale=1.2000000000000002,inner sep=0pt,left=12.80776406404415pt,below=1pt] {$a^{\unaryminus 1}$};
\draw (1,0.6666666666666666) node[scale=1.2000000000000002,inner sep=0pt,shape=circle,minimum size=3pt,fill=black] {};
\draw (1,0.6666666666666666) node[scale=1.2000000000000002,inner sep=0pt,right=10.05397531527947pt,above=2pt] {$t$};
\draw (1,0.6666666666666666) node[scale=1.2000000000000002,inner sep=0pt,left=0pt,above=7.123105625617663pt] {$t$};
\draw (1,0.6666666666666666) node[scale=1.2000000000000002,inner sep=0pt,right=9pt,below=2.5pt] {$t^{\unaryminus 1}$};
\draw (1,0.6666666666666666) node[scale=1.2000000000000002,inner sep=0pt,left=7pt,below=2.5pt] {$t^{\unaryminus 1}$};
\draw[] (1,0.3333333333333333) -- (1,0.6666666666666666);
\draw[] (1,0.6666666666666666) -- (0.5,0.8333333333333334);
\draw[] (0.5,0.8333333333333334) -- (0.5,1);
\draw[] (1,0.3333333333333333) -- (0.5,0.16666666666666666);
\draw[] (0.5,0.16666666666666666) -- (0.5,0);
\draw[] (1,0.6666666666666666) -- (1.5,0.8333333333333334);
\draw[] (1.5,0.8333333333333334) -- (1.5,1);
\draw[] (1,0.3333333333333333) -- (1.5,0.16666666666666666);
\draw[] (1.5,0.16666666666666666) -- (1.5,0);
\draw[] (0.25,0.3333333333333333) -- (1,0.3333333333333333);
\draw[] (1.75,0.6666666666666666) -- (1,0.6666666666666666);
\draw[] (0.75,0) -- (0.75,0.041666666666666664);
\draw[] (0.75,0.041666666666666664) -- (1.25,0.041666666666666664);
\draw[] (1.25,0.041666666666666664) -- (1.25,0);
\draw[] (1.75,0.6666666666666666) -- (1.75,1);
}


\begin{scope}[shift={(0,0)}]
    \DrN
\end{scope}

\begin{scope}[shift={(4,0)}]
    \Dplustwos
\end{scope}

\end{tikzpicture}

%% file: images/D_14_4.tikz
\begin{tikzpicture}[xscale=1.5,yscale=6]
\providecommand{\unaryminus}{\scalebox{0.5}[0.75]{\( - \)}}
\fontsize{10}{12}\selectfont
\newcommand*{\DrN}{%
\draw[thick] (0,1) -- (8,1);
\draw[thick] (8,0) -- (0,0);
\draw[dashed, thick] (0,0) -- (0,1);
\draw (0,0) node[scale=1.2000000000000002,inner sep=0pt,shape=isosceles triangle,minimum size=3pt,fill=black] {};
\draw (0,1) node[scale=1.2000000000000002,inner sep=0pt,shape=isosceles triangle,minimum size=3pt,fill=black] {};
\draw (1,0) node[scale=1.2000000000000002,inner sep=0pt,shape=isosceles triangle,minimum size=3pt,fill=black] {};
\draw (1,1) node[scale=1.2000000000000002,inner sep=0pt,shape=isosceles triangle,minimum size=3pt,fill=black] {};
\draw (2,0) node[scale=1.2000000000000002,inner sep=0pt,shape=isosceles triangle,minimum size=3pt,fill=black] {};
\draw (2,1) node[scale=1.2000000000000002,inner sep=0pt,shape=isosceles triangle,minimum size=3pt,fill=black] {};
\draw (3,0) node[scale=1.2000000000000002,inner sep=0pt,shape=isosceles triangle,minimum size=3pt,fill=black] {};
\draw (3,1) node[scale=1.2000000000000002,inner sep=0pt,shape=isosceles triangle,minimum size=3pt,fill=black] {};
\draw (4,0) node[scale=1.2000000000000002,inner sep=0pt,shape=isosceles triangle,minimum size=3pt,fill=black] {};
\draw (4,1) node[scale=1.2000000000000002,inner sep=0pt,shape=isosceles triangle,minimum size=3pt,fill=black] {};
\draw (5,0) node[scale=1.2000000000000002,inner sep=0pt,shape=isosceles triangle,minimum size=3pt,fill=black] {};
\draw (5,1) node[scale=1.2000000000000002,inner sep=0pt,shape=isosceles triangle,minimum size=3pt,fill=black] {};
\draw (6,0) node[scale=1.2000000000000002,inner sep=0pt,shape=isosceles triangle,minimum size=3pt,fill=black] {};
\draw (6,1) node[scale=1.2000000000000002,inner sep=0pt,shape=isosceles triangle,minimum size=3pt,fill=black] {};
\draw (7,0) node[scale=1.2000000000000002,inner sep=0pt,shape=isosceles triangle,minimum size=3pt,fill=black] {};
\draw (7,1) node[scale=1.2000000000000002,inner sep=0pt,shape=isosceles triangle,minimum size=3pt,fill=black] {};
\draw (8,0) node[scale=1.2000000000000002,inner sep=0pt,shape=isosceles triangle,minimum size=3pt,fill=black] {};
\draw (8,1) node[scale=1.2000000000000002,inner sep=0pt,shape=isosceles triangle,minimum size=3pt,fill=black] {};
\draw (0.5,0) node[scale=1.2000000000000002,inner sep=0pt,minimum size=14pt,below=0pt] {$e_{2}$};
\draw (0.5,1) node[scale=1.2000000000000002,inner sep=0pt,minimum size=14pt,above=0pt] {$e_{1}$};
\draw (1.5,0) node[scale=1.2000000000000002,inner sep=0pt,minimum size=14pt,below=0pt] {$e_{1}$};
\draw (1.5,1) node[scale=1.2000000000000002,inner sep=0pt,minimum size=14pt,above=0pt] {$e_{2}$};
\draw (2.5,0) node[scale=1.2000000000000002,inner sep=0pt,minimum size=14pt,below=0pt] {$e_{4}$};
\draw (2.5,1) node[scale=1.2000000000000002,inner sep=0pt,minimum size=14pt,above=0pt] {$e_{3}$};
\draw (3.5,0) node[scale=1.2000000000000002,inner sep=0pt,minimum size=14pt,below=0pt] {$e_{3}$};
\draw (3.5,1) node[scale=1.2000000000000002,inner sep=0pt,minimum size=14pt,above=0pt] {$e_{4}$};
\draw (4.5,0) node[scale=1.2000000000000002,inner sep=0pt,minimum size=14pt,below=0pt] {$e_{6}$};
\draw (4.5,1) node[scale=1.2000000000000002,inner sep=0pt,minimum size=14pt,above=0pt] {$e_{5}$};
\draw (5.5,0) node[scale=1.2000000000000002,inner sep=0pt,minimum size=14pt,below=0pt] {$e_{5}$};
\draw (5.5,1) node[scale=1.2000000000000002,inner sep=0pt,minimum size=14pt,above=0pt] {$e_{6}$};
\draw (6.5,0) node[scale=1.2000000000000002,inner sep=0pt,minimum size=14pt,below=0pt] {$e_{8}$};
\draw (6.5,1) node[scale=1.2000000000000002,inner sep=0pt,minimum size=14pt,above=0pt] {$e_{7}$};
\draw (7.5,0) node[scale=1.2000000000000002,inner sep=0pt,minimum size=14pt,below=0pt] {$e_{7}$};
\draw (7.5,1) node[scale=1.2000000000000002,inner sep=0pt,minimum size=14pt,above=0pt] {$e_{8}$};
\draw (1.5,0.3333333333333333) node[scale=1.2000000000000002,inner sep=0pt,shape=circle,minimum size=3pt,fill=black] {};
\draw (1.5,0.3333333333333333) node[scale=1.2000000000000002,inner sep=0pt,below=6pt,left=1pt] {$A^{-}$};
\draw (1.5,0.6666666666666666) node[scale=1.2000000000000002,inner sep=0pt,shape=circle,minimum size=3pt,fill=black] {};
\draw (1.5,0.6666666666666666) node[scale=1.2000000000000002,inner sep=0pt,above=7pt,left=1pt] {$A^{+}$};
\draw[] (0.5,1) -- (0.5,0.6666666666666666);
\draw[] (0.5,0.6666666666666666) -- (1,0.5);
\draw[] (1,0.5) -- (1.5,0.6666666666666666);
\draw[] (1.5,0.6666666666666666) -- (1.5,1);
\draw[] (0.5,0) -- (0.5,0.3333333333333333);
\draw[] (0.5,0.3333333333333333) -- (1,0.5);
\draw[] (1,0.5) -- (1.5,0.3333333333333333);
\draw[] (1.5,0.3333333333333333) -- (1.5,0);
\draw (1,0.5) node[scale=1.2000000000000002,inner sep=0pt,shape=circle,minimum size=3pt,fill=black] {};
\draw (1,0.5) node[scale=1.2000000000000002,inner sep=0pt,left=6pt] {$t$};
\draw (1,0.5) node[scale=1.2000000000000002,inner sep=0pt,right=6pt] {$t$};
\draw (1,0.5) node[scale=1.2000000000000002,inner sep=0pt,right=2pt,above=8.246211251235327pt] {$t^{\unaryminus 1}$};
\draw (1,0.5) node[scale=1.2000000000000002,inner sep=0pt,below=8.246211251235327pt] {$t^{\unaryminus 1}$};
\draw[] (2.5,1) -- (2.5,0.8333333333333334);
\draw[] (2.5,0.8333333333333334) -- (1.5,0.6666666666666666);
\draw[] (2.5,0) -- (2.5,0.16666666666666666);
\draw[] (2.5,0.16666666666666666) -- (1.5,0.3333333333333333);
\draw (2.5,0.9166666666666666) node[scale=1.2000000000000002,inner sep=0pt,shape=circle,minimum size=3pt,fill=black] {};
\draw (2.5,0.9166666666666666) node[scale=1.2000000000000002,inner sep=0pt,left=3.8pt] {$t$};
\draw (2.5,0.9166666666666666) node[scale=1.2000000000000002,inner sep=0pt,right=10pt,above=-3.3pt] {$t^{\unaryminus 1}$};
\draw[] (3.5,0) -- (3.5,0.16666666666666666);
\draw[] (3.5,1) -- (3.5,0.8333333333333334);
\draw[] (3.5,0.8333333333333334) -- (1.5,0.6666666666666666);
\draw (3.5,0.9166666666666666) node[scale=1.2000000000000002,inner sep=0pt,shape=circle,minimum size=3pt,fill=black] {};
\draw (3.5,0.9166666666666666) node[scale=1.2000000000000002,inner sep=0pt,left=3.8pt] {$t$};
\draw (3.5,0.9166666666666666) node[scale=1.2000000000000002,inner sep=0pt,right=10pt,above=-3.3pt] {$t^{\unaryminus 1}$};
\draw[] (1.5,0.3333333333333333) -- (3,0.3333333333333333);
\draw[] (3,0.3333333333333333) -- (4,0.5);
\draw[] (3.5,0.16666666666666666) -- (4,0.2222222222222222);
\draw (5,0.3333333333333333) node[scale=1.2000000000000002,inner sep=0pt,shape=circle,minimum size=3pt,fill=black] {};
\draw (5,0.3333333333333333) node[scale=1.2000000000000002,inner sep=0pt,below=9.895453501482391pt] {$A^{-}$};
\draw (5,0.6666666666666666) node[scale=1.2000000000000002,inner sep=0pt,shape=circle,minimum size=3pt,fill=black] {};
\draw (5,0.6666666666666666) node[scale=1.2000000000000002,inner sep=0pt,above=14.895453501482391pt,right=-4pt] {$A^{+}$};
\draw[] (5,0.3333333333333333) -- (5,0.6666666666666666);
\draw (5,0.5) node[scale=1.2000000000000002,inner sep=0pt,shape=circle,minimum size=3pt,fill=black] {};
\draw (5,0.5) node[scale=1.2000000000000002,inner sep=0pt,right=3.8pt] {$t$};
\draw (5,0.5) node[scale=1.2000000000000002,inner sep=0pt,left=8pt,above=-3.3pt] {$t^{\unaryminus 1}$};
\draw[] (5,0.3333333333333333) -- (4.5,0.16666666666666666);
\draw[] (4.5,0.16666666666666666) -- (4.5,0);
\draw[] (5,0.6666666666666666) -- (4.5,0.8333333333333334);
\draw[] (4.5,0.8333333333333334) -- (4.5,1);
\draw (4.5,0.9166666666666666) node[scale=1.2000000000000002,inner sep=0pt,shape=circle,minimum size=3pt,fill=black] {};
\draw (4.5,0.9166666666666666) node[scale=1.2000000000000002,inner sep=0pt,left=3.8pt] {$t$};
\draw (4.5,0.9166666666666666) node[scale=1.2000000000000002,inner sep=0pt,right=10pt,above=-3.3pt] {$t^{\unaryminus 1}$};
\draw[] (5,0.6666666666666666) -- (4,0.5);
\draw[] (5,0.3333333333333333) -- (4,0.2222222222222222);
\draw (4,0.5) node[scale=1.2000000000000002,inner sep=0pt,shape=circle,minimum size=3pt,fill=black] {};
\draw (4,0.5) node[scale=1.2000000000000002,inner sep=0pt,right=2.496150883013531pt,below=3.7442263245202967pt] {$t$};
\draw (4,0.5) node[scale=1.2000000000000002,inner sep=0pt,left=0.49615088301353083pt,above=3.7442263245202967pt,] {$t^{\unaryminus 1}$};
\draw[] (5.5,0.16666666666666666) -- (5.5,0);
\draw[] (5,0.6666666666666666) -- (5.5,0.8333333333333334);
\draw[] (5.5,0.8333333333333334) -- (5.5,1);
\draw (5.5,0.9166666666666666) node[scale=1.2000000000000002,inner sep=0pt,shape=circle,minimum size=3pt,fill=black] {};
\draw (5.5,0.9166666666666666) node[scale=1.2000000000000002,inner sep=0pt,left=3.8pt] {$t$};
\draw (5.5,0.9166666666666666) node[scale=1.2000000000000002,inner sep=0pt,right=10pt,above=-3.3pt] {$t^{\unaryminus 1}$};
\draw[] (5,0.3333333333333333) -- (6,0.5);
\draw[] (5.5,0.16666666666666666) -- (6,0.2222222222222222);
\draw (7,0.3333333333333333) node[scale=1.2000000000000002,inner sep=0pt,shape=circle,minimum size=3pt,fill=black] {};
\draw (7,0.3333333333333333) node[scale=1.2000000000000002,inner sep=0pt,below=9.895453501482391pt] {$A^{-}$};
\draw (7,0.6666666666666666) node[scale=1.2000000000000002,inner sep=0pt,shape=circle,minimum size=3pt,fill=black] {};
\draw (7,0.6666666666666666) node[scale=1.2000000000000002,inner sep=0pt,above=14.895453501482391pt,right=-4pt] {$A^{+}$};
\draw[] (7,0.3333333333333333) -- (7,0.6666666666666666);
\draw (7,0.5) node[scale=1.2000000000000002,inner sep=0pt,shape=circle,minimum size=3pt,fill=black] {};
\draw (7,0.5) node[scale=1.2000000000000002,inner sep=0pt,right=3.8pt] {$t$};
\draw (7,0.5) node[scale=1.2000000000000002,inner sep=0pt,left=8pt,above=-3.3pt] {$t^{\unaryminus 1}$};
\draw[] (7,0.3333333333333333) -- (6.5,0.16666666666666666);
\draw[] (6.5,0.16666666666666666) -- (6.5,0);
\draw[] (7,0.6666666666666666) -- (6.5,0.8333333333333334);
\draw[] (6.5,0.8333333333333334) -- (6.5,1);
\draw (6.5,0.9166666666666666) node[scale=1.2000000000000002,inner sep=0pt,shape=circle,minimum size=3pt,fill=black] {};
\draw (6.5,0.9166666666666666) node[scale=1.2000000000000002,inner sep=0pt,left=3.8pt] {$t$};
\draw (6.5,0.9166666666666666) node[scale=1.2000000000000002,inner sep=0pt,right=10pt,above=-3.3pt] {$t^{\unaryminus 1}$};
\draw[] (7,0.6666666666666666) -- (6,0.5);
\draw[] (7,0.3333333333333333) -- (6,0.2222222222222222);
\draw (6,0.5) node[scale=1.2000000000000002,inner sep=0pt,shape=circle,minimum size=3pt,fill=black] {};
\draw (6,0.5) node[scale=1.2000000000000002,inner sep=0pt,right=2.496150883013531pt,below=3.7442263245202967pt] {$t$};
\draw (6,0.5) node[scale=1.2000000000000002,inner sep=0pt,left=0.49615088301353083pt,above=3.7442263245202967pt,] {$t^{\unaryminus 1}$};
\draw[] (7,0.3333333333333333) -- (7.5,0.16666666666666666);
\draw[] (7.5,0.16666666666666666) -- (7.5,0);
\draw[] (8,0.9583333333333334) -- (7.5,0.9583333333333334);
\draw[] (7.5,0.9583333333333334) -- (7.5,1);
\draw[] (7,0.6666666666666666) -- (7.75,0.6666666666666666);
\draw[] (7.75,0.6666666666666666) -- (8,0.5);
\draw (7.5,0.6666666666666666) node[scale=1.2000000000000002,inner sep=0pt,shape=circle,minimum size=3pt,fill=black] {};
\draw (7.5,0.6666666666666666) node[scale=1.2000000000000002,inner sep=0pt,above=3pt] {$t$};
\draw (7.5,0.6666666666666666) node[scale=1.2000000000000002,inner sep=0pt,below=5pt,right=-2.3pt] {$t^{\unaryminus 1}$};
}
\newcommand*{\Dplustwos}{%
\draw[thick] (0,1) -- (2,1);
\draw[thick] (2,0) -- (0,0);
\draw[dashed, thick] (2,1) -- (2,0);
\draw (0,0) node[scale=1.2000000000000002,inner sep=0pt,shape=isosceles triangle,minimum size=3pt,fill=black] {};
\draw (0,1) node[scale=1.2000000000000002,inner sep=0pt,shape=isosceles triangle,minimum size=3pt,fill=black] {};
\draw (1,0) node[scale=1.2000000000000002,inner sep=0pt,shape=isosceles triangle,minimum size=3pt,fill=black] {};
\draw (1,1) node[scale=1.2000000000000002,inner sep=0pt,shape=isosceles triangle,minimum size=3pt,fill=black] {};
\draw (2,0) node[scale=1.2000000000000002,inner sep=0pt,shape=isosceles triangle,minimum size=3pt,fill=black] {};
\draw (2,1) node[scale=1.2000000000000002,inner sep=0pt,shape=isosceles triangle,minimum size=3pt,fill=black] {};
\draw (0.5,0) node[scale=1.2000000000000002,inner sep=0pt,minimum size=14pt,below=0pt] {$e_{10}$};
\draw (0.5,1) node[scale=1.2000000000000002,inner sep=0pt,minimum size=14pt,above=0pt] {$e_{9}$};
\draw (1.5,0) node[scale=1.2000000000000002,inner sep=0pt,minimum size=14pt,below=0pt] {$e_{9}$};
\draw (1.5,1) node[scale=1.2000000000000002,inner sep=0pt,minimum size=14pt,above=0pt] {$e_{10}$};
\draw[] (0,0.9583333333333334) -- (0.25,0.9583333333333334);
\draw[] (0.25,0.9583333333333334) -- (0.25,1);
\draw[] (0,0.5) -- (0.25,0.3333333333333333);
\draw (1,0.3333333333333333) node[scale=1.2000000000000002,inner sep=0pt,shape=circle,minimum size=3pt,fill=black] {};
\draw (1,0.3333333333333333) node[scale=1.2000000000000002,inner sep=0pt,right=6pt,above=3pt] {$a$};
\draw (1,0.3333333333333333) node[scale=1.2000000000000002,inner sep=0pt,right=0pt,below=6.123105625617663pt] {$a$};
\draw (1,0.3333333333333333) node[scale=1.2000000000000002,inner sep=0pt,left=7pt,above=3pt] {$a^{\unaryminus 1}$};
\draw (1,0.3333333333333333) node[scale=1.2000000000000002,inner sep=0pt,left=12.80776406404415pt,below=1pt] {$a^{\unaryminus 1}$};
\draw (1,0.6666666666666666) node[scale=1.2000000000000002,inner sep=0pt,shape=circle,minimum size=3pt,fill=black] {};
\draw (1,0.6666666666666666) node[scale=1.2000000000000002,inner sep=0pt,right=10.05397531527947pt,above=2pt] {$t$};
\draw (1,0.6666666666666666) node[scale=1.2000000000000002,inner sep=0pt,left=0pt,above=7.123105625617663pt] {$t$};
\draw (1,0.6666666666666666) node[scale=1.2000000000000002,inner sep=0pt,right=9pt,below=2.5pt] {$t^{\unaryminus 1}$};
\draw (1,0.6666666666666666) node[scale=1.2000000000000002,inner sep=0pt,left=7pt,below=2.5pt] {$t^{\unaryminus 1}$};
\draw[] (1,0.3333333333333333) -- (1,0.6666666666666666);
\draw[] (1,0.6666666666666666) -- (0.5,0.8333333333333334);
\draw[] (0.5,0.8333333333333334) -- (0.5,1);
\draw[] (1,0.3333333333333333) -- (0.5,0.16666666666666666);
\draw[] (0.5,0.16666666666666666) -- (0.5,0);
\draw[] (1,0.6666666666666666) -- (1.5,0.8333333333333334);
\draw[] (1.5,0.8333333333333334) -- (1.5,1);
\draw[] (1,0.3333333333333333) -- (1.5,0.16666666666666666);
\draw[] (1.5,0.16666666666666666) -- (1.5,0);
\draw[] (0.25,0.3333333333333333) -- (1,0.3333333333333333);
\draw[] (1.75,0.6666666666666666) -- (1,0.6666666666666666);
\draw[] (0.75,0) -- (0.75,0.041666666666666664);
\draw[] (0.75,0.041666666666666664) -- (1.25,0.041666666666666664);
\draw[] (1.25,0.041666666666666664) -- (1.25,0);
\draw[] (1.75,0.6666666666666666) -- (1.75,1);
}


\begin{scope}[shift={(0,0)}]
    \DrN
\end{scope}

\begin{scope}[shift={(8,0)}]
    \Dplustwos
\end{scope}

\end{tikzpicture}

%% file: images/D_3_2.tikz
\begin{tikzpicture}[xscale=1.5,yscale=6]
\providecommand{\unaryminus}{\scalebox{0.5}[0.75]{\( - \)}}
\fontsize{10}{12}\selectfont
\newcommand*{\DrN}{%
\draw[thick] (0,1) -- (2,1);
\draw[thick] (2,0) -- (0,0);
\draw[dashed, thick] (0,0) -- (0,1);
\draw[dashed, thick] (2,1) -- (2,0);
\draw (0,0) node[scale=1.2000000000000002,inner sep=0pt,shape=isosceles triangle,minimum size=3pt,fill=black] {};
\draw (0,1) node[scale=1.2000000000000002,inner sep=0pt,shape=isosceles triangle,minimum size=3pt,fill=black] {};
\draw (1,0) node[scale=1.2000000000000002,inner sep=0pt,shape=isosceles triangle,minimum size=3pt,fill=black] {};
\draw (1,1) node[scale=1.2000000000000002,inner sep=0pt,shape=isosceles triangle,minimum size=3pt,fill=black] {};
\draw (2,0) node[scale=1.2000000000000002,inner sep=0pt,shape=isosceles triangle,minimum size=3pt,fill=black] {};
\draw (2,1) node[scale=1.2000000000000002,inner sep=0pt,shape=isosceles triangle,minimum size=3pt,fill=black] {};
\draw (0.5,0) node[scale=1.2000000000000002,inner sep=0pt,minimum size=14pt,below=0pt] {$e_{2}$};
\draw (0.5,1) node[scale=1.2000000000000002,inner sep=0pt,minimum size=14pt,above=0pt] {$e_{1}$};
\draw (1.5,0) node[scale=1.2000000000000002,inner sep=0pt,minimum size=14pt,below=0pt] {$e_{1}$};
\draw (1.5,1) node[scale=1.2000000000000002,inner sep=0pt,minimum size=14pt,above=0pt] {$e_{2}$};
\draw (1.5,0.3333333333333333) node[scale=1.2000000000000002,inner sep=0pt,shape=circle,minimum size=3pt,fill=black] {};
\draw (1.5,0.3333333333333333) node[scale=1.2000000000000002,inner sep=0pt,above=7pt,left=1pt] {$A^{-}$};
\draw (1.5,0.6666666666666666) node[scale=1.2000000000000002,inner sep=0pt,shape=circle,minimum size=3pt,fill=black] {};
\draw (1.5,0.6666666666666666) node[scale=1.2000000000000002,inner sep=0pt,above=7pt,left=1pt] {$A^{+}$};
\draw (1.5,0.5) node[scale=1.2000000000000002,inner sep=0pt,shape=circle,minimum size=3pt,fill=black] {};
\draw (1.5,0.5) node[scale=1.2000000000000002,inner sep=0pt,right=3.8pt] {$t$};
\draw (1.5,0.5) node[scale=1.2000000000000002,inner sep=0pt,left=8pt,above=-3.3pt] {$t^{\unaryminus 1}$};
\draw (0.5,0.6666666666666666) node[scale=1.2000000000000002,inner sep=0pt,shape=circle,minimum size=3pt,fill=black] {};
\draw (0.5,0.6666666666666666) node[scale=1.2000000000000002,inner sep=0pt,right=4pt] {$a$};
\draw (0.5,0.6666666666666666) node[scale=1.2000000000000002,inner sep=0pt,above=1.9pt,left=1.5pt] {$a^{\unaryminus 1}$};
\draw[] (0.5,1) -- (0.5,0.3333333333333333);
\draw[] (0.5,0.3333333333333333) -- (0.6666666666666667,0.3333333333333333);
\draw[] (0.6666666666666667,0.3333333333333333) -- (1.3333333333333333,0.16666666666666666);
\draw[] (1.3333333333333333,0.16666666666666666) -- (1.5,0.16666666666666666);
\draw[] (1.5,0.16666666666666666) -- (1.5,0);
\draw[] (1.5,1) -- (1.5,0.3333333333333333);
\draw[] (1.5,0.3333333333333333) -- (1.3333333333333333,0.3333333333333333);
\draw[] (1.3333333333333333,0.3333333333333333) -- (0.6666666666666667,0.16666666666666666);
\draw[] (0.6666666666666667,0.16666666666666666) -- (0.5,0.16666666666666666);
\draw[] (0.5,0.16666666666666666) -- (0.5,0);
\draw (1,0.25) node[scale=1.2000000000000002,inner sep=0pt,shape=circle,minimum size=3pt,fill=black] {};
\draw (1,0.25) node[scale=1.2000000000000002,inner sep=0pt,above=0.4pt,left=6.5pt] {$t$};
\draw (1,0.25) node[scale=1.2000000000000002,inner sep=0pt,above=5.5pt] {$t$};
\draw (1,0.25) node[scale=1.2000000000000002,inner sep=0pt,above=1.4pt,right=6pt] {$t^{\unaryminus 1}$};
\draw (1,0.25) node[scale=1.2000000000000002,inner sep=0pt,below=5pt] {$t^{\unaryminus 1}$};
\draw[] (1.5,0) -- (1.5,0.16666666666666666);
}


\begin{scope}[shift={(0,0)}]
    \DrN
\end{scope}

\end{tikzpicture}

%% file: images/D_5_4.tikz
\begin{tikzpicture}[xscale=1.5,yscale=6]
\providecommand{\unaryminus}{\scalebox{0.5}[0.75]{\( - \)}}
\fontsize{10}{12}\selectfont
\newcommand*{\DrN}{%
\draw[thick] (0,1) -- (4,1);
\draw[thick] (4,0) -- (0,0);
\draw[dashed, thick] (0,0) -- (0,1);
\draw[dashed, thick] (4,1) -- (4,0);
\draw (0,0) node[scale=1.2000000000000002,inner sep=0pt,shape=isosceles triangle,minimum size=3pt,fill=black] {};
\draw (0,1) node[scale=1.2000000000000002,inner sep=0pt,shape=isosceles triangle,minimum size=3pt,fill=black] {};
\draw (1,0) node[scale=1.2000000000000002,inner sep=0pt,shape=isosceles triangle,minimum size=3pt,fill=black] {};
\draw (1,1) node[scale=1.2000000000000002,inner sep=0pt,shape=isosceles triangle,minimum size=3pt,fill=black] {};
\draw (2,0) node[scale=1.2000000000000002,inner sep=0pt,shape=isosceles triangle,minimum size=3pt,fill=black] {};
\draw (2,1) node[scale=1.2000000000000002,inner sep=0pt,shape=isosceles triangle,minimum size=3pt,fill=black] {};
\draw (3,0) node[scale=1.2000000000000002,inner sep=0pt,shape=isosceles triangle,minimum size=3pt,fill=black] {};
\draw (3,1) node[scale=1.2000000000000002,inner sep=0pt,shape=isosceles triangle,minimum size=3pt,fill=black] {};
\draw (4,0) node[scale=1.2000000000000002,inner sep=0pt,shape=isosceles triangle,minimum size=3pt,fill=black] {};
\draw (4,1) node[scale=1.2000000000000002,inner sep=0pt,shape=isosceles triangle,minimum size=3pt,fill=black] {};
\draw (0.5,0) node[scale=1.2000000000000002,inner sep=0pt,minimum size=14pt,below=0pt] {$e_{2}$};
\draw (0.5,1) node[scale=1.2000000000000002,inner sep=0pt,minimum size=14pt,above=0pt] {$e_{1}$};
\draw (1.5,0) node[scale=1.2000000000000002,inner sep=0pt,minimum size=14pt,below=0pt] {$e_{1}$};
\draw (1.5,1) node[scale=1.2000000000000002,inner sep=0pt,minimum size=14pt,above=0pt] {$e_{2}$};
\draw (2.5,0) node[scale=1.2000000000000002,inner sep=0pt,minimum size=14pt,below=0pt] {$e_{4}$};
\draw (2.5,1) node[scale=1.2000000000000002,inner sep=0pt,minimum size=14pt,above=0pt] {$e_{3}$};
\draw (3.5,0) node[scale=1.2000000000000002,inner sep=0pt,minimum size=14pt,below=0pt] {$e_{3}$};
\draw (3.5,1) node[scale=1.2000000000000002,inner sep=0pt,minimum size=14pt,above=0pt] {$e_{4}$};
\draw (1.5,0.3333333333333333) node[scale=1.2000000000000002,inner sep=0pt,shape=circle,minimum size=3pt,fill=black] {};
\draw (1.5,0.3333333333333333) node[scale=1.2000000000000002,inner sep=0pt,above=7pt,left=1pt] {$A^{-}$};
\draw (1.5,0.6666666666666666) node[scale=1.2000000000000002,inner sep=0pt,shape=circle,minimum size=3pt,fill=black] {};
\draw (1.5,0.6666666666666666) node[scale=1.2000000000000002,inner sep=0pt,above=7pt,left=1pt] {$A^{+}$};
\draw (1.5,0.5) node[scale=1.2000000000000002,inner sep=0pt,shape=circle,minimum size=3pt,fill=black] {};
\draw (1.5,0.5) node[scale=1.2000000000000002,inner sep=0pt,right=3.8pt] {$t$};
\draw (1.5,0.5) node[scale=1.2000000000000002,inner sep=0pt,left=8pt,above=-3.3pt] {$t^{\unaryminus 1}$};
\draw (0.5,0.6666666666666666) node[scale=1.2000000000000002,inner sep=0pt,shape=circle,minimum size=3pt,fill=black] {};
\draw (0.5,0.6666666666666666) node[scale=1.2000000000000002,inner sep=0pt,right=4pt] {$a$};
\draw (0.5,0.6666666666666666) node[scale=1.2000000000000002,inner sep=0pt,above=1.9pt,left=1.5pt] {$a^{\unaryminus 1}$};
\draw[] (0.5,1) -- (0.5,0.3333333333333333);
\draw[] (0.5,0.3333333333333333) -- (0.6666666666666667,0.3333333333333333);
\draw[] (0.6666666666666667,0.3333333333333333) -- (1.3333333333333333,0.16666666666666666);
\draw[] (1.3333333333333333,0.16666666666666666) -- (1.5,0.16666666666666666);
\draw[] (1.5,0.16666666666666666) -- (1.5,0);
\draw[] (1.5,1) -- (1.5,0.3333333333333333);
\draw[] (1.5,0.3333333333333333) -- (1.3333333333333333,0.3333333333333333);
\draw[] (1.3333333333333333,0.3333333333333333) -- (0.6666666666666667,0.16666666666666666);
\draw[] (0.6666666666666667,0.16666666666666666) -- (0.5,0.16666666666666666);
\draw[] (0.5,0.16666666666666666) -- (0.5,0);
\draw (1,0.25) node[scale=1.2000000000000002,inner sep=0pt,shape=circle,minimum size=3pt,fill=black] {};
\draw (1,0.25) node[scale=1.2000000000000002,inner sep=0pt,above=0.4pt,left=6.5pt] {$t$};
\draw (1,0.25) node[scale=1.2000000000000002,inner sep=0pt,above=5.5pt] {$t$};
\draw (1,0.25) node[scale=1.2000000000000002,inner sep=0pt,above=1.4pt,right=6pt] {$t^{\unaryminus 1}$};
\draw (1,0.25) node[scale=1.2000000000000002,inner sep=0pt,below=5pt] {$t^{\unaryminus 1}$};
\draw[] (2.5,1) -- (2.5,0.8333333333333334);
\draw[] (2.5,0.8333333333333334) -- (1.5,0.6666666666666666);
\draw[] (2.5,0) -- (2.5,0.16666666666666666);
\draw[] (2.5,0.16666666666666666) -- (1.5,0.3333333333333333);
\draw (2.5,0.9166666666666666) node[scale=1.2000000000000002,inner sep=0pt,shape=circle,minimum size=3pt,fill=black] {};
\draw (2.5,0.9166666666666666) node[scale=1.2000000000000002,inner sep=0pt,left=3.8pt] {$t$};
\draw (2.5,0.9166666666666666) node[scale=1.2000000000000002,inner sep=0pt,right=10pt,above=-3.3pt] {$t^{\unaryminus 1}$};
\draw[] (3.5,0) -- (3.5,0.16666666666666666);
\draw[] (3.5,0.16666666666666666) -- (1.5,0.3333333333333333);
\draw[] (3.5,1) -- (3.5,0.8333333333333334);
\draw[] (3.5,0.8333333333333334) -- (1.5,0.6666666666666666);
\draw (3.5,0.9166666666666666) node[scale=1.2000000000000002,inner sep=0pt,shape=circle,minimum size=3pt,fill=black] {};
\draw (3.5,0.9166666666666666) node[scale=1.2000000000000002,inner sep=0pt,left=3.8pt] {$t$};
\draw (3.5,0.9166666666666666) node[scale=1.2000000000000002,inner sep=0pt,right=10pt,above=-3.3pt] {$t^{\unaryminus 1}$};
}


\begin{scope}[shift={(0,0)}]
    \DrN
\end{scope}

\end{tikzpicture}

%% file: images/D_9_4.tikz
\begin{tikzpicture}[xscale=1.5,yscale=6]
\providecommand{\unaryminus}{\scalebox{0.5}[0.75]{\( - \)}}
\fontsize{10}{12}\selectfont
\newcommand*{\DrN}{%
\draw[thick] (0,1) -- (6,1);
\draw[thick] (6,0) -- (0,0);
\draw[dashed, thick] (0,0) -- (0,1);
\draw[dashed, thick] (6,1) -- (6,0);
\draw (0,0) node[scale=1.2000000000000002,inner sep=0pt,shape=isosceles triangle,minimum size=3pt,fill=black] {};
\draw (0,1) node[scale=1.2000000000000002,inner sep=0pt,shape=isosceles triangle,minimum size=3pt,fill=black] {};
\draw (1,0) node[scale=1.2000000000000002,inner sep=0pt,shape=isosceles triangle,minimum size=3pt,fill=black] {};
\draw (1,1) node[scale=1.2000000000000002,inner sep=0pt,shape=isosceles triangle,minimum size=3pt,fill=black] {};
\draw (2,0) node[scale=1.2000000000000002,inner sep=0pt,shape=isosceles triangle,minimum size=3pt,fill=black] {};
\draw (2,1) node[scale=1.2000000000000002,inner sep=0pt,shape=isosceles triangle,minimum size=3pt,fill=black] {};
\draw (3,0) node[scale=1.2000000000000002,inner sep=0pt,shape=isosceles triangle,minimum size=3pt,fill=black] {};
\draw (3,1) node[scale=1.2000000000000002,inner sep=0pt,shape=isosceles triangle,minimum size=3pt,fill=black] {};
\draw (4,0) node[scale=1.2000000000000002,inner sep=0pt,shape=isosceles triangle,minimum size=3pt,fill=black] {};
\draw (4,1) node[scale=1.2000000000000002,inner sep=0pt,shape=isosceles triangle,minimum size=3pt,fill=black] {};
\draw (5,0) node[scale=1.2000000000000002,inner sep=0pt,shape=isosceles triangle,minimum size=3pt,fill=black] {};
\draw (5,1) node[scale=1.2000000000000002,inner sep=0pt,shape=isosceles triangle,minimum size=3pt,fill=black] {};
\draw (6,0) node[scale=1.2000000000000002,inner sep=0pt,shape=isosceles triangle,minimum size=3pt,fill=black] {};
\draw (6,1) node[scale=1.2000000000000002,inner sep=0pt,shape=isosceles triangle,minimum size=3pt,fill=black] {};
\draw (0.5,0) node[scale=1.2000000000000002,inner sep=0pt,minimum size=14pt,below=0pt] {$e_{2}$};
\draw (0.5,1) node[scale=1.2000000000000002,inner sep=0pt,minimum size=14pt,above=0pt] {$e_{1}$};
\draw (1.5,0) node[scale=1.2000000000000002,inner sep=0pt,minimum size=14pt,below=0pt] {$e_{1}$};
\draw (1.5,1) node[scale=1.2000000000000002,inner sep=0pt,minimum size=14pt,above=0pt] {$e_{2}$};
\draw (2.5,0) node[scale=1.2000000000000002,inner sep=0pt,minimum size=14pt,below=0pt] {$e_{4}$};
\draw (2.5,1) node[scale=1.2000000000000002,inner sep=0pt,minimum size=14pt,above=0pt] {$e_{3}$};
\draw (3.5,0) node[scale=1.2000000000000002,inner sep=0pt,minimum size=14pt,below=0pt] {$e_{3}$};
\draw (3.5,1) node[scale=1.2000000000000002,inner sep=0pt,minimum size=14pt,above=0pt] {$e_{4}$};
\draw (4.5,0) node[scale=1.2000000000000002,inner sep=0pt,minimum size=14pt,below=0pt] {$e_{6}$};
\draw (4.5,1) node[scale=1.2000000000000002,inner sep=0pt,minimum size=14pt,above=0pt] {$e_{5}$};
\draw (5.5,0) node[scale=1.2000000000000002,inner sep=0pt,minimum size=14pt,below=0pt] {$e_{5}$};
\draw (5.5,1) node[scale=1.2000000000000002,inner sep=0pt,minimum size=14pt,above=0pt] {$e_{6}$};
\draw (1.5,0.3333333333333333) node[scale=1.2000000000000002,inner sep=0pt,shape=circle,minimum size=3pt,fill=black] {};
\draw (1.5,0.3333333333333333) node[scale=1.2000000000000002,inner sep=0pt,above=7pt,left=1pt] {$A^{-}$};
\draw (1.5,0.6666666666666666) node[scale=1.2000000000000002,inner sep=0pt,shape=circle,minimum size=3pt,fill=black] {};
\draw (1.5,0.6666666666666666) node[scale=1.2000000000000002,inner sep=0pt,above=7pt,left=1pt] {$A^{+}$};
\draw (1.5,0.5) node[scale=1.2000000000000002,inner sep=0pt,shape=circle,minimum size=3pt,fill=black] {};
\draw (1.5,0.5) node[scale=1.2000000000000002,inner sep=0pt,right=3.8pt] {$t$};
\draw (1.5,0.5) node[scale=1.2000000000000002,inner sep=0pt,left=8pt,above=-3.3pt] {$t^{\unaryminus 1}$};
\draw (0.5,0.6666666666666666) node[scale=1.2000000000000002,inner sep=0pt,shape=circle,minimum size=3pt,fill=black] {};
\draw (0.5,0.6666666666666666) node[scale=1.2000000000000002,inner sep=0pt,right=4pt] {$a$};
\draw (0.5,0.6666666666666666) node[scale=1.2000000000000002,inner sep=0pt,above=1.9pt,left=1.5pt] {$a^{\unaryminus 1}$};
\draw[] (0.5,1) -- (0.5,0.3333333333333333);
\draw[] (0.5,0.3333333333333333) -- (0.6666666666666667,0.3333333333333333);
\draw[] (0.6666666666666667,0.3333333333333333) -- (1.3333333333333333,0.16666666666666666);
\draw[] (1.3333333333333333,0.16666666666666666) -- (1.5,0.16666666666666666);
\draw[] (1.5,0.16666666666666666) -- (1.5,0);
\draw[] (1.5,1) -- (1.5,0.3333333333333333);
\draw[] (1.5,0.3333333333333333) -- (1.3333333333333333,0.3333333333333333);
\draw[] (1.3333333333333333,0.3333333333333333) -- (0.6666666666666667,0.16666666666666666);
\draw[] (0.6666666666666667,0.16666666666666666) -- (0.5,0.16666666666666666);
\draw[] (0.5,0.16666666666666666) -- (0.5,0);
\draw (1,0.25) node[scale=1.2000000000000002,inner sep=0pt,shape=circle,minimum size=3pt,fill=black] {};
\draw (1,0.25) node[scale=1.2000000000000002,inner sep=0pt,above=0.4pt,left=6.5pt] {$t$};
\draw (1,0.25) node[scale=1.2000000000000002,inner sep=0pt,above=5.5pt] {$t$};
\draw (1,0.25) node[scale=1.2000000000000002,inner sep=0pt,above=1.4pt,right=6pt] {$t^{\unaryminus 1}$};
\draw (1,0.25) node[scale=1.2000000000000002,inner sep=0pt,below=5pt] {$t^{\unaryminus 1}$};
\draw[] (2.5,1) -- (2.5,0.8333333333333334);
\draw[] (2.5,0.8333333333333334) -- (1.5,0.6666666666666666);
\draw[] (2.5,0) -- (2.5,0.16666666666666666);
\draw[] (2.5,0.16666666666666666) -- (1.5,0.3333333333333333);
\draw (2.5,0.9166666666666666) node[scale=1.2000000000000002,inner sep=0pt,shape=circle,minimum size=3pt,fill=black] {};
\draw (2.5,0.9166666666666666) node[scale=1.2000000000000002,inner sep=0pt,left=3.8pt] {$t$};
\draw (2.5,0.9166666666666666) node[scale=1.2000000000000002,inner sep=0pt,right=10pt,above=-3.3pt] {$t^{\unaryminus 1}$};
\draw[] (3.5,0) -- (3.5,0.16666666666666666);
\draw[] (3.5,1) -- (3.5,0.8333333333333334);
\draw[] (3.5,0.8333333333333334) -- (1.5,0.6666666666666666);
\draw (3.5,0.9166666666666666) node[scale=1.2000000000000002,inner sep=0pt,shape=circle,minimum size=3pt,fill=black] {};
\draw (3.5,0.9166666666666666) node[scale=1.2000000000000002,inner sep=0pt,left=3.8pt] {$t$};
\draw (3.5,0.9166666666666666) node[scale=1.2000000000000002,inner sep=0pt,right=10pt,above=-3.3pt] {$t^{\unaryminus 1}$};
\draw[] (1.5,0.3333333333333333) -- (3,0.3333333333333333);
\draw[] (3,0.3333333333333333) -- (4,0.5);
\draw[] (3.5,0.16666666666666666) -- (4,0.2222222222222222);
\draw (5,0.3333333333333333) node[scale=1.2000000000000002,inner sep=0pt,shape=circle,minimum size=3pt,fill=black] {};
\draw (5,0.3333333333333333) node[scale=1.2000000000000002,inner sep=0pt,below=9.895453501482391pt] {$A^{-}$};
\draw (5,0.6666666666666666) node[scale=1.2000000000000002,inner sep=0pt,shape=circle,minimum size=3pt,fill=black] {};
\draw (5,0.6666666666666666) node[scale=1.2000000000000002,inner sep=0pt,above=14.895453501482391pt,right=-4pt] {$A^{+}$};
\draw[] (5,0.3333333333333333) -- (5,0.6666666666666666);
\draw (5,0.5) node[scale=1.2000000000000002,inner sep=0pt,shape=circle,minimum size=3pt,fill=black] {};
\draw (5,0.5) node[scale=1.2000000000000002,inner sep=0pt,right=3.8pt] {$t$};
\draw (5,0.5) node[scale=1.2000000000000002,inner sep=0pt,left=8pt,above=-3.3pt] {$t^{\unaryminus 1}$};
\draw[] (5,0.3333333333333333) -- (4.5,0.16666666666666666);
\draw[] (4.5,0.16666666666666666) -- (4.5,0);
\draw[] (5,0.6666666666666666) -- (4.5,0.8333333333333334);
\draw[] (4.5,0.8333333333333334) -- (4.5,1);
\draw (4.5,0.9166666666666666) node[scale=1.2000000000000002,inner sep=0pt,shape=circle,minimum size=3pt,fill=black] {};
\draw (4.5,0.9166666666666666) node[scale=1.2000000000000002,inner sep=0pt,left=3.8pt] {$t$};
\draw (4.5,0.9166666666666666) node[scale=1.2000000000000002,inner sep=0pt,right=10pt,above=-3.3pt] {$t^{\unaryminus 1}$};
\draw[] (5,0.6666666666666666) -- (4,0.5);
\draw[] (5,0.3333333333333333) -- (4,0.2222222222222222);
\draw (4,0.5) node[scale=1.2000000000000002,inner sep=0pt,shape=circle,minimum size=3pt,fill=black] {};
\draw (4,0.5) node[scale=1.2000000000000002,inner sep=0pt,right=2.496150883013531pt,below=3.7442263245202967pt] {$t$};
\draw (4,0.5) node[scale=1.2000000000000002,inner sep=0pt,left=0.49615088301353083pt,above=3.7442263245202967pt,] {$t^{\unaryminus 1}$};
\draw[] (5,0.3333333333333333) -- (5.5,0.16666666666666666);
\draw[] (5.5,0.16666666666666666) -- (5.5,0);
\draw[] (5,0.6666666666666666) -- (5.5,0.8333333333333334);
\draw[] (5.5,0.8333333333333334) -- (5.5,1);
\draw (5.5,0.9166666666666666) node[scale=1.2000000000000002,inner sep=0pt,shape=circle,minimum size=3pt,fill=black] {};
\draw (5.5,0.9166666666666666) node[scale=1.2000000000000002,inner sep=0pt,left=3.8pt] {$t$};
\draw (5.5,0.9166666666666666) node[scale=1.2000000000000002,inner sep=0pt,right=10pt,above=-3.3pt] {$t^{\unaryminus 1}$};
}


\begin{scope}[shift={(0,0)}]
    \DrN
\end{scope}

\end{tikzpicture}

%% file: main.bib
@article{Ch18,
    author={Chen, L.},
    title={Spectral gap of scl in free products},
    year={2018},
    Journal={Proc. Amer. Math. Soc.},
    volume={146},
    number={7},
    pages={3143--3151},
    DOI={10/htck},
    arXiv={some ses}
}

@article{CCE91,
    author={Comerford, J. A. and Comerford Jr., L. P. and Edmunds, C. C.},
    title={Powers as products of commutators},
    year={1991},
    Journal={Comm. Algebra},
    volume={19},
    number={2},
    pages={675--684},
    publisher={Taylor & Francis},
    DOI={10/bb2qvb},
}

@article{Cull81,
    title = {Using surfaces to solve equations in free groups},
    journal = {Topology},
    volume = {20},
    number = {2},
    pages = {133--145},
    year = {1981},
    DOI = {10/cvcqkk},
    author = {Culler, M.}
}

@article{DH91,
    author={Duncan, A. J. and Howie, J.},
    title={The genus problem for one-relator products of locally indicable groups},
    year={1991},
    Journal={Math. Z.},
    volume={208},
    number={1},
    pages={225--237},
    DOI={10/b7cfpr}
}

@article{IK18,
    author={Ivanov, S. V. and Klyachko, A. A.},
    title={Quasiperiodic and mixed commutator factorizations in free products of groups},
    year={2018},
    Journal={Bull. London Math. Soc.},
    volume={50},
    number={5},
    pages={832--844},
    DOI={10/gfc6bq}
}

@article{Sch59,
    author={Sch\"utzenberger, M. P.},
    title={Sur l'equation $a^{2+n}=b^{2+m}c^{2+p}$ dans
un groupe libre},
    year={1959},
    Journal={C. R. Acad. Sci.  Paris S\'er. I Math.},
    volume={248},
    pages={2435--2436}
}

@article{BK22,
    author={Bereznyuk, V. Yu. and Klyachko, A. A.},
    title={Commutator length of powers in free products of groups},
    year={2022},
    journal={Proc. Edinburgh Math. Soc.},
    volume={65},
    number={1},
    pages={102--119},
    DOI={10/htcj},
    publisher={Cambridge University Press}
}

@article{How83,
    title={The solution of length three equations over groups},
    volume={26},
    DOI={10/d9jmjz},
    number={1},
    journal={Proc. Edinburgh Math. Soc.},
    publisher={Cambridge University Press},
    author={Howie, J.},
    year={1983},
    pages={89--96}
}

@article{Kl93,
author = {Klyachko, A. A.},
title = {A funny property of sphere and equations over groups},
journal = {Comm. Algebra},
volume = {21},
number = {7},
pages = {2555--2575},
year  = {1993},
publisher = {Taylor & Francis},
DOI = {10/b57tnm},
}

@article{Le09,
author = {Le Thi Giang},
DOI = {10/dv538q},
title = {The relative hyperbolicity of one-relator relative presentations},
journal = {J. Group Theory},
number = {6},
volume = {12},
year = {2009},
pages = {949--959},
}
